\documentclass[12pt]{amsart}

\usepackage{amsmath,amsfonts,amssymb,amsthm,tikz-cd,colonequals,enumitem,todonotes,soul,mathtools, bbm, url}
\usepackage[all]{xy}

\usepackage{listlbls}

\usepackage{color}

\newcounter{thm}
\numberwithin{thm}{section}
\newtheorem{theorem}[thm]{Theorem}
\newtheorem{proposition}[thm]{Proposition}
\newtheorem{lemma}[thm]{Lemma}
\newtheorem{corollary}[thm]{Corollary}
\theoremstyle{definition}
\newtheorem{definition}[thm]{Definition}
\newtheorem{remark}[thm]{Remark}
\newtheorem{example}[thm]{Example}
\newtheorem{conjecture}[thm]{Conjecture}
\newtheorem*{mquestion}{Motivating Question}

\newtheorem*{Proposition*}{Proposition}
\newtheorem*{Theorem*}{Theorem}
\newtheorem*{Claim*}{Claim}
\newtheorem*{Corollary*}{Corollary}
\newtheorem*{Conjecture*}{Conjecture}

\newcommand{\spar}
[1]{\medskip\refstepcounter{thm}\noindent\textbf{\arabic{section}.\arabic{thm}.}\label{#1}}

\newcounter{sspar}[thm]
\counterwithin{sspar}{thm}

\newcommand{\sspar}
[1]{\medskip\refstepcounter{sspar}\noindent\textbf{\arabic{section}.\arabic{thm}.\arabic{sspar}.}\label{#1}}

\title{A motivic Weil height machine for curves}
\author{L.\ Alexander Betts and Ishai Dan-Cohen}
\thanks{\textbf{Grant acknowledgements:} I.D. was supported by ISF grants 726/17 and 621/21. \textbf{Political disclaimer:} I. Dan-Cohen asks not to be considered responsible for the actions of any government that does not fully embrace the principles of democracy. }

\newcommand{\new}{\newcommand}
\new{\oo}{$\infty$}
\new{\opnm}{\operatorname}

\DeclareMathOperator{\Spec}{Spec}

\DeclareMathOperator{\Aug}{Aug}
\DeclareMathOperator{\cAug}{\mathcal{A}ug}
\DeclareMathOperator{\rk}{rk}
\DeclareMathOperator{\Pic}{Pic}

\newcommand{\lgeom}{{\mathrm{l.g.}}}
\newcommand{\bA}{\mathbb A}
\newcommand{\bG}{\mathbb G}
\newcommand{\bN}{\mathbb N}
\newcommand{\bQ}{\mathbb Q}
\newcommand{\bR}{\mathbb R}
\newcommand{\bV}{\mathbb V}
\newcommand{\bZ}{\mathbb Z}
\newcommand{\rH}{\mathrm H}
\newcommand{\cA}{\mathcal A}
\newcommand{\cC}{\mathcal C}
\newcommand{\cD}{\mathcal D}
\newcommand{\cL}{\mathcal L}
\newcommand{\cM}{\mathcal M}
\newcommand{\cO}{\mathcal O}
\newcommand{\cX}{\mathcal X}

\newcommand{\pr}{\mathrm{pr}}
\newcommand{\kbar}{{\bar k}}
\newcommand{\Obar}{\bar\cO}
\DeclareMathOperator{\GL}{GL}
\DeclareMathOperator{\rad}{rad}

\newcommand{\DM}{\mathbf{DM}}
\newcommand{\DA}{\mathbf{DA}}
\newcommand{\Sm}{\mathbf{Sm}}
\newcommand{\CAlg}{\opnm{\mathbf{CAlg}}}
\newcommand{\ho}{\opnm{ho}}
\newcommand{\op}{\mathrm{op}}

\newcommand{\et}{\mathrm{\acute{e}t}}

\newcommand{\tors}{\mathrm{tors}}

\DeclareMathOperator{\Hom}{Hom}
\DeclareMathOperator{\Iso}{Iso}
\DeclareMathOperator{\Aut}{Aut}
\DeclareMathOperator{\Sec}{Sec}
\DeclareMathOperator{\Homeo}{Homeo}
\DeclareMathOperator{\MCG}{MCG}

\DeclareMathOperator{\loc}{loc}

\DeclareMathOperator{\MHS}{\mathbf{MHS}}

\newcommand{\MRH}{\mathrm{MRH}}

\newcommand{\m}[1]{\mathrm{#1}}

\newcommand{\bb}[1]{\mathbb{#1}}
\newcommand{\cl}[1]{\mathcal{#1}}

\newcommand{\ka}{\kappa}
\newcommand{\si}{\sigma}
\newcommand{\Si}{\Sigma}
\newcommand{\ze}{\zeta}
\newcommand{\ga}{\gamma}
\newcommand{\al}{\alpha}
\newcommand{\be}{\beta}
\newcommand{\om}{\omega}
\newcommand{\Om}{\Omega}
\newcommand{\ep}{\epsilon}
\newcommand{\de}{\delta}
\new{\La}{\Lambda}

\newcommand{\Cc}{\mathcal{C}}
\newcommand{\bC}{\mathbb{C}}

\newcommand{\ZZ}{\bb{Z}}
\newcommand{\CC}{\bb{C}}
\newcommand{\GG}{\mathbb{G}}

\newcommand{\MM}{\bb{M}}

\newcommand{\NN}{\bb{N}}
\newcommand{\Ss}{\cl{S}}
\newcommand{\QQ}{\bb{Q}}
\newcommand{\RR}{\bb{R}}
\newcommand{\PP}{\bb{P}}

\newcommand{\Bb}{\mathcal{B}}

\newcommand{\EE}{\bb{E}}
\newcommand{\Gg}{\mathcal{G}}
\newcommand{\Ff}{\mathcal{F}}
\newcommand{\Hh}{\mathcal{H}}
\renewcommand{\AA}{\bb{A}}
\newcommand{\Yy}{\mathcal{Y}}
\newcommand{\Vv}{\mathcal{V}}

\newcommand{\Oo}{\mathcal{O}}

\newcommand{\Aa}{\mathcal{A}}
\newcommand{\Ee}{\mathcal{E}}

\newcommand{\Dd}{\mathcal{D}}
\newcommand{\Mm}{\mathcal{M}}
\newcommand{\Xx}{\mathcal{X}}

\newcommand{\from}{\leftarrow}
\newcommand{\xto}{\xrightarrow}
\newcommand{\xfrom}{\xleftarrow}

\newcommand{\inj}{\hookrightarrow}

\new{\inv}{^{-1}}
\new{\til}{\widetilde}

\new{\colim}{\opnm{colim}}
\new{\Alg}{\opnm{Alg}}
\new{\Sp}{\opnm{\cl{S}p}}
\new{\LMod}{\opnm{LMod}}
\new{\Sym}{\opnm{Sym}}
\new{\Cpx}{\opnm{Cpx}}
\new{\Bdry}{\opnm{Bdry}}
\new{\PSh}{\opnm{Psh}}
\new{\Ho}{\opnm{Ho}}
\new{\Sus}{\opnm{Sus}}
\new{\Mdga}{\opnm{Mdga}}
\new{\Aaug}{\opnm{\Aa ug}}
\new{\Sh}{\opnm{Sh}}
\new{\Vect}{\opnm{Vect}}
\new{\Ext}{\opnm{Ext}}
\new{\cdga}{\opnm{cdga}}
\new{\Isom}{\opnm{Isom}}
\new{\Mod}{\opnm{Mod}}
\new{\Loc}{\opnm{Loc}}

\new{\Gm}{\GG_m}

\new{\set}[2]{\big\{ #1 \; \big| \; #2 \big\} }

\new{\mar}{\marginpar}
\new{\comment}[1]
{***\mar{see foot}\footnote{Comment: #1}}

\new{\Fun}{\opnm{Fun}}

\new{\Finst}{\m{Fin}_*}

\new{\lgl}{\langle}
\new{\rgl}{\rangle}
\new{\angles}[1]{\lgl #1 \rgl}

\new{\Mul}{\opnm{Mul}}

\new{\qq}{\langle}
\new{\pp}{\rangle}

\new{\Lan}{\opnm{Lan}}

\new{\qandq}{\quad \text{and} \quad}

\new{\Fin}{\m{Fin}}

\new{\Nn}{\mathcal{N}}

\new{\Ll}{\mathcal{L}}

\new{\Rr}{\mathcal{R}}

\newcommand{\one}{\mathbbm{1}}

\new{\onef}{\one^f}

\new{\GrVect}{\opnm{GrVect}}
\new{\GrAlg}{\opnm{GrAlg}}

\new{\un}{\m{un}}

\new{\uHom}{\opnm{\underline{Hom}}}

\new{\Gro}{\opnm{Gro}}

\new{\dR}{\m{dR}}

\new{\Ind}{\opnm{Ind}}

\new{\MHM}{\opnm{MHM}}

\new{\gr}{\opnm{gr}}

\new{\run}{\m{run}}

\new{\ct}{\m{ct}}

\new{\ctr}{\m{ctr}}

\new{\rat}{\opnm{rat}}

\new{\VV}{\mathbb{V}}

\new{\pvs}{\m{pvs}}

\new{\mf}{\m{f}}
\new{\mc}{\m{c}}

\new{\cell}{\m{cell}}

\new{\End}{\opnm{End}}

\new{\an}{\mathrm{an}}

\new{\ca}{\mathrm{ca}}
\newcommand{\MMm}{\mathcal{MM}}

\begin{document}
\maketitle

\begin{abstract}
The rational points of a smooth curve $X$ over a number field $k$ map to the set of augmentations of the associated motivic algebra. An expectation, related to Kim's conjecture, is that for $X$ hyperbolic, the set of augmentations which come locally at each place of $k$ from a point is equal to the set of rational points. Our view is that this should provide a relative of the Grothendieck section conjecture which may be both more accessible, and more directly applicable, than the latter. 

As a first step in this direction, we extend aspects of the ``Weil height machine'' to the set of such augmentations, and use this to prove a Manin--Dem'janenko-style finiteness result for motivic augmentations for particular curves. Along the way, we determine the structure of the cohomological motive of a $\Gm$-bundle over an algebraic variety as a highly structured algebra in the derived \oo-category of mixed motives with rational coefficients.
\end{abstract}

\setcounter{tocdepth}{1}
\tableofcontents

\section{Introduction}

In the 1980's, Alexander Grothendieck proposed a programme to study rational points on curves over number fields by means of their profinite \'etale fundamental groups \cite[pp285--293]{SchnepsGeometricGaloisactions}. This fundamental group is a functor from the category of connected schemes and morphisms thereof to the category of profinite groups and outer (continuous) homomorphisms\footnote{An outer homomorphism $G\to H$ is a homomorphism modulo the natural action of~$H$ on $\Hom(G,H)$ by conjugation. The reason we consider outer homomorphisms here is because we are neglecting the choice of basepoints in fundamental groups for the sake of exposition.}, so whenever~$X$ is a connected variety over a field~$k$, the structure map $X\to\Spec(k)$ induces an outer homomorphism
\begin{equation}\label{eq:structure_map_pi_1}\tag{$\ast$}
	\pi_1^\et(X) \to \pi_1^\et(\Spec(k)) = G_k
\end{equation}
on profinite \'etale fundamental groups, where~$G_k$ is the absolute Galois group of~$k$. A $k$-rational point on~$X$, being a splitting of the structure map $X\to\Spec(k)$, induces by functoriality a splitting of~\eqref{eq:structure_map_pi_1} in the category of profinite groups and outer homomorphisms. So the functoriality of the \'etale fundamental group provides a natural map
\begin{equation}\label{eq:section_map}\tag{$\dagger$}
	X(k) \to \Sec(X/k) \coloneqq \{\text{outer splittings of~\eqref{eq:structure_map_pi_1}}\} \,,
\end{equation}
which one can use to study $k$-rational points. This is the subject of Grothendieck's famous Section Conjecture.

\begin{conjecture}
	If~$k$ is a number field and~$X/k$ is a smooth projective curve of genus~$\geq2$, then~\eqref{eq:section_map} is bijective.
\end{conjecture}

Although the Section Conjecture is thought to be far out of reach with current methods, it nonetheless yields plentiful predictions about the section set $\Sec(X/k)$ which one could try to verify unconditionally. After all, if sections of~\eqref{eq:section_map} are the same as rational points, then anything we can prove or define for rational points should also be provable or definable for sections (cf., \cite{BettsStix,BettsKumpitschLuedtke}). The starting point of this paper is a desire to develop a good theory of heights for sections.

\begin{mquestion}
	If~$X$ is a smooth projective curve of genus~$\geq2$ over a number field~$k$ and~$L$ is a line bundle on~$X$, can one define a \emph{Weil height function}
	\[
	h_L\colon \Sec(X/k) \to \bR \,,
	\]
	extending the Weil height \cite[\S2.3]{bombieri-gubler:heights} and satisfying similar properties? Given such a height function, can one use it to prove finiteness of $\Sec(X/k)$ unconditionally in some cases?
\end{mquestion}


As stated, this goal too seems far beyond the reach of current techniques. Instead, in this paper we will be interested in the analogous question when the profinite \'etale fundamental group is replaced by a different homotopical invariant of a more motivic nature. Fix a field~$k$, and let~$\cD=\DM(k,\bQ)$ denote the $\infty$-category of Voevodsky motives with coefficients in~$\bQ$. The functor\footnote{This is a functor in the sense of $\infty$-categories, when~$\Sm_k$ is viewed as an $\infty$-category via the Cech nerve construction.}
\[
C^*\colon \Sm_k^\op \to \cD
\]
sending a smooth $k$-scheme~$X$ to its cohomological motive~$C^*(X)$ is symmetric monoidal with respect to the usual product on~$\Sm_k$ and the tensor product on~$\cD$, so~$C^*(X)$ carries a canonical commutative algebra structure, induced from the diagonal map~$X\to X\times X$ and the unique map $X\to\Spec(k)$. Consequently, the cohomological motive functor lifts to a functor
\[
C^*\colon \Sm_k^\op \to \cC \coloneqq \CAlg(\cD) \,;
\]
when viewed as a commutative algebra, we refer to~$C^*(X)$ as the \emph{motivic cochain algebra} of~$X$.

For any smooth variety~$X/k$, the structure map $X\to\Spec(k)$ induces by functoriality a map $\bQ(0)\to C^*(X)$. Any $k$-rational point on~$X$ induces, again by functoriality, a splitting of this map, i.e., a map $C^*(X)\to\bQ(0)$ in~$\cC$.

\begin{definition}
	If~$X/k$ is a smooth variety, we define~$\cAug(X/k)$ to be the space of maps~$C^*(X)\to\bQ(0)$ in~$\cC$, and write
	\[
	\Aug(X/k) \coloneqq \pi_0(\cAug(X/k)) = \Hom_{\ho(\cC)}(C^*(X),\bQ(0))
	\]
	for its set of connected components.
\end{definition}

In this notation, the above construction defines a function
\begin{equation}\label{eq:augmentation_map}\tag{$\dagger\dagger$}
	X(k) \to \Aug(X/k) \,,
\end{equation}
which we call the \emph{augmentation map}. Our view is that the motivic augmentation set $\Aug(X/k)$ may be thought of as a rational\footnote{Here, the word \textit{rational} refers to the fact that passage from $X$ to $C^*(X)$ is somewhat analogous to the process of \textit{rationalisation} in homotopy theory.} motivic homotopy analog of the section set~$\Sec(X/k)$. This view is informed, in large part, by the close connection between $\Aug(X/k)$ and the so-called \emph{Selmer variety} of Chabauty-Kim theory; see \cite{dancohen2021rational}\footnote{Unfortunately, the published version of \cite{dancohen2021rational} contains numerous errors of grammar and idiom; the reader should consider referring to the arXiv preprint \cite{DCRat_arXiv} instead.} and \cite{MRMHT} for more on this point of view. By contrast with the section conjecture, it seems likely that when~$X$ is a smooth projective curve of genus~$\geq2$ over a number field, the map \eqref{eq:augmentation_map} is injective but far from surjective; that is, passage from smooth $k$-schemes to motivic cochain algebras leads to an overabundance of points.

%
%

To compensate for this overabundance, we define a notion of \textit{local geometricity}\footnote{The corresponding property for sections played a crucial role in \cite{BettsStix} and \cite{BettsKumpitschLuedtke}, where it was referred to as being \emph{Selmer}.} for motivic augmentations: if~$k$ is a global field, or an algebraic extension of such, then an augmentation in $\Aug(X/k)$ is locally geometric just when its image in $\Aug(X_{k_v}/k_v)$ is the augmentation attached to a $k_v$-rational point of~$X$ for all places~$v$ of~$k$. (See Definition~\ref{def:locally_geometric}.) Notably, compared with \cite{BettsStix}, we require this local geometricity condition even when~$v$ is an infinite place. The motivic augmentation attached to a $k$-rational point is automatically locally geometric, and so the augmentation map~\eqref{eq:augmentation_map} restricts to a map
\begin{equation}\label{eq:lg_augmentation_map}
	\tag{$\dagger\dagger \dagger$}
	X(k) \to \Aug(X/k)^\lgeom \,.
\end{equation}
Our belief is that it is this map which should be bijective.

\begin{conjecture}\label{conj:lg_augmentation_conjecture}
	\eqref{eq:lg_augmentation_map} is bijective whenever~$X$ is a smooth projective curve of genus~$\geq2$ over a finitely generated field~$k$.
\end{conjecture}

In this vein, one might hope to study $\Aug(X/k)^\lgeom$ using analogs of techniques used in the study of rational points. Specifically, one might hope to be able to define, in some reasonable level of generality, the \emph{height} of a point in~$\Aug(X/k)^\m{l.g.}$ with respect to a line bundle on~$X$. This is exactly what our main result achieves.

Let us now assume that~$k$ is a global field and that~$X$ is a smooth projective curve\footnote{For this paper, a curve is a one-dimensional geometrically integral separated $k$-scheme of finite type.} of genus~$\geq1$. The so-called \textit{Weil height machine} \cite[Theorem~2.3.8]{bombieri-gubler:heights} consists of a group homomorphism
\[
h\colon \Pic(X_\kbar) \to \frac{\{\text{functions $X(\kbar)\to\bR$}\}}{\{\text{bounded functions}\}} \,.
\]


\begin{Theorem*}[See \ref{thm:weil_heights} below.]
	There exists a canonical homomorphism
	\[
	h_\m{motivic}: \Pic(X_\kbar) \to \frac{\{\text{functions $\Aug(X_\kbar/\kbar)^\lgeom\to\bR$}\}}{\{\text{bounded functions}\}} \,,
	\]
	natural in $X$, lifting the Weil height machine~$h$ along the evident map
	\[
	\frac{\{\text{functions $\Aug(X_\kbar/\kbar)^\lgeom\to\bR$}\}}{\{\text{bounded functions}\}} \to \frac{\{\text{functions $X(\kbar)\to\bR$}\}}{\{\text{bounded functions}\}}
	\]
\end{Theorem*}

We refer to $h_\m{motivic}$ as the \emph{motivic Weil height machine}. Starting now, we drop the subscript `motivic'. For $L$ a line bundle on~$X_{\kbar}$, we let $h_L\colon\Aug(X/\kbar)^\lgeom\to\bR$ denote any representative of the image of $h$, and we refer to $h_L$ as a \emph{height function associated to $L$}.


%
%

The usefulness of our motivic Weil height machine is best judged through what it enables us to prove about the \emph{a priori} rather mysterious set~$\Aug(X/k)^\lgeom$. To this end, we prove an analog of the Manin--Dem'janenko Theorem \cite[Theorem~1]{Maninptorsion}, giving certain criteria under which we can \emph{almost} prove that $\Aug(X/k)^\lgeom$ is finite.

\begin{Theorem*}[See \ref{thm:manin-demjanenko} below]
	Let $X$ be a smooth projective curve over a number field $k$, and let~$J$ be the Jacobian of~$X$. Assume for simplicity that~$X$ has a $k$-rational point, and embed~$X$ inside~$J$ via the Abel--Jacobi map. Suppose there exists a simple abelian variety $A$ over $k$ whose multiplicity as an isogeny factor of~$J$ is strictly greater than $\rk A(k)$. Then the image of~$\Aug(X/k)^\lgeom$ inside~$\Aug(J/k)$ is finite.
\end{Theorem*}

Examples of curves~$X$ for which the assumption of the theorem holds include the modular curves~$X_0(p^n)$ and~$X_1(p^n)$ for $n\gg0$ \cite[Lemma~2]{Maninptorsion}.
\begin{remark}
	We will see shortly that the fibres of the augmentation map $J(k) \to \Aug(J/k)$ are finite (they are cosets of~$J(k)_\tors$). Thus, the map $\Aug(X/k)^\lgeom\to\Aug(J/k)$ having finite image implies that~$X(k)$ is finite, and in this way our motivic Manin--Dem'janenko Theorem generalises the classical version.
\end{remark}


The construction of our motivic Weil height machine is based on the observation that, while the motivic augmentation set~$\Aug(X/k)$ of a curve~$X$ is rather mysterious, curves possess maps to other varieties whose motivic augmentations sets are much more computable. One such computation is a direct consequence of work of Ancona et al. \cite{AnconaHuberI} and Iwanari \cite{Iwanari}:

\begin{Theorem*}[See \ref{10OctA} below]
	Let~$k$ be a field and~$A/k$ an abelian variety. Then we have
	\[
	\Aug(A/k) = A(k)\otimes\bQ \,.
	\]
\end{Theorem*}

A second such computation is as follows:

\begin{Theorem*}[See \ref{w24d} below]
	Let~$k$ be a field, $A/k$ an abelian variety, and~$M^\times$ a torsor under~$\Gm$ over~$A$. Then $\Aug(M^\times/k)$ is a torsor under~$k^\times\otimes\bQ$ over~$A(k)\otimes\bQ$.
\end{Theorem*}

This computation of the motivic augmentation set $\Aug(M^\times/k)$ relies, in turn, on a computation of the motivic cochain algebra $C^*(M^\times)$; many of the pages below are devoted to this computation. We'll return to this motivic aspect of our work later in the introduction.

Our strategy for defining the motivic Weil height associated to a line bundle $L$ on $X_\kbar$ is then as follows. After embedding $X_\kbar$ inside $J_\kbar$ via the Abel--Jacobi map, some power $L^{\otimes m}$ of the line bundle~$L$ is the restriction of a line bundle~$M$ on~$J_\kbar$. We are going to define the motivic Weil height~$h_L$ to be the composition
\begin{equation}\label{eq:height_composition}
\Aug(X_\kbar/\kbar)^\lgeom \to \Aug(J_\kbar/\kbar) = J(\kbar)\otimes\bQ \xrightarrow{\frac1m\hat h_M} \bR \,,
\end{equation}
where~$\hat h_M$ denotes the usual N\'eron--Tate height function associated to~$M$. (Because the N\'eron--Tate height~$\hat h_M$ is a quadratic function $J(\kbar)\to\bR$, it extends uniquely to a quadratic function $J(\kbar)\otimes\bQ\to\bR$.)

There is a significant difficulty which needs to be overcome here, namely that the line bundle~$M$ which gives rise to~$L$ need not be unique, and so it is not clear that the composition~\eqref{eq:height_composition} gives a well-defined function on~$\Aug(X_\kbar/\kbar)^\lgeom$, even up to bounded functions. Equivalently, one needs to show that if~$M$ is a line bundle on~$J_\kbar$ whose restriction to~$X_\kbar$ is trivial, then the composition~\eqref{eq:height_composition} is a bounded function.

We prove this by combining a geometric trick with our second motivic computation. If~$M^\times\to J_\kbar$ denotes the $\Gm$-torsor corresponding to the line bundle~$M$, then the fact that~$M$ trivialises over~$X_\kbar$ implies that the Abel--Jacobi embedding $X_\kbar\hookrightarrow J_\kbar$ lifts to an embedding $X_\kbar \hookrightarrow M^\times$. Then, using our description of~$\Aug(M^\times/\kbar)$, we show that the composition $\Aug(M^\times/\kbar)\to\Aug(J_\kbar/\kbar)\xrightarrow{\hat h_M}\bR$ decomposes as an integral of local height functions
\[
\lambda_{M,v}\colon \Aug(M^\times_{\kbar_v}/\kbar_v)\to \bR \,,
\]
indexed over the places~$v$ of~$\kbar$, including the infinite places. Using the usual properties of local heights, we find that each composition
\[
\Aug(X_\kbar/\kbar)^\lgeom \to \Aug(M^\times/\kbar) \to \Aug(M^\times_{\kbar_v}/\kbar_v) \xrightarrow{\lambda_{M,v}} \bR
\]
is bounded; integrating over all places~$v$ then shows that the composition~\eqref{eq:height_composition} is bounded and thus that our motivic Weil height is well-defined.


\begin{remark}
	It is more common in the literature to define the local components of the N\'eron--Tate height as \emph{N\'eron functions}, which are functions $\lambda_{D,v}\colon A(\kbar_v)\smallsetminus D(\kbar_v) \to \bR$ where~$D$ is a divisor on~$A$ \cite[Theorem~11.1.1]{lang:fundamentals_of_diophantine_geometry}. The alternative perspective we adopt, of defining local N\'eron--Tate heights as functions on the $\Gm$-torsor associated to~$M$, follows Bombieri--Gubler \cite[\S2.7 \&~\S9.5]{bombieri-gubler:heights} (see also \cite{Zhangadmissiblepairing}). The relationship between the two perspectives is that if~$D$ is a divisor representing the line bundle~$M$, and if~$s$ is a section of~$M$ with divisor~$D$, then $\lambda_{D,v} = \lambda_{M,v}\circ s$.
\end{remark}


With the Weil height machine set up, the proof of the motivic Manin--Dem'janenko Theorem then proceeds in two steps. By following the classical proof \cite{manin:bounded_torsion}, one shows that the motivic Weil height~$h_L$ associated to any ample line bundle~$L$ on~$X$ is bounded on the image of~$\Aug(X/k)^\lgeom \to \Aug(X_\kbar/\kbar)^\lgeom$. To turn this into the claimed finiteness result, we appeal to a Northcott-like property for motivic Weil heights.
\begin{Theorem*}[See \ref{thm:northcott_abelian} below]
	Suppose that~$X$ is a curve of genus~$\geq2$ over a number field~$k$. Let $\Sigma\subseteq\Aug(X_{\kbar}/\kbar)^\lgeom$ be a subset which is bounded in height and degree. Then for any morphism $f\colon X_{\kbar}\to A_{\kbar}$ from $X$ to an abelian variety~$A$ defined over~$\kbar$, the image
	\[
	f(\Sigma) \subseteq \Aug(A_{\kbar}/\kbar) = A(\kbar)/A(\kbar)_\tors
	\]
	is finite.
\end{Theorem*}
\begin{Theorem*}[See \ref{thm:northcott_finite_subscheme} below]
	Assumptions as above, there exists a finite subscheme~$Z\subset X$ defined over~$k$ with the following property: for every non-archimedean place~$v$ of~$\kbar$ and every point~$x_v\in X(\kbar_v)$ whose associated augmentation in $\Aug(X_{\kbar_v}/\kbar_v)$ is the image of an element of~$\Sigma$, we have
	\[
	x_v\in Z(\kbar_v) \,.
	\]
	In particular, the set of such points~$x_v$ is finite for all non-archimedean place~$v$.
\end{Theorem*}

\subsection{Motivic results}
As indicated above, our computation of the augmentation set of an abelian variety $A$ over our base field $k$ follows directly from work of Ancona \ et \ al.  \cite{AnconaHuberI,AnconaHuberII}, upgraded to the \oo-categorical level by Iwanari \cite{Iwanari}. Indeed, Iwanari shows that there is an isomorphism
\[
C^*(A) = \Sym^* M^1(A)
\]
in~$\cC$, where $M^1(A)\in\cD$ is a certain object considered by Ancona \ et \ al. which realises to $H^1(A)[-1]$. 

Our computation of the augmentation set of a $\Gm$-torsor over an abelian variety in Theorem \ref{w24d} follows from a more general result concerning the structure of the motivic cochain algebra of a $\Gm$-torsor~$M^\times$ over a smooth $k$-scheme~$Y$. Let
\[
c^M\colon \QQ(-1)[-2] \to C^*(Y)
\]
be the motivic first Chern class of $M$, and let 
\[
\QQ(-1)[-2] \to C^*(Y) \to \ep(M^\times)
\to \QQ(-1)[-1] 
\]
be the cofiber in $\Dd$. Then $C^*(M^\times)$ fits into a commuting square
\[
\xymatrix{
	\Sym^* C^*(Y)
	\ar[r] \ar[d]
	&
	\Sym^* \ep(M^\times)
	\ar[d]
	\\
	C^*(Y)
	\ar[r]
	&
	C^*(M^\times).
}
\]

\begin{Theorem*}[See \ref{30428Aa} below]
	The square above is a (homotopy) pushout. 
\end{Theorem*}

In other words, $C^*(M^\times)$ is obtained from $C^*(Y)$ by quotienting out the Chern class $c^M$ in a homotopy coherent manner. In the case~$Y=A$ is an abelian variety, this gives the claimed description of~$\Aug(M^\times/k)$.

The structure of $C^*(M^\times)$ as an object of $\Dd = \DM(k, \QQ)$ may be determined in a straightforward way. It is, however, its identity as a highly structured algebra that is needed for our purposes here. This contains a great deal more information than the former, and is, correspondingly, a good deal harder to obtain. Indeed, the \oo-category $\Cc = \CAlg \Dd$ is sandwiched between $\Dd$ on the one hand and the unstable motivic homotopy category $\Hh(k)$ on the other, and so our determination of $C^*(M^\times)$ as an object of $\Cc$ may be regarded as a step towards the determination of $M^\times$ as an object of $\Hh(k)$; this, however, appears to be unknown. 

In fact, there's another midpoint of sorts between the (relative!) simplicity of $\DM(k, \QQ)$ and the difficulty of $\Hh(k)$, namely, the motivic stable homotopy category $\Ss \Hh(k)$; it too may be seen as a compromise. Thus, $\Ss \Hh(k)$ and $\CAlg \Dd$ may be viewed as alternative (and largely orthogonal) compromises. In the quarter century since its introduction, the former has been studied intensely, while the latter has received little attention. Our success in determining $C^*(M^\times)$, and its close connection to the study of rational points, helps to bolster our view that this alternative  flavor of motivic homotopy deserves to be studied.

Our proof of Theorem \ref{30428Aa} in Section \ref{sec:AlgebraOfGmBundle} is based on the notion of the \emph{relatively free commutative algebra on a pointed object}. For instance, in comparison with the free commutative algebra $\Sym V$ on a given vector space $V$, the relatively free commutative algebra $\Sym'(V, v)$ insists that a pre-chosen vector $v \in V$ serve as unit. This is characterised by an evident mapping property, and may be constructed as a quotient of the ordinary free commutative algebra
\[
\Sym'(V,v) \simeq \Sym(V)/\langle v \rangle,
\]
or, alternatively, as a direct limit
\[
\Sym'(V,v) = \colim_{n} \Sym^n V
\]
where the transition maps are given by multiplication by $v$. It's the equivalence between these two (upgraded to our highly structured motivic setting) that gives us what we want. 

In rough outline, by a simple base-change from $\CAlg \Dd$ to the category
\[
\CAlg \Mod_{C^*(X)}(\Dd) \simeq (\CAlg \Dd)_{C^*(X)/}
\]
of $C^*(X)$-algebras, we find that the relative tensor product occurring on the right hand side of (*) may be identified (as a quotient of a free commutative algebra) with the relatively free algebra
\[
\Sym' \big(
C^*(X) \to C^*(L^\times)
\big)
\]
generated by the pointed object 
\[
C^*(X) \to C^*(L^\times)
\]
of $\Mod_{C^*(X)}(\Dd)$. Lurie's theory of free algebras shows that 
\[
\tag{**}
\Sym' \big(
C^*(X) \to C^*(L^\times)
\big) = \colim C^*(L^\times)^{c^L \otimes}
\]
is given by the colimit of a diagram
\[
C^*(L^*)^{c^L \otimes}: \m{Fin}_\m{inj} \to \Mod_{C^*(X)}(\Dd)
\]
(indexed by finite sets and injective maps) which mixes the symmetric group actions on the tensor powers of $C^*(L^\times)$ with multiplication by the 1st Chern class $c^L$.  

In the preliminary section \ref{sec:relA} we record a certain refinement of Quillen's theorem A (provided by Lior Yanovski and the members of the \textit{Homotopy Theory Seminar} at the Hebrew University of Jerusalem) which may be tailored to \oo-categories that are tensored over the \oo-category $\Hh_\QQ$ of rational homotopy types. This allows us to decompose the colimit (**) as 
\[
\tag{***}
\colim_n \Sym^n_{C^*(X)} C^*(L^\times)
\]
(Proposition \ref{21023F}). Taking symmetric powers over $C(X)$ in the exact triangle
\[
C^*(X)(-1)[-2] \xto{\text{mult. by } c^L}
C^*(X) \to C(L^\times) \to C^*(X)(-1)[-1],
\]
we get exact triangles
\[
\Sym^{n-1}_{C^*(X)} C(L^\times) \to 
\Sym^n_{C^*(X)} C^*(L^\times)
\to
\Sym^n_{C^*(X)} C^*(X)(-1)[-1]
\]
(Proposition \ref{30412A}). Since
\[
\Sym^n_{C^*(X)} C^*(X)(-1)[-1]
= 0 \quad \text{ for }  n \ge 2,
\]
the colimit (***) is given by
\[
\colim \big( C^*(X) \to C^*(L^\times) \xto{=} C^*(L^\times) \xto{=} C^*(L^\times) \xto{=} \cdots
\big)
\simeq C^*(L^\times), 
\]
as hoped. In fact, we discuss relatively free algebras and prove Propositions \ref{21023F} and \ref{30412A} in an abstract setting in the preliminary section \ref{sec:FreeEinftyAlgebras}. 

In view of our computation of the augmentation set of an abelian variety in Theorem \ref{10OctA} and of our computation of the motivic cochain algebra of a $\Gm$-torsor in Theorem \ref{30428Aa}, our result for the augmentation set of a $\Gm$-bundle over an abelian variety in Theorem \ref{w24d} essentially follows from the fact that an exact triangle in the stable homotopy category $\Ss \Hh$ gives rise to a ``pseudotorsoric comodule'' in the \oo-category $\Hh_*$ of pointed homotopy types (see Definition \ref{ts2}). While this must surely be well known, we explain this briefly in the preliminary section \ref{topscrap}, since we were unable to find a precise reference.

\subsection{Further discussion}

The question of giving a motivic interpretation of heights has been considered in several places in the literature. In the first author's thesis \cite{alex:motivic_anabelian}, it was shown that for~$k_v$ a non-archimedean local field, the local N\'eron--Tate height associated to a rigidified line bundle~$M$ on an abelian variety~$A$ factors through a function on the set
\[
\rH^1(Rep(G_{k_v}),\pi_1^{\bQ_l}(M^\times_{\kbar_v};\tilde0))
\]
of torsors under the $\bQ_l$-pro-unipotent \'etale fundamental group $\pi_1^{\bQ_l}(M^\times_{\kbar_v};\tilde0)$ in the category of representations of the Galois group $G_{k_v}$. And in work of Scholl \cite{scholl:height_pairings}, it was shown that for~$k_v$ a non-archimedean local field, the local height pairing between homologically trivial cycles with disjoint support factors through a function on the set
\[
\mathrm{MExt}_{Rep(G_{k_v})}(\bQ_l(0),\rH^{2a-1}(X_{\kbar_v},\bQ_l)(a),\bQ_l(1))
\]
of mixed extensions of $\bQ_l(0)$ by $\rH^{2a-1}(X_{\kbar_v},\bQ_l)(a)$ by $\bQ_l(1)$ (i.e., filtered Galois representations with graded pieces $\bQ_l(0)$, $\rH^{2a-1}(X_{\kbar_v},\bQ_l)(a)$ and $\bQ_l(1)$). Analogous descriptions hold in both cases when~$k_v$ is archimedean, using the category of mixed Hodge structures in place of local Galois representations.

The philosophy of motives predicts that $\pi_1^{\bQ_l}(M^\times_{\kbar_v};\tilde0)$ and $\rH^{2a-1}(X_{\kbar_v},\bQ_l)$ and their Hodge counterparts should be realisations of a motivic fundamental group $\pi_1^{mot}(M^\times;\tilde0)$ and a motive $h^{2a-1}(X)$ living in an abelian category~$\MMm_k$ of mixed motives. Accordingly, one can sum up these local heights across various realisations to obtain a function
\[
\rH^1(\MMm_k,\pi_1^{mot}(M^\times;\tilde0)) \to \bR
\]
in the first case, and a function
\[
\mathrm{MExt}_{\MMm_k}(\bQ(0),h^{2a-1}(X)(a),\bQ(1)) \to \bR
\]
in the second case. Moreover, if one assumes the expected properties of the category~$\MMm_k$ (chiefly that $\Ext^1_{\MMm_k}(\bQ(0),\bQ(1))=k^\times\otimes\bQ$), then these sums of local heights should factor through functions
\[
\rH^1(\MMm_k,\pi_1^{mot}(A)) \to \bR
\]
and
\[
\Ext^1_{\MMm_k}(\bQ(0),h^{2a-1}(X)(a))\times\Ext^1_{\MMm_k}(h^{2a-1}(X)(a),\bQ(1)) \to \bR \,,
\]
respectively. It is these factored maps which would constitute a motivic height function on torsors or a motivic height pairing on extension-classes.

This construction of a global height function or pairing relies on certain properties of the category~$\MMm_k$ which are currently unproven for all proposed constructions of~$\MMm_k$, chiefly the expected identity $\Ext^1_{\MMm_k}(\bQ(0),\bQ(1))=k^\times\otimes\bQ$. It is precisely to overcome these sorts of issues that we work instead in Voevodsky's derived category of motives, where we can construct a Weil height machine unconditionally.

We remark that the corresponding procedure for defining global $p$-adic height pairings was worked out unconditionally by Nekov\' a\v r \cite{nekovar:p-adic_height_pairings}. We also remark that Kato has developed a theory of heights of pure motives, however the kinds of heights which Kato generalises (Faltings heights of abelian varieties) are of rather a different kind to the heights we are considering.

\begin{remark}
	Using the results in \cite{alex:motivic_anabelian}, it follows that all of our main theorems hold also for a weaker version of local geometricity. If~$k_v$ is a non-archimedean local field of mixed characteristic~$(0,p)$ and $X/k_v$ is a smooth connected variety with a basepoint $x\in X(k_v)$, then there is, as explained in \cite{dancohen2021rational}, an \'etale realisation map
	\[
	R^\et_v\colon \Aug(X/k_v) \to \rH^1(Rep(G_{k_v}),\pi_1^{\bQ_l}(X_{\kbar_v};x)) \,,
	\]
	whose composition with the augmentation map $X(k_v) \to \Aug(X/k_v)$ is the pro-unipotent Kummer map $j_v\colon X(k_v) \to \rH^1(Rep(G_{k_v}),\pi_1^{\bQ_l}(X_{\kbar_v};x))$. Similarly, if~$k_v$ is archimedean, there is a corresponding Hodge realisation map
	\[
	R^\rH_v\colon \Aug(X/k_v) \to \rH^1(MHS_\bR,\pi_1^{\bR}(X_{\bC};x)) \,.
	\]
	
	If~$k$ is a number field and~$X/k$ is a smooth connected pointed variety, then we say that a motivic augmentation $\alpha\in\Aug(X/k)$ is \emph{realisationwise locally geometric} just when there is an adelic point $(x_v)_{v\in\Mm(k)}\in X(\bA_k)$ such that $R^\et_v(\alpha)=j_v(x_v)$ for all non-archimedean places~$v$, and $R^\rH_v(\alpha)=j_v^\rH(x_v)$ for all archimedean places~$v$. It is clear that every locally geometric augmentation is realisationwise locally geometric, so the set $\Aug(X/k)^{R{-}\lgeom}$ is a superset of~$\Aug(X/k)^\lgeom$. As we shall explain in section \ref{s:realisationwise_augs}, all of the main results of this paper also hold for realisationwise locally geometric augmentations. 
	
	We find this stronger form of our results particularly appealing because of its closer connection with Chabauty--Kim theory. However, this material comes with a slight caveat. These stronger results require that we rely on a good deal more background, including a foundational aspect that, to our knowledge, is not fully fleshed out in the literature. The proof of Theorem 6.5 of \cite{MRMHT} contains a gap. While the gap is largely filled by the Corrigendum \cite{DCSchlankCorr}, this still relies on a certain symmetric monoidal enhancement of the Koszul duality equivalence of Positselski \cite{Positselski} which, while perhaps regarded as folklore, does not appear to have been constructed explicitly. 
\end{remark}

Our results can also be thought of as describing a new obstruction to rational points on curves. If~$X$ is a smooth projective curve over a number field~$k$ and $(x_v)_{v\in\Mm(k)}\in X(\bA_k)$ is an adelic point, then let us say that~$(x_v)_v$ \emph{survives the motivic rational homotopy obstruction} just when there exists an augmentation~$\alpha\in\Aug(X/k)$ whose image in~$\Aug(X_{k_v}/k_v)$ is the augmentation attached to~$x_v$ for all places~$v$ of~$k$. (The augmentation~$\alpha$ is necessarily locally geometric.) Let us denote the set of all such points by
$X(\bA_k)^{\MRH}$. Although no direct implication in either direction is obvious to us, the following conjecture is similar in spirit to Conjecture~\ref{conj:lg_augmentation_conjecture}.

\begin{conjecture}
	For all smooth projective curves~$X/k$ of genus~$\geq2$, we have $X(\bA_k)^{\MRH}=X(k)$.
\end{conjecture}

This conjecture is also closely related to the implication 
\[
``X(Z) = X(\ZZ_p)_{\text{motivically global}}\text{''}
\]
of Kim's conjecture \cite{nabsd} observed in \cite{MRMHT}.

We remark that when the conditions of the Manin--Dem'janenko Theorem hold for~$X$, then the image of~$X(\bA_k)^{\MRH}$ inside $J(\bA_k)\otimes\bQ$ is contained in a finite subset of~$J(k)\otimes\bQ$. This implies, by the Manin--Mumford Conjecture (a theorem of Raynaud \cite{raynaud:torsion_on_curves}), that there is a finite subscheme $Z\subset X$ defined over~$k$ such that $X(\bA_k)^{\MRH}\subseteq Z(\bA_k)$. In particular, this means that the intersection of~$X(\bA_k)^{\MRH}$ with the finite descent locus is (ignoring technical issues at infinite places) equal to~$X(k)$ in this case \cite[Theorem~8.2]{stoll:finite_descent}.

Finally, this local-to-global perspective afforded by the motivic rational homotopy obstruction has a reformulation akin to Grothendieck's original anabelian program, in which $\pi_1^\et$ should form a fully faithful functor on the category of so-called \textit{anabelian varieties}. For a number field~$k$, let us consider the $\infty$-category $LtG$ whose objects are tuples $((Y_v)_{v\in\Mm(k)},C^*,(\phi_v)_{v\in\Mm(k)})$ where each~$Y_v$ is a smooth variety over~$k_v$, $C^*\in\CAlg \DM(k,\QQ)$ is a motivic algebra, and each $\phi_v$ is an isomorphism between $C^*_{k_v}$ and $C^*(Y_v)$ in $\CAlg \DM(k_v,\QQ)$, or in a suitable realisation. One might speculate that the evident functor from smooth $k$-varieties to $LtG$ is fully faithful on a subcategory $\Vv$ containing the anabelian varieties. Indeed, this expectation stands to Kim's conjecture roughly as the anabelian program stands to the Grothendieck section conjecture. In fact, since here we do not pass to $\pi_1$, $\Vv$ may well be larger than the category of anabelian varieties.

\subsection*{Structure and outline} Sections \ref{sec:relA},  \ref{topscrap}, and \ref{sec:FreeEinftyAlgebras} are  preliminary; we advise the reader to skip these sections on a first reading and refer back to them as needed. Then in \S\ref{sec:ReviewRatMot} we review the theory of rational motivic algebras as developed by Iwanari and by Schlank and the second author. This allows us to define motivic augmentation spaces in \S\ref{sec:AugmentationSpaces} and to compute them for abelian varieties. Building on the results of \S\ref{sec:AlgebraOfGmBundle}, \S\ref{sec:AugmentationsOfGmBundle} then contains the corresponding computations for $\Gm$-torsors over abelian varieties.

With this theory developed, we are able to set up our motivic Weil height machine in \S\ref{sec:HeightMachine} and prove its basic properties, especially the Northcott property. The proof of the motivic Manin--Dem'janenko Theorem then comes in \S\ref{sec:ManinDemjanenko}.

Finally, in sections \ref{s:realisations}--\ref{s:realisationwise_augs}, we explain how to extend our results from locally geometric augmentations to realisationwise locally geometric augmentations. Section~\ref{s:realisations} discusses realisation functors and how to use them to construct structured fundamental groupoids from objects of~$\Cc$; this section and section~\ref{s:VMHS} relate this construction to more usual Tannakian constructions. Finally, \S\ref{s:realisationwise_augs} explains the few adjustments to the strategy of~\S\ref{sec:HeightMachine} needed in order to define a height function on realisationwise locally geometric augmentations.

\subsection*{Notations and conventions.}
Our main reference for \oo-categories and higher algebra are Lurie's works \cite{LurieTopos, LurieAlg}, whose notation and terminology we follow with two exceptions: we use $\Hh$ for the \oo-category denoted $\Ss$ in loc. cit. and we refer to its objects as \textit{homotopy types}, and we use $\Ss\Hh$ for the stable \oo-category denoted $\Ss p$ in loc. cit. whose objects we refer to variously as \textit{spectra} or \textit{stable homotopy types}.

Our use of (\oo)-category theory is quite tame by modern standards. Thus, we elect to avoid fixing a universe and distinguishing between \textit{large} and \textit{small}. Our view is that correcting the resulting errors caused by this oversight would amount to a routine exercise in adding set theoretic decorations and caveats, while carrying out this exercise within the body of our text would somewhat cloud the exposition.

If $\Cc$ is an \oo-category, we write $h\Cc$ for its homotopy category. If, instead, $\Cc$ is a model category, we write $\ho \Cc$ for its homotopy category.

\subsection*{Acknowledgements}
We wish to thank Giuseppe Ancona, Aravind Asok, Elden Elmanto, Asaf Horev, Bruno Kahn, Katie Mann, Maxime Ramzi, Swann Tubach and Lior Yanovski for helpful conversations, and the members of the \textit{Homotopy Theory Seminar} at the Hebrew University of Jerusalem for providing us with the solution to a problem that arose in conversations with Yanovski (see section \ref{sec:relA}).

\section{Preliminaries: Relative Quillen's theorem A}
\label{sec:relA}

Quillen's ``Theorem A'' \cite[tag 02NX]{Kerodon} provides a criterion for cofinality of index-categories for diagrams in \oo-categories. If we only intend to use our index-categories in \oo-categories that are tensored over a given monoidal \oo-category $\Vv$, then the criterion may be weakened accordingly. The result will be applied in \S\ref{sec:FreeEinftyAlgebras} with $\Vv = \Hh_\QQ$ the \oo-category of rational homotopy types. 

The material of this section was provided by Lior Yanovski and the members of the \textit{Homotopy Theory Seminar} at the Hebrew University of Jerusalem. We restrict attention to presentable \oo-categories, although we expect that with some care, our assumptions may be weakened somewhat. 

\spar{30615A}
Let $\de: I \to J$ be a functor between \oo-categories. Given $j \in J$, we denote the pullback $I \times_J J_{j/}$ by $I_{j/}$.

Let $\Vv$ be a presentably monoidal \oo-category. 
We denote by $\Yy_\Vv$ the $\Vv$-linearised covariant Yoneda functor 
\[
\Yy_\Vv: J \to \Fun(J^\op, \Vv)
\]
\[
\Yy_\Vv(j)(k) = \Hom_J(k,j)\otimes \one_\Vv.
\]
We define the $\Vv$-valued presheaf
\[
W_\pi^\Vv: J^\op \to \Vv
\]
by taking the colimit of the composition
\[
I \xto{\pi} J \xto{\Yy_\Vv} \Fun(J^\op, \Vv).
\]

Let $F: J \to \Cc$ be a functor to an \oo-category presentably left-tensored over $\Vv$ and let $W$ be a functor $J^\op \to \Vv$. The left Kan extension 
\[
F_\Vv: \Fun(J^\op, \Vv) \to \Cc
\]
of $F$ along $\Yy_\Vv$ preserves colimits and is $\Vv$-linear. We define the \emph{weighted colimit of $F$ by $W$} by
\[
\colim^W F = F_\Vv(W).
\]

\begin{theorem}
\label{30615B}
Let $\pi: I \to J$ be a functor between \oo-categories, let $\Vv$ be a presentably monoidal \oo-category, and let $F: J \to \Cc$ be a functor to an \oo-category presentably left-tensored over $\Vv$. Assume the $\Vv$-linearised Yoneda functor $\Yy_\Vv$ is fully faithful. Then 
\[
\tag{*}
\colim(F \circ \pi) = \colim^{W_\pi^\Vv}F.
\]
\end{theorem}

\begin{proof}
We let 
\[
F_\Vv: \Fun(J^\op, \Vv) \to \Cc
\]
denote the left Kan extension of $F$ along the $\Vv$-linearised Yoneda functor $\Yy_\Vv$ as above. Our assumption that $\Yy_\Vv$ is fully faithful ensures that
\[
F_\Vv \circ \Yy_\Vv \simeq F.
\] 
Thus, staring from the right hand side of (*), we have
\begin{align*}
\colim^{W^\Vv_\pi}F
&= F_\Vv(W^\Vv_\pi)
\\
&=F_\Vv\colim(\Yy_\Vv \circ \pi)
\\
&= \colim(F_\Vv \circ \Yy_\Vv \circ \pi)
\\
&= \colim(F \circ \pi).\qedhere
\end{align*}

\end{proof}

\spar{30616A}
If $K$ is a simplicial set, we let $K \to K^\m{grp}$ denote fibrant replacement with respect to the Kan model structure (``groupoidification''). Let $\Vv$ be a presentably monoidal \oo-category. We define the \emph{$\Vv$-linearised groupoidification of $K$} by 
\[
K^\m{grp}\otimes \one_\Vv.
\]

If $\Cc$ is an \oo-category (with appropriate colimits), we may reasonably say that $\pi:I \to J$ is \emph{cofinal rel. $\Cc$} if for any functor $F:J \to \Cc$, the natural map
\[
\colim (F \circ \pi) \to \colim F
\]
is an equivalence in $\Cc$. As the next corollary shows, Theorem \ref{30615B} gives rise to a criterion for cofinality relative to any $\Cc$ which is presentably left tensored over $\Vv$ in terms of the $\Vv$-linearised groupoidifications of the overcategory pullbacks $I_{j/}$ ($j \in J$).

\begin{corollary}[Realtive Quillen's Theorem A]
\label{30615C}
In the situation and the notation of theorem \ref{30615B}, assume the unit object $\one_\Vv$ of $\Vv$ is terminal. Suppose for each $j \in J$ the natural map
\[
(I_{j/})^\m{grp} \otimes \one_\Vv \to \one_\Vv
\]
is an equivalence. Then there's an equivalence in $\Cc$
\[
\tag{*}
\colim(F \circ \pi) \simeq \colim F.
\]
\end{corollary}

\begin{proof}
By theorem \ref{30615B}, we have
\[
\colim (F \circ \pi) = \colim^{W^\Vv_\pi} F, 
\]
which, by definition, is given by
$
F_\Vv(W^\Vv_\pi).
$
We claim that the essentially unique map $W^\Vv_\pi \to \underline \one_\Vv$ to the terminal presheaf ($\underline \one_\Vv(x) = \one_\Vv$) is an equivalence. To see this, fix $j \in J$, let $\ep_j$ be the associated \textit{evaluation} functor
\[
\ep_j: \opnm{PSh}(J) \to \Hh
\] 
and let $\ep_j^\Vv$ be its $\Vv$-linearised cousin
\[
\ep_j^\Vv: \Fun(J^\op, \Vv) \to \Vv.
\]
Since both $\ep_j$ and $(\cdot \otimes \one_\Vv)$ preserve colimits, we have
\[
W^\Vv_\pi(j)
= \colim( \ep^\Vv_j \circ \Yy_\Vv \circ  \pi)
= \colim(\ep_j \circ \Yy \circ \pi) \otimes \one_\Vv.
\]
By Corollary 7.4.5.5 (tag 02VF) of \cite{Kerodon}, the colimit of $\ep_j \circ \Yy \circ \pi$ is given by the goupoidification (=cofibrant replacement with respect to the Kan model structure) of the total space of the associated left fibration. In turn, the associated left fibration is given by $I_{j/}$. Thus, our assumption (*) guarantees that $W^\Vv_\pi(j) \simeq \one_\Vv$ and hence that $W^\Vv_\pi \to \underline \one_\Vv$ is an equivalence as claimed. 

Hence, using again that $\underline \one_\Vv$ is terminal in $\Fun(J^\op, \Vv)$, we have equivalences in $\Cc$,
\[
F_\Vv(W^\Vv_\pi)
= F_\Vv(\underline \one_\Vv)
= \colim F. \qedhere
\]
\end{proof}

\begin{example}
\label{30616A}
In the situation and the notation of Theorem \ref{30615B}, let $\Vv = \Hh$ be the \oo-category of homotopy types with its Cartesian symmetric monoidal structure, and assume that for each $j \in J$, the \oo-category $I_{j/}$ is weakly contractible. Since the unit object of $\Hh$ is terminal and the (ordinary) Yoneda embedding of any \oo-category is fully faithful, Corollary \ref{30615C} applies to provide an equivalence in $\Cc$,
\[
\colim F = \colim(F \circ \pi).
\]
This is closely related to Quillen's ``Theorem A''; see, for instance, Section 7.2.3 (tag 02NX) of Kerodon \cite{Kerodon}.
\end{example}

\begin{example}
\label{30616B}
In the situation and the notation of Theorem \ref{30615B}, let $\Vv = \Hh_\QQ$ be the \oo-category of rational homotopy types with its Cartesian symmetric monoidal structure, and assume $J$ is a 1-category. Assume that for each $j\in J$, the homotopy type $(I_{j/})^\m{grp}$ is rationally contractible. 

The unit object of $\Hh_\QQ$ is terminal. Moreover, since the composition 
\[
\m{Set} \to \Hh \to \Hh_\QQ
\]
is fully faithful, so is the $\Hh_\QQ$-linearised Yoneda functor 
\[
\Yy_\QQ: J \to \Fun(J^\op, \Hh_\QQ).
\]
Hence, in this case too, Corollary \ref{30615C} applies to provide an equivalence in $\Cc$ (which is now required to be tensored over $\Hh_\QQ$),
\[
\colim F = \colim(F \circ \pi).
\]
\end{example}

\section{Preliminaries: The torsor associated to a triangle of spectra}
\label{topscrap}

This section collects a number of general results to be used in section \ref{sec:AugmentationsOfGmBundle}. The reader may wish to skip this section on first reading. 

\spar{ts1}
By an \emph{$E_\infty$-space} we mean an object of the $\infty$-category $\CAlg \Hh_*$ of commutative algebras in the $\infty$-category $\Hh_*$ of pointed homotopy types equipped with its Cartesian symmetric monoidal structure. Although the algebras occurring below do have such a structure, we will for the most part forget the commutativity structure and regard them as objects of the $\infty$-category $\Alg \Hh_*$ of algebras. An $E_\infty$-space $A$ is \emph{grouplike} if the induced monoid structure makes $\pi_0(A)$ into a group. We denote the full subcategory spanned by the grouplike $E_\infty$-spaces by $\CAlg^\m{gp} \Hh_*$.

If $E \in \Ss\Hh$ is a stable homotopy type, then $\Om^\infty E$ has the structure of a grouplike $E_\infty$-space. Better: there's a functor
\[
\Om^a\colon \Ss\Hh \to \CAlg^\m{gp} \Hh_*
\] 
lifting $\Om^\infty$ along the forgetful functor, whose restriction to the full subcategory of connective stable homotopy types $\Ss\Hh^{\ge 0}$ is an equivalence of categories.

\spar{}
We explain how the construction of $\Om^a$ follows from Remark 5.2.6.26 of \cite{LurieAlg}. Proposition 2.4.2.5 of loc. cit. describes an equivalence
\[
p\colon \CAlg(\Hh_*) \xto{\sim} 
\opnm{Mon}_\m{Comm}(\Hh_*)
\]
of $\infty$-categories between the $\infty$-category $\CAlg(\Hh_*)$ and the $\infty$-category of commutative monoids in $\Hh_*$ (see Definition 2.4.2.1 of loc. cit.). If $A$ is a commutative algebra in $\Hh_*$, then $p$ induces an isomorphism of monoids
\[
\pi_0(A) \xto{\sim} \pi_0(pA).
\]
Hence $p$ restricts to an equivalence of full subcategories
\[
p\colon \CAlg^\m{gp}(\Hh_*) \xto{\sim} 
\opnm{Mon}^\m{gp}_\m{Comm}(\Hh_*)
\]
of objects such that the induced monoid structure on $\pi_0$ is a group. There's an equivalence of $\infty$-operads $\m{Comm} \simeq \EE_\infty$ which induces an equivalence of $\infty$-categories
\[
\opnm{Mon}^\m{gp}_\m{Comm}(\Hh_*)
\simeq
\opnm{Mon}^\m{gp}_{\EE_\infty}(\Hh_*).
\]
We may thus equivalently construct a functor
\[
\tag{*}
\Ss\Hh \to \opnm{Mon}^\m{gp}_{\EE_\infty}(\Hh_*).
\]
\textit{Dunn additivity} furnishes an equivalence
\[
\opnm{Mon}_{\EE_k}(\Hh_\ast) \simeq
\opnm{Mon}_{\EE_k} \big( \opnm{Mon}_{\EE_0}(\Hh) \big)
\simeq
\opnm{Mon}_{\EE_k}(\Hh); 
\]
see Remark 5.2.6.24 of loc. cit. Thus, Remark 5.2.6.26 of loc. cit. (along with the constructions to which it alludes) furnishes diagrams of \oo-categories
\[
\xymatrix{
\Hh_* \ar[r]^-\Om \ar[d] & \Hh_* \ar[d]
\\
\opnm{Mon}_{\EE_{k+1}}(\Hh_*)
\ar[r]
&
\opnm{Mon}_{\EE_{k}}(\Hh_*)
}
\]
which commute up to homotopy. Taking limits, we obtain the desired functor (*).

\begin{definition}
\label{ts2}
Let $\Cc^{\prod}$ be an $\infty$-category endowed with a Cartesian monoidal structure closed under fibered products; in particular, the tensor unit $\one$ is terminal. Let $A$ be an algebra in $\Cc^{\prod}$ with underlying object $\underline A \in \Cc$. Let $P$ be a left $A$-module and let $X \in \Cc$ be an object. Endow $X$ with the structure of a left $A$-module via the essentially canonical augmentation $A \to \one$ and the essentially unique structure of $X$ as left $\one$-module; we refer to this structure as the \emph{trivial left $A$-module structure on $X$}. Let $P \to X$ be a morphism of left $A$-modules. We say that $P$ is \emph{pseudotorsoric over $X$} if the map
\[
\underline A \times P \to P \times_X P
\]
is an equivalence in $\Cc$.
\end{definition}

\begin{remark}
	The prefix ``pseudo'' here designates that we make no restriction on the map $P \to X$. For instance, if $A$ is a group and $X$ is a set, then a pseudotorsoric left $A$-module over $X$ is a left $G$-set $P$ equipped with a map of sets $f:P \to X$ such that for each $x \in X$, the fiber $f\inv(x)$ is either a $G$-torsor or empty.
\end{remark}

\spar{ts4}
\textbf{Pseudotorsoric left $A$-modules are stable under base-change.}
Pseudotorsoric left $A$-modules are stable under base-change in the following sense. In the situation and the notation of \S\ref{ts2} with $P$ a pseudotorsoric left $A$-module over $X$ in $\Cc^{\prod}$, let $f\colon X' \to X$ be a morphism in $\Cc$. Then after endowing $X'$ with the structure of trivial left $A$-module, $f$ may be upgraded in an essentially unique way to a morphism in the $\infty$-category $\LMod_A(\Cc)$ of left $A$-modules. Let $f'\colon P' \to X'$ denote the pullback of $P \to X$ along $f$ in $\LMod_A(\Cc)$. Since the forgetful functor 
\[
\LMod_A(\Cc) \to \Cc
\]
admits a left adjoint, after neglect of left $A$-module structures, $P' \to X'$ is also the pullback in $\Cc$. Moreover, the morphism
\[
\underline A \times P' \to P'\times_{X'} P'
\]
in $\Cc$ is an equivalence. Indeed, it may be identified with the pullback of the equivalence
\[
\underline A \times P \xto{\sim} P \times_X P
\]
along $f$.

\spar{foptlm}
\textbf{The nonempty fibers are torsors.}
The nonempty fibers of a pseudotorsoric left $A$-module are $A$-torsors in the following sense. In the situation and the notation of paragraph \ref{ts2}, suppose given a map $p\colon \one \to P$, denote by $x$ its composition with the map $g\colon P \to X$, and let $P_x$ denote the pullback of $P \to X$ along $x$. Then the induced orbit map $o(p)\colon  \underline A \to P_x$ is an equivalence. Indeed, by the Yoneda lemma, this reduces to the case $\Cc = \Hh$, which we treat via elementary topology in paragraph \ref{t23a}.

\sspar{t23a}
We continue with the situation and the notation of paragraph \ref{foptlm} with $\Cc^{\prod} = \Hh^{\prod}$ the $\infty$-category of homotopy types with Cartesian monoidal structure. When referring to the underlying object of $A$ we occasionally decorate with an underline for emphasis, but we do not insist on doing so consistently. We tacitly identify Kan complexes with their geometric realisations whenever this seems helpful and not dangerous. Our argument is based on the close relationship between the $\infty$-category $\Hh$ of homotopy types and the model category of topological spaces (and in particular on the relationship between limits in the former and the old constructions of homotopy limits in the latter) which we do not make fully explicit. 

By paragraph \ref{ts4}, we're reduced to the case that $X = *$ is terminal. By assumption, the morphism
\[
\tau\colon \underline A \times P \to P \times P
\]
is a homotopy equivalence. Let
\[
\iota\colon P \to P \times P
\]
be the map
\[
\iota(q) = (p,q)
\]
and let $Q$ be the pullback 
\[
\xymatrix{
Q \ar[r]^-\sim_\upsilon \ar[d]_\ka & P \ar[d]^\iota
\\
A \times P \ar[r]_\tau^\sim & P \times P
}
\]
in $\Hh$. We will show that the orbit map $o(p)\colon \underline A \to P$ factors through a homotopy equivalence $\tilde o(p)\colon \underline A \to Q$. 

Up to essentially canonical homotopy equivalence, $Q$ may be identified with a topological space whose points are in bijection with quintuples $(a,r, q, \al, \be)$ where $a \in A$, $r,q \in P$, and
\[
\al\colon r \to p, \quad \be\colon ar \to q
\]
are paths in $P$, so that $\upsilon$ and $\ka$ correspond to the projections $\ka(a,r,q,\al,\be) = (a,r)$, $\upsilon(a, r, q, \al, \be) = q$, and the homotopy filling in the square is determined by $\al \times \be$. We define 
\[
\tilde o(p)(a) \coloneqq (a, p, ap, id_p, id_{ap})
\]
(where `$id$' refers to a constant path). It's clear that $o(p)$ factors through $\tilde o(p)$. Meanwhile, the projection $\pi$ gives us a map in the opposite direction $A \from Q$ which is left-inverse to $\tilde o(p)$. We construct an explicit homotopy 
\[
h\colon Q \times I \to Q
\]
from $h_0 = id$ to $h_1 = \tilde o(p) \circ \pi$. We use the symbol `$\ast$' for concatenation of paths in the functorial order. Given $t \in I = [0,1]$, we let $\al_t$ be the path $\al(t) \to p$ in $P$ determined by $\al$, and we let $\ga_t$ denote the path
\[
a \al(t) \to (a\al \ast \be\inv)(t)
\]
in $P$ determined by the two paths
\[
ap \overset{a \al}{\from} ar \xto{\be} q.
\]
In terms of these notations, we define
\[
h(a,r,q, \al, \be, t) = 
\big(
a, \al(t), (a \al \ast \be\inv)(t), \al_t, \ga_t
\big).
\]
The topology on $Q$ is such that $h$ is continuous. This completes the verification and establishes the claim made in paragraph \ref{foptlm}.

\begin{proposition}[The torsor associated to a triangle]
\label{ts3}
Let
\[
E \to F \to G
\]
be an exact triangle in the stable $\infty$-category $\Ss\Hh$ of stable homotopy types. Then $\Om^\infty F$ admits a structure of pseudotorsoric left $\Om^a E$-module over $\Om^\infty G$ in the Cartesian-monoidal $\infty$-category $\Hh_*^{\prod}$ of pointed homotopy types.
\end{proposition}

\begin{proof}
Let $E \in \Ss\Hh$ be a stable homotopy type. Endow $*$ and $\Om^\infty \Si E$ with the trivial left $\Om^a E$-module structures. Then the essentially unique map $* \to \Om^\infty \Si E$ in $\LMod_{\Om^a E}\Hh_*$ makes $*$ into a pseudotorsoric left $\Om^a E$-module over $\Om^\infty \Si E$. Indeed, we have isomorphisms of pointed homotopy types
\[
* \times_{\Om^\infty \Si E} * = \Om \Om^\infty \Si E
= \Om^\infty E
=\underline{\Om^a E}
= \underline {\Om^a E} \times *.
\]

Let $E \to F \to G$ be an exact triangle in the $\infty$-category $\Ss\Hh$ of stable homotopy types as in the proposition. We then have an associated exact triangle
\[
F \to G \to \Si E
.
\]
Since $\Om^\infty$ preserves limits, we obtain a pullback square
\[
\xymatrix{
\Om^\infty F \ar[r] \ar[d] & \ast \ar[d] \\
\Om^\infty G \ar[r]  & \Om^\infty \Si E
}
\]
in the $\infty$-category $\Hh_*$ of pointed homotopy types. By the above discussion, $*$ equipped with the structure of trivial left $\Om^a E$-module is pseudotorsoric over $\Om^\infty \Si E$. By \S\ref{ts4}, the pullback in $\LMod_{\Om^a E} \Hh_*$ gives $\Om^\infty F$ the structure of pseudotorsoric $\Om^a E$-module over $\Om^\infty G$. This completes the proof of Proposition \ref{ts3}.
\end{proof}

\spar{t23bc}
\textbf{The connected components of a pseudotorsoric module.}
If $f\colon P\to X$ is a pseudotorsoric left $A$-module over $X$ in the Cartesian-monoidal $\infty$-category $\Hh_*^{\prod}$ of pointed homotopy types, then $\pi_0A$ acts transitively on the fibers of
\[
\pi_0f\colon\pi_0P \to \pi_0 X.
\]
If moreover, all fundamental groups of $X$ are trivial, then the action on each nonempty fiber is faithful. Indeed, by paragraph \ref{foptlm},
\[
A \xto{o(p)} P \to X
\]
(where $o(p)$ denotes the orbit-map of the base-point) forms a fiber sequence, the action of $\pi_0(A)$ on $\pi_0(P)$ induced by the module structure coincides with the one associated with the final terms in the long exact sequence of homotopy groups, and the assertions above are shown for instance in Bousfield--Kan \cite{BousfieldKan}. We nevertheless include a direct explanation using elementary topology in paragraphs \ref{el-start}--\ref{el-end}.

\sspar{el-start}
Encoded, for instance, in the fact that $\Hh_*$ is an $\infty$-category, is the fact that the equivalence
\[
\tau\colon A \times P \to P \times_X P
\]
admits a quasi-inverse given by a map $\upsilon$ and homotopies
\[
\eta\colon \tau \upsilon \xto{\simeq} id
\quad
\zeta\colon id \xto{\simeq} \upsilon \tau.
\]
As above, we tacitly identify objects of $\Hh$ with their geometric realisations. The fibered product $P \times_X P$ may be identified  (up to essentially canonical homotopy equivalence) with a topological space whose points are in bijection with triples $(p,q,\al)$ with $p,q$ points of $P$ and
\[
\al\colon f(p) \to f(q)
\] 
a path in $X$, and whose paths
\[
(p,q,\al) \to (p', q', \al')
\]
are in bijection with triples $(\be, \ga, \theta)$ with 
\[
\be\colon p \to p', 
\quad
\ga\colon q \to q'
\]
paths in $P$ and $\theta$ a homotopy in $X$ filling in the square
\[
\xymatrix{
fp \ar[r]^-{f\be} \ar[d]_-{\al} & fp' \ar[d]^-{\be}
\\
fq \ar[r]_-{f\ga} & fq'.
}
\]
The left-module structure provides a map of topological spaces 
\[
A \times P \to P
\]
\[
(a,p) \mapsto ap.
\]
The triviality of the module structure of $X$ then furnishes a path
\[
\ep_{a,p}\colon f(p) \to f(ap)
\]
in $X$, and $\tau$ may be identified with the map
\[
(a,p) \mapsto (p, ap, \ep_{a,p}). 
\]
Denoting an identity element of $A$ by $1$, the left-module structure of $A$ provides a path $\ze_p\colon p \to 1 \cdot p$ in $P$. The triviality of the module structure of $X$ then provides a path-homotopy
\[
E_p\colon \ep_{1, p} \to f(\ze_p).
\]
in $X$.

\sspar{}
\label{t23bd}
We check the first statement of paragraph \ref{t23bc} (transitivity).
Suppose the classes of $p, q \in P$ lie in the same fiber of $\pi_0f$. Then there's a path
\[
\ka\colon f(p) \to f(q)
\]
in $X$. The data $(p,q,\ka)$ determines a point of $P \times_X P$. Let
\[
(a,r) \coloneqq \upsilon(p,q, \ka).
\]
Then $\eta$ provides a path
\[
\tau(a,r) = (r, ar, \ep_{a,r}) \xto{(\be, \ga, \theta)} 
(p,q,\ka)
\]
in $P\times_X P$, from which we obtain paths
\[
ap \overset{a\be} {\from} ar \xto{\ga} q
\]
in $P$. This shows that $ap$ and $q$ lie in the same component of $P$, and completes the verification.

\sspar{}
\label{el-end}
We check the second statement of paragraph \ref{t23bc} (faithfulness). Suppose given $a \in A$, $p \in P$, and a path
\[
\phi\colon 1 \cdot p \to a \cdot p
\]
in $P$ witnessing the triviality of the action of $a$ on $p$. We consider the following square of paths in $X$:
\[
\xymatrix{
f(p) \ar@{=}[d] \ar[r]^-{\ep_{1,p}} & f(1 \cdot p) \ar[d]^{f(\phi)} 
\\
f(p) \ar[r]_-{\ep_{a,p}} & f(a \cdot p).
}
\]
Since the fundamental groups of $X$ all vanish, there exists a homotopy $\theta$ filling in the square. The data $(id_p, \phi, \theta)$ defines a path
\[
(p, 1 \cdot p, \ep_{1,p}) \to (p, a \cdot p, \ep_{a,p})
\]
in $P \times_X P$. Since $\tau$ is a homotopy equivalence, $(id_p, \phi, \theta)$ lifts to a path
\[
(1, p) \to (a, p)
\]
in $A\times P$. In particular, $a$ lies in the component of the identity of $A$. This completes the verification of the assertion made in paragraph \ref{t23bc}.

\begin{remark}
\label{w24bb}
Let $E \to F \to G$ be an exact triangle in $\Ss\Hh$ and suppose $\pi_{-1}E = 0$. Then the induced map
\[
\pi_0 \Om^\infty F \to \pi_0 \Om^\infty G
\]
is surjective. Indeed, associated to the exact triangle is a long exact sequence of stable homotopy groups
\[
\cdots \to \pi_0 F \to \pi_0 G \to \pi_{-1} E \to \cdots
\]
and for any stable homotopy type $H$ and $i \ge 0$, 
$
\pi_i \Om^\infty H = \pi_i H.
$
\end{remark}

\section{Preliminaries: The free $\EE_\infty$-algebra on an $\EE_0$-algebra}
\label{sec:FreeEinftyAlgebras}

If $\Cc$ is a suitable symmetric monoidal \oo-category, and $\one \to V$ is a pointed object, then the \emph{relatively free cummutative algebra $\Sym'(\one \to V)$ generated by $\one \to V$} (\ref{21023C}) is given, on the one hand, by the relative tensor product
\[
\Sym'(\one \to V) \simeq
\one \otimes_{\Sym \one} \Sym V
\]
(Remark \ref{21023D}), and, on the other hand, by a seemingly quite different colimit (Proposition \ref{21025C}). If $\Cc$ is suitably tensored over rational homotopy types, then the latter may be identified with a colimit of the form
\[
\Sym'(\one \to V) \simeq
\colim_i \Sym^i V
\]
(Proposition \ref{21023F}). If, moreover, $\Cc$ is stable and $\QQ$-linear, then the transition maps behind this last colimit fit into quite simple exact triangles (Proposition \ref{30412A}). 

As explained in the introduction, these results are used in our computation of the motivic cochain algebra of a $\Gm$-bundle in section \ref{sec:AlgebraOfGmBundle}. The reader may wish to skip this section on first reading.

\spar{21023A}
We recall notation and terminology associated with the foundations of the theory of $\infty$-operads as presented in \cite{LurieAlg}. Given $n \in \NN$, we let $\langle n \rangle$ denote the set $\{\ast,1, \dots, n\}$ and we let $\langle n \rangle^ \circ$ denote the subset $\{1, \dots, n\}$. Thus, for example, $\langle 0 \rangle = \{\ast\}$ and $\langle 0 \rangle ^\circ = \emptyset$. We let $\Finst$ denote the category whose objects are the sets $\langle n \rangle$ ($n \in \NN$) and whose morphisms are maps of sets which preserve the distinguished elements $\ast \in \angles{n}$. Recall that a morphism $f: \angles{n} \to \angles{m}$ in $\Finst$ is said to be \emph{inert} if for each $i \in \angles{m}^\circ$, $f\inv(i)$ has exactly one element, and \emph{active} if $f\inv\{*\} = \{*\}$. 

\spar{21023A}
It will be convenient for us to distinguish a third class of morphisms in $\Finst$: we say $f$ is \emph{semiinert} if for each $i \in \angles{m}$, $f\inv\{i\}$ has \textit{at most} one element. Recall from \cite[2.1.1.19]{LurieAlg} that the $\infty$-operad $\EE_0^\otimes$ is defined to be (the nerve of) the 1-category whose objects are the same as those of $\Finst$ and whose morphisms are the semiinert maps, equipped with its natural functor $\EE_0^\otimes \to \Finst$. If $\Cc^\otimes$ is a symmetric monoidal $\infty$-category with unit object $\one$, Remark 2.1.3.10 of \cite{LurieAlg} provides an equivalence of $\infty$-categories 
\[
\Alg_{\EE_0}(\Cc) \simeq \Cc_{\one/}.
\]

\spar{21023C}
We denote by $\EE_\infty^\otimes$ (the nerve of) the 1-category $\Finst$ in its role as an $\infty$-operad via the identity functor $\Finst \to \Finst$. We fix for the remainder of this section a symmetric monoidal $\infty$-category
\[
q: \Cc^\otimes \to \Finst
\]
in which $\otimes$ respects small colimits separately in each variable. By Corollary 3.1.3.5 of loc. cit., the forgetful functor
\[
\CAlg(\Cc) \simeq
\Alg_{\EE_\infty}(\Cc)
\to 
\Alg_{\EE_0}(\Cc)
\simeq
\Cc_{1/}
\]
admits a left adjoint; let us denote it by
\[
\tag{*}
\Sym': \CAlg(\Cc) \from \Cc_{1/}.
\]
We refer to $\Sym'$ applied to an object $\one \to V$ as the \emph{relatively free commutative algebra generated by $\one \to V$}. The same corollary applied with the trivial $\infty$-operad in place of $\EE_0^\otimes$ provides a left adjoint $\Sym$ to the forgetful functor
\[
\CAlg \Cc \to \Cc;
\]
this is the usual free symmetric algebra functor. 

\begin{remark}
	\label{21023D}
	Given $(\one\xto{x}V) \in \Cc_{\one/}$, the commutative algebra $\Sym'(\one\xto{x}V)$ fits into a pushout square in $\CAlg \Cc$:
	\[
	\xymatrix{
		\Sym \one
		\ar[d]_\de \ar[r]
		&
		\Sym V
		\ar[d]
		\\
		\one
		\ar[r]
		&
		\Sym'(x)
	}
	\]
	in which $\de$ is the counit of the adjunction. Indeed, the pushout and $\Sym' V$ share the same $\infty$-categorical mapping property. 
\end{remark}

\spar{21023E}
Fix $V \in \Cc$ and $n \in \NN$. We recall the precise definition of the symmetric power $\Sym^n V \in \Cc$ (a special case of Construction 3.1.3.9 of loc. cit.). Recall that $\m{Triv}^\otimes \subset \Finst$ is the subcategory with same objects and inert morphisms. Example 2.1.3.5 of loc. cit. provides a $\m{Triv}$-algebra
\[
\overline V: \m{Triv}^\otimes \to \Cc^\otimes
\]
which sends 
\[
\angles{1} \mapsto V.
\]
Let $\Si_n$ denote the symmetric group on $n$ letters. We have a canonical inclusion of 1-categories
\[
\m{Triv}^\otimes 
\times_{\Finst}
(\Finst)_{/\angles{1}}
\supset
B \Si_n
\]
which identifies $B\Si_n$ with the full subcategory spanned by the active morphism $\angles{n} \to \angles{1}$. (The fibered product on the left refers to fibered product in the 1-category of simplicial sets applied to the nerves of the surrounding 1-categories.) Let $h_0$ be the composite functor
\[
B\Si_n \subset
\m{Triv}^\otimes 
\times_{\Finst}
(\Finst)_{/\angles{1}}
\to 
\Finst.
\]
In concrete terms, $h_0$ sends the unique object of $B\Si_n$ to $\angles{n}$, and sends an automorphism of the set $\{1, \dots, n\}$ to the associated automorphism of the pointed set $\angles{n}$. Let $h_1$ denote the constant functor
\[
B\Si_n \to \Finst
\]
with value $\angles{1}$. Let $h$ be the natural transformation $h_0 \to h_1$ which sends the unique object of $B\Si_n$ to the active morphism $\angles{n}\to \angles{1}$. The composite $V_{\oplus n} =$
\[
B\Si_n \subset
\m{Triv}^\otimes 
\times_{\Finst}
(\Finst)_{/\angles{1}}
\to \m{Triv}^\otimes
\xto{\overline V}
\Cc^\otimes
\]
provides a lifting of $h_0$ along $q: \Cc^\otimes \to \Finst$. Since $q$ is a coCartesian fibration, the natural transformation $h$ lifts to a $q$-coCartesian natural transformation
\[
\overline h: V^\oplus_n \to V^\otimes_n
\]
for some functor 
\[
V^\otimes_n : B\Si_n \to \Cc = 
\Cc^\otimes_{\angles{1}}.
\]
The colimit
\[
\Sym^n V := \colim V^\otimes_n 
\]
exists by Proposition 3.1.3.13 of loc.\ cit.

\begin{proposition}
	\label{21023F}
	Let $\Cc$ be a presentably symmetric monoidal $\infty$-category. Fix $(\one \xto{x} V) \in \Cc_{\one/}$. Assume $\Cc$ is presentably tensored over the symmetric monoidal \oo-category $\Hh_\QQ$ of rational homotopy types. Then the underlying object of $\Sym'(\one \xto{x} V)$ in $\Cc$ is given by the colimit of the diagram
	\[
	\tag{*}
	\one \xto{f_0} V
	\xto{f_1} \Sym^2 V
	\xto{f_2} \cdots
	\]
	in which $f_i$ is given by the composition
	\[
	\Sym^i V \simeq (\Sym^i V)
	\otimes \one
	\xto{id \otimes x}
	(\Sym^i V) \otimes V \to \Sym^{i+1} V.
	\]
	Moreover, the unit of the adjunction is the evident map.
\end{proposition}

We note that the Dold-Thom theorem suggests that a similar result may hold in the category $\Hh$ of homotopy types (despite not being tensored over $\Hh_\QQ$). However, we are unaware of a reference for this fact. The proof spans paragraphs \ref{21023G}--\ref{21023Ha}. 

\spar{21023G}
We continue with our fixed symmetric monoidal $\infty$-category in which $\otimes$ respects colimits separately in each variable (while dispensing for now with the additional assumption imposed in proposition \ref{21023F}). Fix $(\one \xto{x} V) \in \Cc_{\one/}$, and let $\widetilde \Fin_{inj}$ denote the category of finite sets and injections. The colimit of \ref{21023F}(*) is a colimit of colimits. It will be easier for us to work with the colimit of a single diagram 
\[
V^{x \otimes}:
\widetilde \Fin_{inj} \to \Cc,
\quad
S \mapsto V^{\otimes S},
\]
which will later be shown to be equivalent to the iterated colimit above. Its construction (which occupies paragraphs \ref{21023I}--\ref{21025B}) is a minor variation on paragraph \ref{21023E}. Actually, it will be convenient for our purposes to restrict attention to the equivalent full subcategory $ \Fin_{inj} \subset \widetilde \Fin_{inj}$ spanned by the sets $\angles{n}^\circ = \{1, \dots, n\}$ ($n\in \NN$). 

\spar{21023I}
A morphism $\angles{m} \to \angles{n}$ in $\Finst$ is active if and only if it commutes with the active maps to $\angles{1}$. A morphism $\angles{m} \to \angles{n}$ in $\Finst$ is both active and semiinert if and only if it restricts to an injective map $\angles{m}^\circ \to \angles{n}^\circ$. It follows that the full subcategory of 
$\EE^\otimes_0 \times_{\Finst} (\Finst)_{/\angles{1}}$ spanned by the active maps $\angles{n}\to \angles{1}$ ($n \in \NN$) may be identified with the category $\Fin_{inj}$.

\spar{21025A}
Remark 2.1.3.10 of loc. cit. provides us with an $\EE_0^\otimes$-algebra
\[
\overline x: \EE_0^\otimes \to \Cc^\otimes
\]
associated to $x: \one \to V$. One may retrieve $x$ from $\overline x$ as follows. In $\EE_0^\otimes$, as in $\Finst$, there's a unique map $x_0: \angles{0} \to \angles{1}$. The choice of unit object $\one \in \Cc$ is determined by the choice of an object of 
$\Cc^\otimes_{\angles{0}}$; we may choose $x_0\angles{0}$. Then the morphism $\overline x (x_0)$ in $\Cc^\otimes$ factors uniquely through a morphism
\[
\overline x (x_0)': \one \to \overline x\angles{1}
\]
and we have $V = \overline x\angles{1}$ and $x = \overline x (x_0)'$.

\spar{21025B}
We have natural maps forming a commuting diagram in the 1-category of simplicial sets
\[
\xymatrix{
	& 
	\EE_0^\otimes
	\ar[r]^{\overline x}
	&
	\Cc^\otimes
	\ar[d]^q
	\\
	\Fin_{inj}
	\ar[r]_-{i}
	&
	\EE_0^\otimes \times_{\Finst}(\Finst)_{/\angles{1}}
	\ar[u] \ar[r]_-{g} \ar[ur]^-{f}
	&
	\Finst
	\\
	&
	\big(\angles{n} \to \angles{1}\big)
	\ar@{|->}[r]
	&
	\angles{n}.
}
\]
Let $V^{x \oplus} = f \circ i$, let $h_0 = g \circ i$, let $h_1$ be the constant functor $\Fin_{inj} \to \Finst$ with value $h_1\angles{n} = \angles{1}$, and let $h$ be the natural transformation $h_0 \to h_1$ which sends $\angles{n} \in \Fin_{inj}$ to the unique active map $\angles{n} \to \angles{1}$ in $\Finst$. Since $q$ is a coCartesian fibration, the natural transformation $h$ has an essentially unique lift $\overline h$ along $q$ to a natural transformation with source $V^{x \oplus}$; we define $V^{x \otimes}$ to be the target of $\overline h$.

\begin{proposition}
	\label{21025C}
	Let $\Cc$ be a symmetric monoidal category
	in which $\otimes$ respects colimits separately in each variable, let 
	\[
	x^\otimes: \EE^\otimes_0 \to \Cc^\otimes
	\]
	be an $\EE_0^\otimes$-algebra in $\Cc$ with underlying pointed object 
	$(\one \xto{x} V) \in \Cc_{/\one}$, and let
	\[
	A^\otimes = \Sym'(\one \xto{x} V):
	\EE_\infty^\otimes \to \Cc
	\]
	be the relatively free commutative algebra in $\Cc$ generated by $x$ (\ref{21023C}). The underlying object of $\Sym'(\one \xto{x} V)$ in $\Cc$ is given by
	\[
	\Sym'(\one \xto{x} V) = \colim V^{x \otimes}.
	\]
\end{proposition}

\begin{proof}
	Let $\theta$ denote neglect of structure along the inclusion functor $\EE_0^\otimes \to \EE_\infty^\otimes$ and let
	\[
	f:x^\otimes \to\theta(A^\otimes)
	\]
	be the counit of the adjunction. Let $h_0^\rhd$ be the functor $(\Fin_{inj})^\rhd \to \Finst$ associated to the composition
	\[
	\Fin_{inj} \xto{i} 
	\EE^\otimes_0 \times_{\Finst}(\Finst)_{/\angles 1}
	\to
	(\Finst)_{/\angles 1}. 
	\]
	Thus, $h_0^\rhd$ extends $h_0$ by sending the cone  point $c \in (\Finst)^\rhd$ to $\angles{1}$ and by sending a morphism $\angles n^\circ \to c$ to the active map $\angles n \to \angles 1$. The construction of \cite[Definition 3.1.3.1]{LurieAlg} associates to $f$ a diagram
	\[ 
	V^{x\oplus}_\rhd:
	(\Fin_{inj})^\rhd \simeq
	\big(
	\EE^\otimes_0 \times_{\Finst} (\Fin_*^{act})_{/\angles{1}}
	\big)^\rhd
	\to
	\Cc^\otimes
	\]
	such that
	\[
	V^{x\oplus}_\rhd|_{\Fin_{inj}} = V^{x\oplus},
	\quad
	V^{x\oplus}_\rhd(c)=A,
	\quad \mbox{and} \quad
	q \circ V^{x\oplus}_\rhd = h_0^\rhd.
	\]
	According to \cite[Corollary 3.1.3.5]{LurieAlg}, $V^{x\oplus}_\rhd$ is a colimit diagram. 
	
	Let $h_1^\rhd$ be the constant functor
	\[
	(\Fin_{inj})^\rhd \to \Finst
	\]
	with value $\angles1$, and let $h^\rhd$ be the natural transformation $h_0^\rhd \to h_1^\rhd$ which associates to each object $x$ of $(\Finst)^\rhd$ the unique active map $h_0^\rhd(x) \to \angles 1$. Since $q: \Cc^\otimes \to \Finst$ is a coCartesian fibration, the natural transformation $h^\rhd$ has an essentially unique lift $\overline h^\rhd$ along $q$ to a natural transformation with source $V^{x\oplus}_\rhd$; we define
	\[
	V^{x\otimes}_\rhd:
	(\Fin_{inj})^\rhd
	\to
	\Cc
	\]
	to be the target of $\overline h^\rhd$. Combining Propositions 3.1.1.15 and 3.1.1.16 of \cite{LurieAlg}, we find that $V^{x\otimes}_\rhd$ is a colimit diagram, which completes the proof of the proposition.
\end{proof}

To complete the proof of proposition \ref{21023F}, we'll need a few preliminaries. We begin with a lemma. 

\begin{lemma}
	\label{40510F}
	Let $\Dd$ be a stable \oo-category and let $A \in \Dd$ be an object. Then $\Dd_{A/}$ is generated by split injections (i.e. objects of the form $i: A \to A \oplus C$ with $i$ the coprojection) under sifted colimits. 
\end{lemma}

The proof proceeds in several steps (paragraphs \ref{40510A}--\ref{40510D}).

\spar{40510A}
An arbitrary object $A \to B$ of the arrow category
\[
\vec \Dd := \Fun(\Delta^1, \Dd)
\]
may be written as the cofiber of a morphism
\[
(0 \to A) \to (A \to A \oplus B). 
\]
Thus, if
\[
\Si = \Delta^1 \coprod_{\{0\}} \Delta^1
\]
is the span-shaped category, then $A \to B$ is a colimit in $\vec{\Dd}$ of a $\Si$-indexed diagram of split injections.  

\spar{40510B}
Any $\Si$-indexed diagram $A_\bullet \to B_\bullet$  in $\vec{\Dd}$ which admits a colimit
\[
\tag{*}
\colim_\Si (A_\bullet \to B_\bullet) = (A \to B)
\]
may be modified to a diagram in $\Dd_{A/}$ with colimit $A \to B$ as follows. Regarding $A$ and $B$ as constant diagrams in $\Dd$, a colimit diagram witnessing (*) gives rise to a commuting square (solid arrow diagram)
\[
\tag{**}
\xymatrix{
	A_\bullet \ar[r] \ar[d]
	& B_\bullet \ar@/^3ex/[dd] \ar@{.>}[d]
	\\
	A \ar@{=}[d] \ar@{.>}[r] & B'_\bullet \ar@{.>}[d]
	\\
	A \ar[r] & B
}
\]
in $\Fun(\Si, \Dd)$. Forming the pushout as indicated, we claim that the bottom square defines a colimit diagram in $\Dd_{A/}$. Recall that a colimit in an undercategory is computed as the colimit of the associated cone-shaped diagram in the underlying category. In the case at and, we may compute the colimit of the $\Si$-indexed diagram $B'_\bullet$ in $\Dd_{A/}$ as a colimit of the associated $\Si^\triangleleft$-indexed diagram in $\Dd$. By Quillen's theorem A \cite[Section 7.2.3, tag 02NX]{Kerodon}, $\Si$ is cofinal in $\Si^\triangleleft$. So it's enough to show that we have the colimit
\[
\tag{***}
\colim B_\bullet' = B
\]
in $\Dd$. We denote the objects of $\Si$ like so:
\[
1 \from 0 \to 2.
\]
With this notation, the upper left portion of (**) gives rise to a commuting diagram (upper left portion below)
\[
\xymatrix{
	A & A_1 \ar[l] \ar[r] & B_1 & B'_1
	\\
	A \ar[u] \ar[d] & A_0 \ar[l] \ar[u] \ar[r] \ar[d]
	& B_0 \ar[u] \ar[d] & B'_0 \ar[u] \ar[d]
	\\
	A & A_2 \ar[l] \ar[r] & B_2 & B'_2
	\\
	A & A \ar[l]^-= \ar[r] & B  &B.
}
\]
We may take colimits along rows and then compute the colimit of the resulting diagram, or, by the Fubini theorem, we may equally reverse the order. As indicated in the diagram, this establishes (***). 

\spar{40510C}
In the setting of paragraph \ref{40510B}, if 
\[
A_i \to B_i
\]
is a split injection, then so is $A \to B'_i$. Hence, applying the construction of \ref{40510B} to the cofiber diagram of paragraph \ref{40510A}, we find that an arbitrary object $A \to B$ of $\Dd_{A/}$ may be written as a pushout of a span-shaped diagram whose vertices are split injections $A \to A \oplus C_i$. 

\spar{40510D}
The Bousfield-Kan formula (see e.g. Corollary 12.3 of Shah \cite{shah2023parametrized}) rewrites any colimit as a geometric realisation of a simplicial diagram whose vertices are coproducts among the terms of the original diagram. If $A \to A \oplus C_1$, $A \to A \oplus C_2$ are split injections, then the coproduct in $\Dd_{A/}$ is given by the pushout in $\Dd$, which is the split injection
\[
A \to A \oplus C_1 \oplus C_2.
\]
Thus, we conclude the proof of lemma \ref{40510F} by applying the Bousfield-Kan formula to the diagram obtained in paragraph \ref{40510C}.

\medskip

We now complete the proof of proposition \ref{21023F}.

\spar{30522B}
Let $\Fin_\m{inj}^{\le n}$ be the full subcategory of $\Fin_\m{inj}$ spanned by sets of cardinality $\le n$. Then the evident inclusion $B\Si_n \subset \Fin_\m{inj}^{\le n}$ obeys the hypothesis imposed on $\delta$ in Example \ref{30616B}. Indeed, for $\langle k\rangle^\circ = \{1, \dots, k\} \in \Fin_\m{inj}^{\le n}$, there's an evident equivalence of 1-categories 
\[
B\Si_{n-k} \xto{\sim} (B\Si_n)_{\langle k \rangle ^\circ/}. 
\]
A categorical equivalence is in particular a weak equivalence in the Kan model structure on the 1-category of simplicial sets. The classifying space of a finite group is rationally contractible.

\spar{21023H}
Fix $(\one \xto{x} V) \in \Cc_{\one/}$, and assume $\Cc$ is presentably tensored over the \oo-category $\Hh_\QQ$ of rational spaces with its Cartesian symmetric monoidal structure. In paragraph \ref{21025B}, we associated to $x$ a diagram
\[
V^{x \otimes} : \m{Fin}_\m{inj} \to \Cc.
\]
Proposition \ref{21025C} provides an equivalence in $\Cc$,
\[
\Sym' (\one \xto{x} V) = \colim V^{x \otimes}.
\]
Let $\vec\NN$ denote the (1-)category associated to the totally ordered set $\NN$; that is, the objects are the natural numbers, and there's a unique morphism $n \to m$ if and only if $n \le m$. Let $r$ denote the evident functor
\[
\Fin_{inj} \to \vec\NN.
\]
Denote the restriction of $V^{x \otimes}$ to $\m{Fin}_\m{inj}^{\le n}$ by $V^{x \otimes}_{\le n}$. By Theorem 0300 of Kerodon \cite{Kerodon}, there exists a left Kan extension 
\[
K: \vec{\NN} \to \Cc
\]
of $V^{x \otimes}$ along $r$. The left Kan extension $K$ comes with an equivalence in $\Cc$
\[
\colim V^{x \otimes} = \colim K
\]
on the one hand, and equivalences in $\Cc$ 
\[
K(n) = \colim V^{x \otimes}_{\le n}
\]
on the other. By $(\text{Quillen's theorem A})_\QQ$ (Example \ref{30616B}) applied as in paragraph \ref{30522B}, for each $n \in \NN$, the inclusion $B\Si_n \subset \m{Fin}_\m{inj}^{\le n}$ induces an equivalence
\[
\Sym^n V \xto{\sim} \colim V^{x \otimes}_{\le n}.
\]

\spar{21023Ha}
It remains to analyze the induced transition maps
\[
\Sym^n V \to \Sym^{n+1}V.
\]
This is pure abstract nonsense based on fundamental techniques for working with colimits. Let
\[
\iota_s: \langle n \rangle^\circ \to 
\langle n+1 \rangle^\circ
\]
be the order-preserving injection
\[
\{1, \dots, n\} \simeq \{1, \dots, \hat s, \dots, n+1\}.
\]
We can associate to $\iota_s$ a diagram of 1-categories and functors
\[
\xymatrix
@C=2ex
@R=2ex
{
	\{\langle n \rangle ^\circ \}
	\ar@{}[r]|\subset^a
	\ar[d]
	&
	B\Si_n
	\ar@{}[r]|\subset^b
	\ar[d]_{\tau_s}
	&
	\m{Fin}_\m{inj}^{\le n}
	\ar@{}[d]|\cap^c
	\\
	\{\langle n+1 \rangle ^\circ \}
	\ar@{}[r]|\subset_d
	&
	B\Si_{n+1}
	\ar@{}[r]|\subset_e
	&
	\m{Fin}_\m{inj}^{\le n+1}
	\ar@{}[r]|\subset_f
	&
	\m{Fin}_\m{inj}
}
\]
in which the left square commutes on the nose, while the right square is filled in by a canonical (noninvertible) natural transformation
\[
\eta_s: fcb \to fe \tau_s.
\]
Taking colimits in $\Cc$, we find that the left square, together with the natural transformation $\eta_s$, establishes an equivalence between the map
\[
(\tau_s)_*:\Sym^n V \to \Sym^{n+1}V
\]
induced by $\tau_s$, and the map induced by 
\begin{align*}
	V^{\otimes n} 
	\simeq 
	V^{\otimes \{1, \dots, s-1\}}
	&
	\otimes \one \otimes 
	V^{\otimes \{s+1, \dots, n+1\}}
	\\
	&
	\xto{\m{Id} \otimes x \otimes \m{Id}}
	V^{\otimes \{1, \dots, s-1\}}
	\otimes V \otimes 
	V^{\otimes \{s+1, \dots, n+1\}}
	\simeq V^{\otimes n+1}.
\end{align*}
Meanwhile, the right square induces a coherently commuting square of colimits in $\Cc$
\[
\xymatrix{
	\Sym^n V 
	\ar[d]_{(\tau_s)_*}
	\ar[r]^\sim
	& 
	\colim (V|_{\m{Fin}_\m{inj}^{\le n}})
	\ar[d]^{c_*}
	\\
	\Sym^{n+1}
	\ar[r]^\sim
	&
	\colim (V|_{\m{Fin}_\m{inj}^{\le n +1}})
}
\]
in which $c_*$ is the map of colimits induced by the left Kan extension. The case $s=1$ (applied to each $n \in \NN$) establishes an equivalence between the diagram constructed here and the diagram indicated in the proposition. This completes the proof of Proposition \ref{21023F}.

\medskip

When $\Cc$ is stable and $\QQ$-linear, we may further analyze the colimit \ref{21023F}(*) using the following proposition.

\begin{proposition}
	\label{30412A}
	Let $\Dd$ be a stable symmetric monoidal $\QQ$-linear $\infty$-category with unit object $\one$ and suppose given an exact triangle 
	\[
	\one \xto{x} V \xto{q} Q.
	\]
	Then for each $n \ge 0$ there's an associated exact triangle
	\[
	\Sym^n V \xto{\mu_n(x)}
	\Sym^{n+1} V \xto{q_n(x)}
	\Sym^{n+1} Q.
	\]
\end{proposition}

The proof spans paragraphs \ref{30412B}--\ref{30412D}.

\spar{30412B}
We fix $n$ throughout the discussion and the proof. There are functors 
\[
\mu_n, q_n:\Dd_{/\one} \rightrightarrows \Fun(\Delta^1, \Dd)
\]
such that for each object $x:\one \to V $ of $\Dd_{/\one}$, $\mu_n(x)$ is a composition 
\[
\Sym^n V 
\xto{x \otimes id} V \otimes \Sym^{n} V
\to \Sym^{n+1}V
\]
and $q_n(x) = \Sym^{n+1}(q)$. This defines the morphisms alluded to in the proposition. Moreover, it's not hard to see that the composition $q_n(x) \circ \mu_n(x)$ is equivalent to the zero morphism.

\spar{30412C}
Consider the functor $\ka: \Fun(\Delta^2, \Dd) \to \Fun(\Delta^1, \Dd)$ induced by the face map $\Delta^1 \to \Delta^2$, $\{0,1\} \simeq \{0,2\} \subset \{0,1,2\}$. There's a natural functor 
\[
z:\Dd \times \Dd \to \Fun(\Delta^1, \Dd)
\]
such that for each $(A,C) \in \Dd \times \Dd$, $z(A,C)$ is a zero morphism $A \to C$. We define the $\infty$-category $ZS(\Dd)$ of \emph{zero sequences in $\Dd$} by the pullback
\[
ZS(\Dd) := \Fun(\Delta^2, \Dd)
\times_{\Fun(\Delta^1, \Dd)}
(\Dd \times \Dd).
\]
Evidently, the functors $\mu_n, q_n$ together give rise to a functor
\[
(\mu_n, q_n):\Dd_{/\one} \to ZS(\Dd).
\]
On the other hand, there's a natural functor 
\[
\ep:ZS(\Dd) \to \Fun(\Delta^1, \Dd)
\]
which sends a zero-squence
\[
A \xto{\al} B \xto{\be} C
\]
to the associated morphism
\[
\m{cof}(\al) \to C.
\]
Composing $(\mu_n,q_n)$ with $\ep$, we obtain a natural transformation
\[
\de:\m{cof}(\Sym^n V \to \Sym^{n+1}V)
\to \Sym^{n+1} Q
\]
of functors 
\[
\Dd_{\one/} \rightrightarrows \Dd.
\]
We wish show that $\de$ is an equivalence. 

\spar{30412D}
Since the tensor product, by assumption, preserves colimits separately in each variable, and since $\Dd$ is $\QQ$-linear, both the source and target of $\de$ preserve sifted colimits. Since $\Dd_{\one/}$ is generated by split injections under sifted colimits (Lemma \ref{40510F}), it suffices to check on split injections. Using again our assumption that the tensor product preserves colimits separately in each variable (and hence distributes over direct sums), we find that $\de$ is an equivalence as hoped. This completes the proof of Proposition \ref{30412A}.

\section{Review of rational motivic algebras}
\label{sec:ReviewRatMot}

The foundations of rational motivic algebras (or ``motivic dga's'') were studied in \cite{dancohen2021rational, Iwanari}. The model-categorical foundations of \cite{dancohen2021rational} will allow us to rely on both model categorical and infinity categorical techniques, while the careful \oo-categorical treatment of \cite{Iwanari} will allow us to import many foundational constructions and expected compatibilities. In the interest of readability, we will be rather less careful about such matters and we strongly recommend the reader consult the above articles for background. 

\spar{27J1}
Let $k$ be a field of characteristic zero. Let $\Sm_k$ denote the category of smooth schemes over $k$, let $\Cpx \PSh(\Sm_k, \QQ)$ denote the category of complexes\footnote{Here and below, complexes are not assumed to be bounded.} of presheaves of $\QQ$-vector spaces. If $X$ is a smooth scheme over $k$, we commit the usual abuse of notation by denoting the associated presheaf again by $X$. Continuing with this abuse, we let $X\otimes \QQ$ denote the presheaf whose value on a smooth $k$-scheme $Y$ is the $\QQ$-vector space of formal linear combinations of elements of $\Hom_k(Y,X)$.

 We let
 \[
 \MM^\m{eff}(k,\QQ)
 =
 \big(
 \Cpx \PSh (\Sm_k,\QQ),
 (\AA^1, \et)\mbox{-local}
 \big)
 \]
 denote the category of complexes of presheaves, endowed with the ($\AA^1$, \'et)\emph{-local model structure}. We recall briefly what this means, referring to \cite[\S3]{AyoubEtale} for details. We start with the projective model structure, in which the weak equivalences are defined to be the quasi-isomorphisms and the fibrations are defined to be the degree-wise surjections. We left-localise with respect to the class of morphisms which induce isomorphisms of cohomology presheaves after sheafification with respect to the \'etale topology. We then left-localise with respect to the class of morphisms 
\[
\AA^1_X\otimes \QQ[n] 
\to
X \otimes \QQ[n] 
\]
for $X \in \Sm_k$ and $n\in \ZZ$.
 
We set 
\[
T\coloneqq 
\frac{\PP^1_k \otimes \QQ[0] }
 {\infty \otimes \QQ[0]}.
\]
We let $\MM(k, \QQ)$ denote the category of symmetric $T$-spectra in $\MM^\m{eff}(k,\QQ)$ as in Hovey \cite[\S7]{HoveySpectra}, or, equivalently but more simply, the category of commutative $T$-spectra of Ayoub \cite[\S4]{AyoubComodules}, endowed with the \emph{stable} $(\AA^1,\et)$-local model structure, as in Ayoub \cite[\S3]{AyoubEtale}. There's a natural left Quillen functor 
\[
\opnm{Sus}^0_T: \MM^\m{eff}(k,\QQ) \to \MM(k,\QQ),
\]
and we set
\[
\QQ(0) := \opnm{Sus}^0_T(\QQ[0]).
\]
This makes $\MM(k,\QQ)$ into a symmetric monoidal model category with unit object $\QQ(0)$ and tensor product given by the natural (levelwise) tensor product. The \emph{triangulated category of \'etale motives over $k$} is given by
\[
h\DM(k,\QQ) := \ho \MM(k,\QQ).
\footnote{In the literature on motives, the category we've constructed is more often denoted $\DA(k, \QQ)$ or $\DA^\et(k, \QQ)$, while the notation $\DM(k, \QQ)$ is reserved for the more central and more well known version with correspondences and the Nisnevich topology. However, with rational coefficients there is a well documented equivalence $h\DM(k, \QQ) \simeq h\DA(k, \QQ)$ coming from a Quillen equivalence of model categories.}
\]
The \emph{homological motives functor}
\[
C_* =
C_*(-,\QQ)
: \Sm_k \to \MM(k,\QQ) \to h\DM(k,\QQ)
\]
is given on an object $X\in \Sm_k$ by
\[
C_*(X,\QQ) = 
\Sus^0_T(X\otimes \QQ[0]).
\]
The \emph{Tate object} is given by $\QQ(1):= \Sus^0_T(T[-2])$.

\spar{27J2}
We define the \emph{rational motivic model category} $\Mdga = \Mdga(k, \QQ)$ to be the category of commutative monoids in $\MM(k,\QQ)$. We define a morphism in $\Mdga$  to be a \emph{weak equivalence} if the associated morphism in $\MM(k,\QQ)$ is a weak equivalence, and a \emph{cofibration} if the associated morphism in $\MM(k,\QQ)$ is a cofibration. According to Propositions 1.6 and 1.12 of \cite{dancohen2021rational}, this makes $\Mdga(k, \QQ)$ into a cofibrantly generated left-proper model category. 

We fix once and for all a fibrant replacement
\[
\one \to \one^f
\]
of the unit object in $\Mdga$. The homological motives functor factors through a functor 
\[
\operatorname{C_*} = 
\operatorname{C_* }( \;\cdot\; , \QQ): \Sm_k \to \opnm{CcoAlg}(\MM(k, \QQ))
\]
to the category of \emph{commutative coalgebras} in $\MM(k,\QQ)$; this functor is given by
\[
C_*(X) := \Sus^0_T(X \otimes \QQ[0])
\]
with counit
\[
C_*(X) \to C_*\big(\Spec(k)\big) = \QQ(0)
\]
induced by the structure map of $X$, and comultiplication
\[
C_*(X) \to C_*(X) \otimes C_*(X)
\]
induced by the diagonal of $X$. As noted in the proof of Theorem 1.14 of \cite{dancohen2021rational}, the formula 
\[
C^*(X, \QQ):=
  \underline \Hom_{\MM(k,\QQ)}(C_*(X, \QQ),\onef)
\] 
defines a functor 
\[
C^* = C^* ( \;\cdot\; , \QQ) : \Sm^\op_k \to \opnm{CcoAlg} \MM(k,\QQ).
\]

\spar{27J21}
Let $\Dd^\otimes = \DM(k,\QQ)^\otimes$ denote the stable symmetric monoidal $\infty$-category associated to $\MM(k, \QQ)$. We do not distinguish notationally between the symmetric monoidal category $\Dd^\otimes$ and its underlying $\infty$-category $\Dd = \Dd^\otimes_{\langle 1 \rangle}$ when we see no danger of confusion. The work of Hinich \cite{HinichRectification} provides an equivalence between the $\infty$-category associated to the model category $\Mdga(k,\QQ)$ and $\Cc = \CAlg \Dd$, and we define the \emph{rational motivic algebra functor} to be the resulting functor
\[
\Sm^\op_k \to \Cc = \CAlg \DM(k, \QQ).
\]
For the most part, we denote the motivic  algebra functor again by $C^*$. When the need to distinguish between the motivic algebra functor and the \emph{cohomological motives functor} 
\[
\Sm^\op_k \to \Dd = \DM(k, \QQ),
\] 
we will denote the latter by $\underline C^*$.

\spar{27J3}
The motivic algebra functor may be retrieved directly from the cohomological motives functor without the need to specify a cofibrant replacement of the unit via standard general infinity-categorical constructions. A symmetric monoidal functor
\[
\Ff^\otimes: \Aa^\otimes \to \Bb^\otimes
\]
between symmetric monoidal $\infty$-categories upgrades naturally to a functor 
\[
\Ff^a: \CAlg \Aa \to \CAlg \Bb. 
\]

To apply this to motives, we endow $\Sm_k^\op$ with its coCartesian structure (in which tensor product is given by the ordinary product of schemes), so that the symmetric monoidal functor~$C^*\colon(\Sm_k^\op)^{\coprod}\to\cD^\otimes$ upgrades to a functor
\[
\Sm_k^\op=\CAlg((\Sm_k^\op)^{\coprod}) \to \CAlg(\cD^\otimes) = \cC \,
\]
(see Corollary 2.4.3.10 of \cite{LurieAlg}). This upgraded functor is automatically coCartesian-symmetric monoidal, so we have equivalences of rational motivic algebras
\[
C^*(X\times Y) \simeq C^*(X) \otimes C^*(Y).
\]

\section{Augmentation spaces, computation for abelian varieties}
\label{sec:AugmentationSpaces}

In this section, as above, we work over an arbitrary field $k$, though much of the material may be generalised to a positive dimensional base with little change. 

\spar{27J4}
For $X \in \Sm_k$, we define the \emph{rational motivic augmentation space of $X$ over $k$} by 
\[
\Aaug(X/k) := \Hom_{\Cc}(C^*(X), \one)
\]
and we define the \emph{rational motivic augmentation set} by
\[
\Aug(X/k) = \pi_0 \Aaug(X/k) = \Hom_{h\Cc}(C^*(X), \one) \,.
\]
Composing with the presheaf 
\[
\Hom_\Cc(?, \one):(\Cc^{\coprod})^\op \to \Hh^{\prod}
\]
to the category of homotopy types with Cartesian monoidal structure, we find that the assignment 
\[
X \mapsto \Aaug(X/k)
\]
 belongs to a monoidal functor
\[
\Aaug(?/k): \Sm_k^{\prod} \to \Hh^{\prod}.
\]

As a direct consequence of the constructions we have augmentation maps
\[
\kappa\colon X(k) \to \Aaug(X/k)
\]
which form part of a natural transformation 
\[
?(k) \to \Aaug(?/k)
\quad
\mbox{of functors}
\quad
\Sm_k \to \Hh. 
\]

\begin{remark}
	If~$k'/k$ is an extension of fields, then there is a base-change functor $\DM(k,\bQ) \to \DM(k',\bQ)$ taking $C^*(X)$ to $C^*(X_{k'})$ for every $X\in\Sm_k$. This functor induces a map $\cAug(X/k)\to\cAug(X_{k'}/k')$ compatible with augmentation maps.
\end{remark}

\spar{27J4a}
The fact that $X\mapsto \Aug(X/k)$ preserves products implies that there is an induced (commutative) group structure on $\Aug(A/k)$ for any (commutative) group scheme $A$, and the augmentation map
\[
A(k) \to \Aug(A/k)
\]
is a group homomorphism. In fact, by work of Ancona et.\ al.\ \cite{AnconaHuberI, AnconaHuberII}, lifted to the $\infty$-categorical level by Iwanari \cite{Iwanari}, the rational motivic augmentations of a semiabelian variety are easy to understand: as we show in Propositions \ref{10OctA} and \ref{10Oct1}, as a set $\Aug(A/k) = \pi_0 \Aaug(A/k) \simeq A(k) \otimes \QQ$ and, moreover, the space $\Aaug(A/k)$ is discrete.

\begin{proposition}
\label{10OctA}
Let $A$ be a semiabelian variety over a field $k$. Then the map
\[
A(k) \to \Aug(A/k)
\]
induces an isomorphism $A(k) \otimes \QQ \simeq \Aug(A/k)$. 
\end{proposition}

The proof spans paragraphs \ref{27Jps}--\ref{27Jpe}.

\spar{27Jps}
We let $\Sh_\et(?)$ denote the category of \'etale sheaves of $\QQ$-\textit{vector spaces}, and we let $\Cpx(?)$ denote the category of complexes. We endow $\Cpx(\Sh_\et(\Sm_k))$ with the \'etale-local projective model structure. 

Let $G/k$ be a semiabelian variety. By \cite{AnconaHuberI, AnconaHuberII}, $M_1(G)$ may be defined simply as the \'etale sheaf 
\[
\Sm_k^\m{op} \to \Vect \QQ
\]
given by
$
M_1(G) = G \otimes \QQ,
$
regarded as a complex concentrated in degree $0$. As shown in loc.\ cit., $M_1(G)$ is homotopy invariant (i.e., $\AA^1$-local), hence belongs to the essential image of the fully faithful right adjoint
\[
\ho\Cpx(\Sh_\et(\Sm_k)) \inj \ho\DM^\m{eff}(k,\QQ)
\]
In turn, according to \cite{VoevTriCat}, $ \ho\DM^\m{eff}(k,\QQ)$ embeds as a full subcategory of $\ho\DM(k,\QQ)$. Finally, we note that $\Sh_\et(\Sm_k)$ embeds in the derived category $\ho\Cpx(\Sh_\et(\Sm_k))$ as the full subcategory consisting of objects concentrated in degree 0.

\spar{27Jpe}
By Iwanari \cite[Theorem 6.16]{Iwanari}, $C^*(A) = \Sym M_1(A)^\lor$ is a free commutative algebra in $\Dd$. Therefore
\begin{align*}
\Aaug(A/k) &= \Hom_\Dd \big(M_1(A)^\lor, \QQ(0) \big).
\end{align*}
Hence
\begin{align*}
\Aug(A/k) &= \Hom_{\ho\Dd}  \big(M_1(A)^\lor, \QQ(0) \big)
\\
&= \Hom_{\ho\Dd} \big(\QQ(0), M_1(A) \big)
\\
&= \Hom_{\Sh_\et(\Sm_k)}  \big( \QQ_{\Sm_k}, M_1(A) \big)
\\
&= A(k) \otimes \QQ. 
\end{align*}
This completes the proof of Proposition \ref{10OctA}.

\begin{theorem}
\label{10Oct1}
Let $A$ be a semiabelian variety over a field $k$. Then the map of spaces
\[
A(k) \to \Aaug(A/k)
\]
induces an isomorphism of (a fortiori discrete) homotopy types
\[
A(k) \otimes \QQ \simeq \Aaug(A/k).
\]
\end{theorem}

The proof (which amounts essentially to noting that the calculation of Proposition \ref{10OctA} lifts to the $\infty$-categorical level without change) spans paragraphs \ref{10Oct2}--\ref{10Oct3}.

\spar{10Oct2}
We recall elementary facts about localisation of $\infty$-categories. Let $\Cc$ be an $\infty$-category, $L:\Cc \to \Dd$ a localisation functor, and $\rho: \Cc \from \Dd$ a right adjoint of $L$, which, by assumption, is fully faithful.

\begin{Claim*}
The counit
$
L \circ \rho \to \m{Id}_\Dd
$
of the adjunction is an equivalence.
\end{Claim*}

\begin{proof}
After replacing $\Dd$ by the essential image of $\rho$, we may assume that $\rho$ is the inclusion of a full subcategory $\Dd \subset \Cc$. Fix $X \in \Dd$. Then $L\rho X \to X$ induces an equivalence 
\[
\Hom_\Dd( L\rho X, \cdot) \to \Hom_\Dd(X, \cdot)
\] 
of functors
\[
\Dd^\m{op} \to \Hh,
\]
since for any $Y \in \Dd$ we have isomorphisms of homotopy types
\[
\Hom_\Dd(L\rho X, Y) = \Hom_\Cc(\rho X, \rho Y)
= \Hom_\Dd(X,Y);
\]
the claim follows by the Yoneda Lemma. 
\end{proof}

\begin{Corollary*}
Suppose given $X, Y \in \Cc$, and assume $Y$ belongs to the essential image of $\rho$. Then we have an isomorphism of homotopy types
\[
\Hom_\Dd(LX, LY) = \Hom_\Cc(X,Y).
\]
\end{Corollary*}
\begin{proof}
By assumption we have an equivalence $Y \simeq \rho Y'$ for some $Y' \in \Dd$. So by the claim,
\begin{align*}
\Hom_\Dd(LX, LY)
&= \Hom_\Cc(X, \rho L \rho Y')
\\
&= \Hom_\Cc(X, \rho Y') 
\\
&= \Hom_\Cc(X,Y). \qedhere
\end{align*}
\end{proof}

\spar{11OctA}
Let $f:\Cc \to \Dd$ be a functor of pointed $\infty$-categories such that the induced functor of homotopy categories
\[
hf\colon h\Cc \to h\Dd
\]
is fully faithful. Assume that $\Cc$ admits suspensions, and that $f$ preserves suspensions. Then $f$ is fully faithful.

\begin{proof}
Given $X,Y \in \Cc$, 
\begin{align*}
\pi_i \Hom_\Cc(X,Y) 
&\simeq
\Hom_{\ho\Cc}(X[i], Y)
\\
&\xto{\simeq}
\Hom_{\ho\Dd}(fX[i], fY)
\\
&\simeq
\pi_i \Hom_\Dd(fX, fY)
\end{align*}
for all $i \ge 0$. Hence
\[
\Hom_\Cc(X, Y) \to \Hom_\Dd(fX, fY)
\]
is a weak homotopy equivalence. 
\end{proof}

\spar{zehamburger}
If $\Aa$ is an abelian category, we denote by $\Dd(\Aa)$ the unbounded derived $\infty$-category of $\Aa$; we recall that the functor $A \mapsto A[0]$ lifts to a fully faithful functor $\Aa \to \Dd(\Aa)$. We denote the localisation functor
\[
\Dd(\Sh_\et(\Sm_k))
\to
\DM^\m{eff}(k,\bQ)
\]
by $L_{\AA^1}$ and the stabilisation functor 
\[
\DM^\m{eff}(k,\bQ) \to \DM(k,\bQ)
\]
by $S_{\QQ(1)}$.

\begin{Claim*}
Let $\Ff, \Gg \in \Sh_\et(\Sm_k)$ be homotopy invariant sheaves. Then we have an isomorphism of (a fortiori discrete) homotopy types
\[
\Hom_{\Sh_\et(\Sm_k)}(\Ff, \Gg) 
=
\Hom_{\DM(k)}
\big(
S_{\QQ(1)}L_{\AA^1} \Ff[0],
S_{\QQ(1)}L_{\AA^1} \Gg[0]
\big).
\]
\end{Claim*}

\begin{proof}
According to \S\ref{11OctA} combined with the full faithfulness of $hS_{\QQ(1)}$, the functor $S_{\QQ(1)}$ is fully faithful. Hence we have an isomorphism of homotopy types
\begin{align*}
\Hom_{\DM(k)}
\big(
S_{\QQ(1)}L_{\AA^1} \Ff[0],
S_{\QQ(1)}L_{\AA^1} \Gg[0]
\big)
&=
\Hom_{\DM^\m{eff}(k)}
\big(
L_{\AA^1} \Ff[0],
L_{\AA^1} \Gg[0]
\big)
\\
&= \Hom_{D(\Sh_\et(\Sm_k))}
(\Ff[0], \Gg[0])
\intertext
{according to \S\ref{10Oct2},}
\\
&=\Hom_{\Sh_\et(\Sm_k)}
(\Ff, \Gg). 
\qedhere
\end{align*}
\end{proof}

\spar{10Oct3}
Consequently, in the situation and the notation of Proposition \ref{10Oct1}, we have isomorphisms of homotopy types
\begin{align*}
\Aaug(A/k) &= \Hom_{\Dd}  \big(M_1(A)^\lor, \QQ(0) \big)
\\
&= \Hom_{\Dd} \big(\QQ(0), M_1(A) \big)
\\
&= \Hom_{\Sh_\et(\Sm_k)}  \big( \QQ_{\Sm_k}, M_1(A) \big)
\\
&= A(k) \otimes \QQ. 
\end{align*}
This completes the proof of Proposition \ref{10Oct1}.

\begin{remark}
\label{10Oct4}
Assume $\Dd = \DM(k,\bQ)$ admits a motivic t-structure with heart $\MMm_k$ with expected properties. Then the vanishing of the homotopy groups $\pi_i$ for $i \ge 1$ of $\Aaug(A/k)$ at the origin follows in a different way, as we now explain. We assume that $\Dd = \Dd\MMm_k$, that for $X \to \Spec(k)$ smooth proper of dimension $d$, $h_i(X)$ is pure of weight $-i$, that for such $X$
\[
h_i(X) \simeq h^{2d-i}(X)(d),
\]
and that for $A \to \Spec(k)$ an abelian variety of dimension $d$,
\[
M_1(A) = h_1(A)[1].
\]
Then we have isomorphisms of groups
\begin{align*}
\pi_i \Aaug(A)
&= \pi_i \Hom_\Dd \big( M_1(A)^\lor , \QQ(0) \big)
\\
&= \Hom_{h\Dd} \big(
M_1(A)^\lor[i], \QQ(0)
\big)
\\
&=  \Hom_{h\Dd} \big(
h_{2d-1}(A)(-d)[i-1], \QQ(0)
\big)
\\
&=  \Ext_{\Mm\Mm(k)}^{1-i} \big(
h_{2d-1}(A), \QQ(d)
\big).
\end{align*} 
For $i=1$ this is $\Hom_{\MMm_k}$ between pure objects of different weights, hence trivial. The higher homotopy groups correspond to negative $\Ext$ groups. 
\end{remark}

\section{The rational motivic algebra of a $\Gm$-torsor}
\label{sec:AlgebraOfGmBundle}

\spar{30428A}
Let $k$ be a field. If $X$ is a smooth $k$-scheme, then we have the associated motivic algebra $C^*(X) \in \Cc = \CAlg \DM(k, \QQ)$ by paragraph \ref{27J21}; although we will not need this for our present applications, $C^*(X) \in \Cc$ may be defined more generally for $X \to \Spec k$ finite type and separated \cite[Remark 3.3]{Iwanari}.

Let $g\colon P = L^\times \to X$ be the $\Gm$-torsor associated to a line bundle $L$, let
\[
\tag{*}
c_1\colon \QQ(-1)[-2] \to C^*(X)
\]
be the 1st Chern class of $L$, a morphism in $\Dd = \DM(k,\QQ)$, and let
\[
\mu\colon C^*(X) \otimes C^*(X) \to C^*(X)
\]
denote the multiplication map. The Gysin triangle \cite[Theorem~15.15]{MazzaVoevodskyWeibel} provides a triangle in $\Dd$ 
\[
C^*(X)(-1)[-2] 
\xto{C_1} C^*(X) \to C^*(P),
\]
in which $C_1$ is the composition
\[
C^*(X) \otimes \QQ(1)[-2]
\xto{id_{C^*(X)} \otimes c_1}
C^*(X) \otimes C^*(X)
\xto{\mu}
C^*(X).
\]
This describes $C^*(P)$ as an object of $\Dd$. Our goal here is to describe its structure as a commutative algebra. Our main result is the following.

\begin{theorem}
\label{30428Aa}
Let $e\colon C^*(X) \to \ep(L^\times)$ be the cofiber of the first Chern class $c_1$ of $L$ (\ref{30428A}*). Then $C^*(L^\times)$ is given by the fibered coproduct in $\Cc = \CAlg \Dd$:
\[
\tag{*}
C^*(L^\times) =  C^*(X) 
\otimes_{\Sym C^*(X)}
\Sym \ep(L^\times).
\]
\end{theorem}

\begin{remark}
\label{30428Ab}
Any trivialisation of $P$ provides a path $c_1 \to 0$ to the zero morphism in 
\[
\Hom_{\Dd}(\QQ(-1)[-2], C^*(X)) \,.
\]
In particular, the canonical trivialisation of the $\Gm$-torsor $P \times_X P \to P$ gives rise to a triangle (solid arrow diagram below)
\[
\xymatrix{
\QQ(-1)[-1] \ar[r]^-{c_1} \ar[dr]_-0 & C^*(X) 
\ar[d]^{C^*(g)} \ar@{.>}[r]^-e 
& \ep(P) \ar@{.>}[dl]
\\
& C^*(P)
}
\]
hence to a triangle as shown to the right. Consequently, the pushout in $\Cc = \CAlg \Dd$ (solid arrow diagram below)
\[
\tag{*}
\xymatrix{
\Sym C^*(X) \ar[r] \ar[d] & \Sym \ep(P) \ar[d] 
\\
C^*(X) \ar[r] & C'(P) \ar@{.>}[r]^{\theta'} & C^*(P).
}
\]
maps to $C^*(P)$ as shown. Our proof of Theorem \ref{30428Aa} however, revolves around a map $\theta$ which will eventually turn out to be quasi-inverse to $\theta'$.
\end{remark}

The proof spans paragraphs \ref{30428B}--\ref{30428qed}.

\spar{30428B}
We put ourselves in the general setting of a presentably symmetric monoidal $\infty$-category\footnote{Or, more generally, a symmetric monoidal category which admits realisations of simplicial objects and such that $\otimes$ preserves realisations separately in each variable.}, for which purpose we temporarily free the variable $\Dd$. Consider an object $E \in \Dd$, a commutative algebra $A \in \Cc\coloneqq \CAlg \Dd$, and a morphism $e\colon A \to E$ in $\Dd$ (this general $e$ will later be specialised back to the particular $e$ of Remark \ref{30428Ab}). Let $E'$ be given by the pushout in $\Dd_A\coloneqq\opnm{Mod}_A(\Dd)$ (lower square)
\[
\xymatrix{
A \ar[r]^-e \ar[d]_{pr_2} 
\ar@/_4ex/@<-.5ex>[dd]_-{id}
& E \ar[d] \\
A \otimes A \ar[r]^-{id_A \otimes e} \ar[d]_\mu & A \otimes E \ar[d] \\
A \ar[r] & E'.
}
\]
We let $\Sym_A$ denote the free commutative algebra functor over $A$
\[
\Dd_A \to \Cc_A\coloneqq \CAlg \Dd_A \simeq \Cc_{A/}.
\]
Then there's a natural equivalence in $\Cc_A$:
\[
\tag{*}
A \otimes_{\Sym A} \Sym E
\simeq
A \otimes_{\Sym_A A} \Sym_A E'.
\]
This is based on tried and true techniques of higher algebra (see especially \cite[\S4.5.2]{LurieAlg}), so we allow ourselves to argue somewhat informally.

We denote the left hand side of (*) by $\Xx$ and the right hand side by $\Xx'$. For fixed $B \in \Cc_A$, we have a sequence of equivalences in the $\infty$-category $\Hh$ of homotopy types
\begin{align*}
\Hom_{\Cc_A} (\Xx', B) 
&= \Hom_{\Dd_A}(E', B)
\times_{\Hom_{\Dd_A}(A,B)} \{\ast\}
\\
&= \Hom_{\Dd_A}(A \otimes E,B)
\times_{\Hom_{\Dd_A}(A \otimes A,B)} \{\ast\}
\\
&= \Hom_\Dd(E,B)
\times_{\Hom_{\Dd}(A,B)} \{\ast\}
\\
&= \Hom_\Cc (\Sym E, B)
\times_{\Hom_{\Cc}(\Sym A,B)} \{\ast\}
\\
&= \Hom_\Cc (\Xx, B)
\times_{\Hom_{\Cc}(A,B)} \{\ast\}
\\
&= 
\Hom_{\Cc_A}(\Xx, B)
\end{align*}
each of which may be promoted to an equivalence of functors.

\spar{30428C}
We return to the situation and the notation of paragraph \ref{30428A}.
There's an evident coherently commutative diagram in $\Mod_{C^*(X)}(\Dd)$:
\[
\tag{$\dagger$}
\xymatrix{
C^*(X) \ar[d] \ar[r] \ar@/_9ex/@{=}[dd]
&
C^*(P) \ar[d] \ar@/^9ex/@{=}[dd]
\\
\Sym_{C^*(X)} C^*(X) \ar[r] \ar[d]
&
\Sym_{C^*(X)} C^*(P) \ar[d]
\\
C^*(X) \ar[r]
& 
C^*(P).
}
\]
Let
\[
C'(P) = C^*(X) 
\otimes_{\Sym C^*(X)}
\Sym \ep(P)
\]
denote the right hand side of the equivalence \ref{30428Aa}(*) to be established. Combining diagram ($\dagger$) with paragraphs \ref{30428B} and \ref{30428A}, we find that $C'(P)$ fits into a  coherently commutative diagram in $\Dd$
\[
\tag{*}
\xymatrix{
\QQ(-1)[-2] \ar[r]^{c_1} \ar[d]
&
C^*(X) \ar[r] \ar[d]
&
\ep(P) \ar[d]
\\
C^*(X)(-1)[-2] \ar[r] \ar[r] \ar@{=}[d]
&
C^*(X) \otimes C^*(X) \ar[r] \ar[d]
&
C^*(X) \otimes \ep(P) \ar[d]
\\
C^*(X)(-1)[-2] \ar[r]^-{C_1}
&
C^*(X) \ar[r] \ar[d]
&
C^*(P) \ar[d] \ar@/^7pt/@{=}[ddr]
\\
&
\Sym_{C^*(X)} C^*(X) \ar[r] \ar[d]
&
\Sym_{C^*(X)} C^*(P) \ar[d]
\\
&
C^*(X) \ar[r]
&
C'(P) \ar[r]_-{\theta'}
&
C^*(P)
}
\]
in which the first three lines are exact triangles, the second square on the right is cocartesian in $\Mod_{C^*(X)}(\Dd)$, and the fourth square on the right is cocartesian in $\Cc_{C^*(X)/}$.\footnote{Thus, loosely speaking, Proposition \ref{30428Aa} says that a $C^*(X)$-linear map from $C^*(P)$ to a $C^*(X)$-algebra whose restriction to $C^*(X)$ is multiplicative, is itself multiplicative in a unique way.}
Included in this diagram is a map of $C^*(X)$-modules
\[
\theta\colon C^*(P) \to C'(P).
\]

\begin{lemma}
\label{30303A}
Suppose $\Dd$ is a presentably symmetric monoidal $\infty$-category, and let $A \in \Cc = \CAlg(\Dd)$ be a commutative algebra and $E \in \Dd$ an object. For $F \in \Dd_A =  \Mod_A(\Cc)$ let $\Sym^i_A(F)$ be the $i$th symmetric power of $F$ in $\Dd_A$. Then for $E \in \Dd$ and $i \ge 0$ there's an equivalence
\[
A \otimes \Sym^i E \simeq \Sym_A^i (A \otimes E)
\]
in $\Dd_A$.
\end{lemma}

\begin{proof}
The expected calculation in $\Dd_A$:
\begin{align*}
\Sym_A^i (A \otimes E)
&= \colim_{B\Si_i} (A \otimes E)^{\otimes_A i}
\\
&= \colim_{B\Si_i} 
\big( A \otimes  E^{\otimes i} \big)
\\
&=  A \otimes \colim_{B\Si_i}   E^{\otimes i} 
= A \otimes \Sym^i E,
\end{align*}
causes no trouble.
\end{proof}

\begin{lemma}
\label{30304A}
We have $\Sym^n (\QQ(i)[j]) = 0$ for $j$ odd and $n \ge 2$.
\end{lemma}

\begin{proof}
This is well known. In lieu of a sensible first proof of this fact, let us connect it with other, perhaps even better known facts. 

The symmetric power may equivalently be computed in the subcategory $\m{DMT}(k) \subset \Dd$ of mixed Tate motives (with no compactness conditions). Regardless of our assumptions on $k$, the calculation evidently reduces to the case $k = \ZZ$ where we have Beilinson--Soul\'e vanishing, hence an equivalence of stable symmetric monoidal $\infty$-categories
\[
\tag{*}
\m{DMT}(k) \simeq \Dd(G)
\]
with the derived $\infty$-category of representations of a proalgebraic group $G$. Moreover, $G$ comes equipped with a map $G \to \Gm$, and the equivalence (*) sends $\QQ(1)$ to a representation of $\Gm$ of rank $1$ concentrated in degree $0$. Finally, since we're working with rational coefficients, the symmetric power may equivalently be computed in the homotopy category $h \Dd(\Gm)$ (see, for instance, Lemma 5.19 of Iwanari \cite{Iwanari}) where the result is evident. 
\end{proof}

\spar{30428qed}
We're ready to complete the proof of Theorem \ref{30428Aa}. We write $\Sym^i_A$ for symmetric powers in $\Dd_A = \Mod_A(\Dd)$ and 
\[
\Sym'_A\colon 
(\Dd_A)_{\one/}=
\big(\Mod_A \Dd \big)_{A/}
 \to \Cc_A = (\CAlg \Cc)_{A/}
\]
for the relatively free commutative $A$-algebra functor (\ref{21023C} applied to $\Dd_A$). By Remark \ref{21023D}, the bottom pushout square on the right of diagram \ref{30428C}(*) identifies $C'(P)$ with the relatively free commutative algebra 
\[
\tag{*}
C'(P) \simeq \Sym'_{C^*(X)}
\big(
C^*(X) \to C^*(P)
\big).
\]
Proposition \ref{21023F} provides a further equivalence in $\Dd_A$ 
\[
(*) = \colim_{i} \Sym^i_{C^*(X)} C^*(P).
\]
By Proposition \ref{30412A}, the cofiber of the associated map
\[
\tag{**}
\Sym^i_{C^*(X)} C^*(P) \to \Sym^{i+1}_{C^*(X)} C^*(P)
\]
is equivalent to $\Sym^{i+1}_{C^*(X)} 
\big(C^*(X)(-1)[-1] \big)$. By Lemma \ref{30303A}, the latter is equivalent to 
\[
C^*(X) \otimes 
\Sym^{i+1} \big(\QQ(-1)[-1] \big). 
\]
By Lemma \ref{30304A}, the latter
\[
= 0 
\quad
\text{for }
i \ge 1.
\]
Consequently, the canonical map
\[
C^*(P)
\to 
\colim_{i} \Sym^i_{C^*(X)} C^*(P)
\]
is an equivalence. Thus, we've shown that the map of $C^*(X)$-modules
\[
\theta\colon 
C^*(P) \to C'(P)
\]
is an equivalence. Hence, the map of commutative $C^*(X)$-algebras $\theta'\colon C'(P) \to C^*(P)$ becomes an equivalence in the $\infty$-category $\Mod_{C^*(X)}\Dd$ of $C^*(X)$-modules. Hence $\theta$ is already an equivalence in the $\infty$-category of commutative algebras as claimed. This completes the proof of Theorem \ref{30428Aa}.

\section{The motivic augmentation space of a $\Gm$-bundle}
\label{sec:AugmentationsOfGmBundle}

We continue to work over a field $k$. As above, we set $\Dd = \DM(k,\bQ)$ and $\Cc = \CAlg \Dd$. The following proposition applies at the same level of generality as Theorem \ref{30428Aa}: if $X \to \Spec k$ is not smooth, then we refer to \cite[Remark 3.3]{Iwanari} for the motivic algebra $C^*(X) \in \Cc$. Our definition of the associated augmentation space applies without change: 
\[
\Aaug(X/k) = \Hom_\Cc 
\big( C^*(X), \one \big), 
\]
an object of the \oo-category $\Hh$ of homotopy types.

\begin{theorem}
\label{e3}
Let $X \to \Spec(k)$ be finite type and separated and let $L^\times \to Y$ be a $\Gm$-bundle. Then the motivic augmentation space $\Aaug(L^\times/k)$ has the structure of a pseudotorsoric left $k^\times \otimes \QQ$-module (Definition \ref{ts2}) over $\Aaug(X/k)$.
\end{theorem}

The proof/construction spans paragraphs \ref{w24a} - \ref{w24b}.

\spar{w24a}
If $\Ee$ is a stable $\infty$-category, we denote by $\uHom_\Ee$ the associated functor
\[
\Ee\times \Ee^\m{op} \to \Ss \Hh
\]
to the $\infty$-category $\Ss \Hh$ of stable homotopy types. We recall that there's an equivalence
\[
\Om^\infty \circ \uHom_\Ee \sim \Hom_\Ee.
\]
The exact triangle
\[
\QQ(-1)[-2] \xto{c_1}
C^*(X) \to \ep(L^\times) \to 
\QQ(1)[-1]
\]
in $\Dd$ gives rise to an exact triangle 
\[
\uHom_\Dd \big( \QQ(0), \QQ(1)[1] \big) 
\to
\uHom_\Dd \big( \ep(L^\times), \QQ(0) \big)
\to
\uHom_\Dd \big( C^*(X), \QQ(0) \big)
\]
in $\Ss \Hh$. Hence by Proposition \ref{ts3}, 
$
\Hom_\Dd \big( \ep(L^\times), \QQ(0) \big)
$
has the structure of a pseudotorsoric left 
$
\Hom_\Dd \big( \QQ(0), \QQ(1)[1] \big)
$-module 
over
\[
\Hom_\Dd \big( C^*(X), \QQ(0) \big).
\]
By Proposition \ref{10Oct1}, we have an equivalence of $E_\infty$-spaces (\ref{ts1})
\[
\Hom_\Dd \big( \QQ(0), \QQ(1)[1] \big) \simeq
k^\times \otimes \QQ \,.
\]

\spar{w24b}
By Theorem \ref{30428Aa}, the augmentation space sits in a pullback square
\[
\tag{*}
\xymatrix{
\Aaug(L^\times/k) \ar[d] \ar[r]^g & \Aaug(X/k) \ar[d]
\\
\Hom_\Dd \big( \ep(L^\times), \QQ(0) \big) \ar[r]_f
&
\Hom_\Dd \big( C^*(X), \QQ(0) \big)
}
\]
in $\Dd$. We complete the construction indicated in Proposition \ref{e3} by applying the construction of paragraph \ref{ts4}. This concludes the proof of Proposition \ref{e3}.

Now let us specialise to the case that the base is an abelian variety. In this case $\Aug(M^\times/k)$ is a $(k^\times\otimes\QQ)$-torsor over~$A(k)\otimes\bQ$:

\begin{corollary}
\label{w24d}
Let $A$ be an abelian variety over the field $k$ and let $M^\times \to A$ be a $\Gm$-torsor. Then the map
\[
\Aug(M^\times/k) \to \Aug(A/k) = A(k)\otimes\QQ
\]
is surjective, and the action of $\Aug(\Gm/k)=k^\times\otimes\QQ$ is simply transitive on each fiber.
\end{corollary}

\begin{proof}
We have
\[
\pi_{-1} \uHom 
_\Dd \big( \QQ(0), \QQ(1)[1] \big)
=
\pi_0\Hom_\Dd \big( \QQ(0), \QQ(1)[2] \big)
= K_0^{(1)}(k) = 0.
\]
Hence (as noted in paragraph \ref{w24bb}), the image of the map $f$ of diagram \ref{w24b}(*) intersects every connected component of its target. Consequently the same holds for the map $g$ of the same diagram, i.e.
\[
\pi_0g\colon \Aug(M^\times/k) \to \Aug(A/k) = A(k) \otimes \QQ
\]
is surjective.

By Proposition \ref{10Oct1}, $\Aaug(A/k) = A(k)\otimes\QQ$ is discrete. In particular all of its fundamental groups vanish. Hence by paragraph \ref{t23bc}, the $\Aaug(\Gm) = k^\times \otimes \QQ$-action is transitive and faithful on fibers. 
\end{proof}

The description of $\Aug(M^\times/k)$ in Proposition~\ref{w24d} is compatible with natural operations on torsors: pullbacks and pushout-products.

\begin{lemma}\label{lem:pullback_of_torsors}
	Let~$f\colon A \to B$ be a morphism of abelian varieties over~$k$, and let~$M^\times$ be a $\Gm$-torsor on~$B$. Let~$f^*M^\times$ denote the pulled-back torsor on~$A$. Then $\Aug(f^*M^\times/k)$ is the pullback of~$\Aug(M/k)$ along $f_*\colon\Aug(A/k)\to\Aug(B/k)$.
\end{lemma}
\begin{proof}
	The pullback projection map $\tilde f\colon f^*M^\times \to M^\times$ fits into a commuting square
	\begin{center}
	\begin{tikzcd}
		f^*M^\times \arrow[r,"\tilde f"]\arrow[d,two heads] & M^\times \arrow[d,two heads] \\
		A \arrow[r,"f"] & B
	\end{tikzcd}
	\end{center}
	in which the top arrow is $\Gm$-equivariant. Since the functor $\Aug(?/k)$ preserves products, it also preserves group actions; hence in the square
	\begin{center}
	\begin{tikzcd}
		\Aug(f^*M^\times/k) \arrow[r,"\tilde f_*"]\arrow[d,two heads] & \Aug(M^\times/k) \arrow[d,two heads] \\
		\Aug(A/k) \arrow[r,"f_*"] & \Aug(B/k)
	\end{tikzcd}
	\end{center}
	the topmost arrow is $\Aug(\Gm/k)$-equivariant. Thus the map~$\tilde f_*$ realises~$\Aug(f^*M^\times/k)$ as the pullback of~$\Aug(M^\times/k)$.
\end{proof}

\begin{lemma}\label{lem:pushout_product_of_torsors}
	Let~$A$ be an abelian variety over~$k$, and let~$M_1^\times$ and~$M_2^\times$ be $\Gm$-torsors over~$A$. Let~$M_1^\times\times_A^{\Gm}M_2^\times$ denote their pushout-product (i.e., the operation corresponding to tensor product of the corresponding line bundles). Then
	\[
	\Aug(M_1^\times\underset{A}{\overset{\Gm}{\times}}M_2^\times/k) = \Aug(M_1^\times/k)\underset{\Aug(A/k)}{\overset{\Aug(\Gm/k)}{\times}}\Aug(M_2^\times/k) \,,
	\]
	and the augmentation map
	\[
	M_1^\times(k)\times_{A(k)}^{k^\times}M_2^\times(k) \to \Aug(M_1^\times\times_A^{\Gm}M_2^\times/k)
	\]
	is the pushout-product of the augmentation maps $M_i^\times(k) \to \Aug(M_i^\times/k)$ for $i=1,2$.
\end{lemma}
\begin{proof}
	Consider the fiber product $M_1^\times\times_AM_2^\times$, which is a $\Gm^2$-torsor on~$A$. For $i=1,2$, the $i$th projection $p_i\colon M_1^\times\times_AM_2^\times \to M_i^\times$ is $\Gm^2$-equivariant, where the $\Gm^2$-action on~$M_i^\times$ is via the $i$th projection $p_i\colon\Gm^2\to\Gm$. So the induced map $p_{i*}\colon\Aug(M_1^\times\times_AM_2^\times/k)\to\Aug(M_i^\times/k)$ on motivic augmentations is $\Aug(\Gm^2/k)$-equivariant. Thus, the induced map
	\[
	\Aug(M_1^\times\underset{A}{\times}M_2^\times/k) \to \Aug(M_1^\times\underset{A}{\times}M_2^\times/k)
	\]
	is $\Aug(\Gm^2/k)$-equivariant, hence an isomorphism since both sides are torsors.
	
	The pushout-product $M_1^\times\times_A^{\Gm}M_2^\times$ is the pushout of $M_1^\times\times_AM_2^\times$ along the multiplication map $\mu\colon\Gm^2\to\Gm$, so there is a morphism $M_1^\times\times_AM_2^\times \to M_1^\times\times_A^{\Gm}M_2^\times$ over~$A$ which is $\Gm^2$-equivariant (via~$\mu$). Hence there is a map
	\[
	\Aug(M_1^\times/k)\underset{\Aug(A/k)}{\times}\Aug(M_2^\times/k) = \Aug(M_1^\times\underset{A}{\times}M_2^\times/k) \to \Aug(M_1^\times\underset{A}{\overset{\Gm}{\times}}M_2^\times/k)
	\]
	over~$\Aug(A/k)$ which is $\Aug(\Gm^2/k)$-equivariant via the multiplication map $\mu\colon\Aug(\Gm/k)^2\to\Aug(\Gm/k)$. This map induces the desired isomorphism
	\[
	\Aug(M_1^\times/k)\underset{\Aug(A/k)}{\overset{\Aug(\Gm/k)}{\times}}\Aug(M_2^\times/k) \xrightarrow\sim \Aug(M_1^\times\underset{A}{\overset{\Gm}{\times}}M_2^\times/k)
	\]
	The final statement about compatibility with augmentation maps is then easy to check.
\end{proof}
\section{Motivic N\'eron--Tate and Weil heights}
\label{sec:HeightMachine}

With the above formalism properly set up, we can now construct a theory of heights on rational motivic augmentation spaces. We do this in two settings: firstly we construct N\'eron--Tate height functions on $\Aug(A/k)$ for any abelian variety $A/k$, and then we use this theory to construct Weil height functions on $\Aug(X/k)^\lgeom$ for any smooth projective curve $X/k$ of genus~$\geq1$. Note that our constructions here are done in the opposite of the usual order: we define our Weil heights in terms of N\'eron--Tate heights instead of the other way around.

\subsection{N\'eron--Tate heights on abelian varieties}
\subsubsection{Global N\'eron--Tate heights}

Suppose that~$k$ is a global field, that~$A/\kbar$ is an abelian variety and~$M$ is a line bundle on~$A$. There is then an associated N\'eron--Tate height function
\[
\hat h_M\colon A(\kbar) \to \bR \,,
\]
which is a quadratic function, i.e., a sum of a quadratic form, a linear form and a constant on~$A(\kbar)$ (the constant is~$0$ in this case). Since we have
\[
\Aug(A/\kbar) = A(\kbar)\otimes\bQ = A(\kbar)/A(\kbar)_\tors
\]
by Proposition~\ref{10OctA}, it follows that~$\hat h_M$ factors uniquely through a quadratic function
\[
\hat h_M\colon \Aug(A/\kbar) \to \bR \,,
\]
which we call the \emph{motivic N\'eron--Tate height} associated to~$M$. This motivic height automatically inherits the corresponding properties of the classical N\'eron--Tate height.

\begin{proposition}\label{prop:neron-tate_heights}\leavevmode
	\begin{enumerate}
		\item For any line bundle~$M$, the composition of the motivic N\'eron--Tate height $\hat h_M\colon\Aug(A/\kbar)\to\bR$ with the augmentation map $A(\kbar)\to\Aug(A/\kbar)$ is the usual N\'eron--Tate height.
		\item For any line bundle~$M$, the motivic N\'eron--Tate height $\hat h_M$ is a quadratic function and satisfies~$\hat h_M(0)=0$. If~$M$ is symmetric (resp.\ antisymmetric) then~$\hat h_M$ is a quadratic form (resp.\ linear). If~$M$ is ample and symmetric then $\hat h_M$ is a positive-definite\footnote{A quadratic form $q\colon V\to\bR$ on a $\bQ$-vector space~$V$ is called positive-definite just when its unique extension to a real quadratic form $q_\bR\colon V\otimes\bR \to \bR$ is positive-definite, i.e., $q(v)>0$ for all $v\in(V\otimes\bR)\smallsetminus\{0\}$. This is a stronger condition than just requiring that $q(v)>0$ for $v\in V\setminus\{0\}$.} quadratic form.
		\item If~$M_1$ and~$M_2$ are line bundles on~$A$, then $\hat h_{M_1\otimes M_2} = \hat h_{M_1} + \hat h_{M_2}$.
		\item If $f\colon B \to A$ is a morphism of abelian varieties and $M$ is a line bundle on~$A$, then $\hat h_{f^*M}(\alpha)=\hat h_M(f(\alpha))$ for all $\alpha\in\Aug(A/\kbar)$.
		\item If~$M_0$ and~$M$ are line bundles on~$A$ with~$M$ symmetric and ample and~$M_0\in\Pic^0(A)$, then there is a constant~$C>0$ such that
		\[
		|\hat h_{M_0}(\alpha)| \leq C\cdot\hat h_M(\alpha)^{1/2}
		\]
		for all~$\alpha\in\Aug(A/\kbar)$.
	\end{enumerate}
	\begin{proof}
		These are all well-known facts about classical N\'eron--Tate heights, except perhaps the last one, which is \cite[Theorem~B.5.9]{hindry-silverman:diophantine_geometry} over number fields, and \cite[Corollary~9.3.8]{bombieri-gubler:heights} in general.
	\end{proof}
\end{proposition}

\subsubsection{Local N\'eron--Tate heights}

Next, we develop the corresponding theory for local N\'eron--Tate heights. We take the perspective on local heights developed in \cite[\S2.7]{bombieri-gubler:heights}. For this, suppose that~$k_v$ is a field equipped with an absolute value~$|\cdot|_v\colon k_v \to \bR$ (e.g., the completion of an algebraic extension of a global field~$k$ at a place~$v$). Let~$X/k_v$ be a smooth algebraic variety and~$L$ a line bundle on~$X$. Recall that a \emph{$v$-adic metric} $||\cdot||_{L,v}$ on~$L$ is a choice of $v$-adic metric on the $\kbar_v$-vector space $M_x$ for each~$x\in Y(\kbar_v)$ \cite[Definition~2.7.1]{bombieri-gubler:heights}. Here, being a $v$-adic metric means that $||\cdot||_{L,v}$ is a vector space norm on $L_x$ satisfying
\[
||\kappa\tilde x||_{L,v} = |\kappa|_v\cdot||\tilde x||_{L,v}
\]
for all~$\kappa\in\kbar_v$ and $\tilde x\in L_x$. Tensor products and duals of $v$-adically metrised line bundles inherit $v$-adic metrics in a canonical way, and the unit vector bundle~$L=\cO_X$ admits a canonical ``trivial'' metric via the identification $L_x=\kbar_v$ for every~$x\in X(\kbar_v)$. We often view a $v$-adic metric on~$L$ as a function
\[
||\cdot||_{L,v}\colon L(\kbar_v) \to \bR
\]
where~$L(\kbar_v)$ denotes the $\kbar_v$-points of the total space of~$L$. The metrics we consider will all satisfy the following boundedness condition.

\begin{definition}[{\cite[Definitions~2.6.2, 2.6.8 \&~2.7.1]{bombieri-gubler:heights}}]
	Let~$Y/k_v$ be an algebraic variety. A subset~$E\subseteq Y(\kbar_v)$ is called \emph{bounded} just when it satisfies either of the following equivalent conditions:
	\begin{itemize}
		\item there exists a finite covering~$(U_i)_{i\in I}$ of~$Y$ by open affines, closed embeddings $U_i\hookrightarrow\bA^{n_i}_{\kbar_v}$ and a decomposition~$E=\bigcup_{i\in I}E_i$ such that $E_i\subseteq U_i(\kbar_v)\subseteq \kbar_v^{n_i}$ is bounded in the $v$-adic norm for all~$i\in I$;
		\item for every finite covering~$(U_i)_{i\in I}$ of~$Y$ by open affines and closed embeddings $U_i\hookrightarrow\bA^{n_i}_{\kbar_v}$, there exists a decomposition $E=\bigcup_{i\in I}E_i$ into subsets~$E_i\subseteq U_i(\kbar_v) \subseteq \kbar_v^{n_i}$ which are bounded in the $v$-adic norm for all~$i\in I$.
	\end{itemize}
	
	A function~$f\colon Y(\kbar_v) \to \bR$ is called \emph{locally bounded} just when~$f(E)\subseteq\bR$ is bounded for all bounded~$E\subseteq Y(\kbar_v)$.
	
	If~$L$ is a line bundle on a smooth variety~$X/k_v$, then we write~$L^\times$ for the total space of~$L$ with the zero section excised. We say that a metric~$||\cdot||_{L,v}$ on~$L$ is \emph{locally bounded} just when the function
	\[
	-\log||\cdot||_{L,v} \colon L^\times(\kbar_v) \to \bR
	\]
	is locally bounded in the above sense.
\end{definition}

\begin{remark}\label{rmk:boundedness_rmks}\leavevmode
	\begin{itemize}
		\item If~$Y/\kbar_v$ is proper, then all subsets of~$Y(\kbar_v)$ are bounded \cite[Proposition~2.6.6]{bombieri-gubler:heights}.
		\item If~$f\colon Y\to Y'$ is a morphism of $\kbar_v$-varieties, then the image of any bounded subset of~$Y(\kbar_v)$ is bounded in~$Y'(\kbar_v)$.
	\end{itemize}
\end{remark}

\begin{example}[cf.\ {\cite[Example~2.7.20]{bombieri-gubler:heights}}]\label{ex:model_metrics}
	Suppose that the field $k_v$ is non-archimedean, with ring of integers~$\cO_v$. Let $\cX/\cO_v$ be a flat proper normal scheme with smooth generic fibre~$X/k_v$. Let~$\cL$ be a line bundle on~$\cX$ with generic fibre~$L$. Then there is a unique metric~$||\cdot||_{\cL,v}$ on~$L$, which called the \emph{model metric}, characterised by the fact that $||\tilde x||_{\cL,v}=1$ for~$\tilde x\in \cL^\times(\Obar_v)\subseteq L^\times(\kbar_v)$. This metric is locally bounded.
\end{example}

In the case that~$X=A$ is an abelian variety, there is a canonical choice of locally bounded $v$-adic metric on any line bundle~$M/A$, up to scaling by a constant \cite[Theorem~9.5.4]{bombieri-gubler:heights}. To make the theory completely canonical, we define a \emph{rigidification} of~$M$ to be a choice of non-zero point $\tilde 0\in M_0$ in the fibre of~$M$ over~$0\in A(\kbar_v)$, see \cite[\S9.5.6]{bombieri-gubler:heights}, and normalise the metric so that~$||\tilde0||_{M,v}=1$. There is an evident notion of the tensor product of two rigidified line bundles, and the trivial line bundle $M=\cO_A$ admits a canonical rigidification by the unit $\tilde0=1\in M_0=\kbar_v$. Moreover, given a morphism $f\colon A\to B$ of abelian varieties and a rigidified line bundle~$M$ on~$B$, there is a canonical rigidification of the pullback~$f^*M$.

\begin{theorem}[{\cite[Theorem~9.5.7]{bombieri-gubler:heights}}]\label{thm:local_metrics}
	For every rigidified line bundle~$M=(M,\tilde 0)$ on an abelian variety~$A/k_v$, there is a canonical $v$-adic metric
	\[
	||\cdot||_{M,v}\colon M(\kbar_v) \to \bR \,,
	\]
	satisfying the following properties:
	\begin{enumerate}[label = (\alph*), ref = (\alph*), font = \upshape]
		\item\label{condn:metric_well-defined} $||\cdot||_{M,v}$ only depends on~$M$ up to isomorphism of rigidified line bundles.
		\item\label{condn:locally_bounded} $||\cdot||_{M,v}$ is locally bounded.
		\item\label{condn:metric_tensors} For any two rigidified line bundles $M_1$, $M_2$ we have
		\[
		||\cdot||_{M_1\otimes M_2,v} = ||\cdot||_{M_1,v}\otimes||\cdot||_{M_2,v} \,.
		\]
		\item\label{condn:metric_trivial} For~$M=\cO_A$ with its canonical rigidification, the metric $||\cdot||_{\cO_A,v}$ is the trivial metric.
		\item\label{condn:metric_pullback} For any morphism $f\colon A\to B$ of abelian varieties over~$k_v$ and any rigidified line bundle~$M$ on~$B$ we have
		\[
		||\cdot||_{f^*M,v} = f^*||\cdot||_{M,v} \,.
		\]
	\end{enumerate}
	Moreover, these canonical metrics are uniquely characterised by conditions \ref{condn:metric_well-defined}--\ref{condn:metric_tensors} and condition~\ref{condn:metric_pullback} for~$f=[2]$ the doubling map on~$A$ (see \cite[Remark~9.5.8]{bombieri-gubler:heights}).
\end{theorem}

This canonical metric can be described explicitly in the case of good reduction. Suppose that~$\cA/\cO_v$ is an abelian scheme with generic fiber~$A/k_v$. An elementary calculation with divisors shows that any rigidified line bundle~$M$ on~$A$ extends essentially uniquely to a rigidified line bundle~$\cM$ on~$\cA$, meaning a line bundle~$\cM$ endowed with a point~$\tilde 0\in\cM(\cO_v)$ lying over~$0\in\cA(\cO_v)$ and disjoint from the zero-section of~$\cM$.

\begin{lemma}\label{lem:good_local_metrics}
	With notation as above, the canonical metric $||\cdot||_{M,v}$ is the model metric of Example~\ref{ex:model_metrics} coming from~$\cM$.
	\begin{proof}
			It is easy to check that the assignment $M\mapsto\cM$ from rigidified line bundles on~$A$ to rigidified line bundles on~$\cA$ is compatible with tensor products and pullbacks along the doubling map~$[2]\colon\cA\to\cA$. So the model metrics $||\cdot||_{\cM,v}$ are compatible with tensor products and pullback along~$[2]$, so we are done by the unicity clause of Theorem~\ref{thm:local_metrics}.
		\end{proof}
\end{lemma}

\begin{remark}
	If~$A$ is not the generic fiber of an abelian scheme, it need not be the case that the canonical metric $||\cdot||_{M,v}$ is the model metric for any model of~$M$.
\end{remark}

\begin{definition}
	If~$M=(M,\tilde0)$ is a rigidified line bundle on an abelian variety~$A/k_v$, we define the \emph{local N\'eron--Tate height} with respect to~$M$ to be the locally bounded function
	\[
	\lambda_{M,v} = -\log||\cdot||_{M,v} \colon M^\times(\kbar_v) \to \bR \,.
	\]
\end{definition}

We reinterpret the local N\'eron--Tate height as a function on a motivic augmentation space using the following lemma.

\begin{lemma}
	Let~$M=(M,\tilde0)$ be a rigidified line bundle on an abelian variety~$A/k_v$. Then the local N\'eron--Tate height with respect to~$M$ factors uniquely through a function
	\[
	\lambda_{M,v}\colon \Aug(M^\times_{\kbar_v}/\kbar_v) \to \bR \,.
	\]
\end{lemma}
\begin{proof}
	By Proposition~\ref{w24d}, we know that $\Aug(M^\times_{\kbar_v}/\kbar_v)$ is a torsor under~$\kbar_v^\times\otimes\bQ$ over~$A(\kbar_v)\otimes\bQ$, so the augmentation map
	\[
	M^\times(\kbar_v) \to \Aug(M^\times_{\kbar_v}/\kbar_v)
	\]
	is surjective. Thus, to prove the result, it suffices to show that if~$\tilde x_1,\tilde x_2\in M^\times(\kbar_v)$ have the same image under the augmentation map, then we have $\lambda_{M,v}(\tilde x_1)=\lambda_{M,v}(\tilde x_2)$.
	
	For this, let~$x_1$ and~$x_2$ denote the images of~$\tilde x_1$ and~$\tilde x_2$ in~$A(\kbar_v)$. Since~$x_1$ and~$x_2$ have the same image in~$\Aug(A_{\kbar_v}/\kbar_v)=A(\kbar_v)\otimes\bQ$, we deduce that~$x_1-x_2\in A(\kbar_v)_\tors$. Choose a positive integer $n\in\bN$ such that $n\cdot(x_1-x_2)=0$, and let~$f\colon A\to A$ denote the multiplication-by-$n$ map. If $M_n$ denotes the rigifidied line bundle $M^{\otimes(n+1)}\otimes[-1]^*M^{\otimes(1-n)}$ on~$A$, then $f^*M=M^{\otimes2n^2}$ by the theorem of the cube. So the morphism $f\colon A\to A$ lifts to a morphism $f\colon (M^{\otimes2n^2})^\times \to M_n^\times$.
	
	Now~$\tilde x_1^{\otimes2n^2}$ and~$\tilde x_2^{\otimes2n^2}$ have the same image in $\Aug((M^{\otimes2n^2})^\times_{\kbar_v}/\kbar_v)$ by Lemma~\ref{lem:pushout_product_of_torsors}. So $f(\tilde x_1^{\otimes2n^2})$ and $f(\tilde x_2^{\otimes2n^2})$ have the same image in $\Aug(M_{n,\kbar_v}^\times/\kbar_v)$, and also have the same image in~$A(\kbar_v)$, namely $nx_1=nx_2$. From Proposition~\ref{w24d}, we deduce that
	\[
	f(\tilde x_1^{\otimes 2n^2}) = \zeta\cdot f(\tilde x_2^{\otimes 2n^2})
	\]
	for some root of unity~$\zeta\in\mu_\infty(\kbar_v)$. It then follows from the properties of local N\'eron--Tate heights that
	\begin{align*}
	\lambda_{M,v}(\tilde x_1) &= \frac1{2n^2}\lambda_{M^{\otimes2n^2},v}(\tilde x_1^{\otimes2n^2}) = \frac1{2n^2}\lambda_{M_n,v}(f(\tilde x_1^{\otimes2n^2})) \\
	 &= \frac1{2n^2}\lambda_{M_n,v}(f(\tilde x_2^{\otimes2n^2})) = \frac1{2n^2}\lambda_{M^{\otimes2n^2},v}(\tilde x_2^{\otimes2n^2}) = \lambda_{M,v}(\tilde x_2) \,,
	\end{align*}
	which is what we wanted to show.
\end{proof}

\subsubsection{Decomposition into local heights}

Let us now return to the setting that~$k$ is a global field, and explain how the global N\'eron--Tate height decomposes into local N\'eron--Tate heights. We let~$\Mm(k)$ denote the set of places of~$k$, and assume that the absolute value~$|\cdot|_v$ on each completion~$k_v$ is normalised so that the product formula
\[
\prod_{v\in\Mm(k)}|\kappa|_v = 1
\]
holds for all~$\kappa\in k^\times$. (In particular, $|\kappa|_v=1$ for all but finitely many~$v$.) Let~$\Mm(\kbar)$ denote the set of places of~$\kbar$. We have
\[
\Mm(\kbar) = \varprojlim_{k'}\Mm(k') \,,
\]
where the inverse limit is taken over finite extensions~$k'$ of~$k$ inside~$\kbar$. There is a natural measure on~$\Mm(\kbar)$, where for every place~$v'\in\Mm(k')$ lying over a place~$v\in\Mm(k)$, the set of places of~$\kbar$ extending~$v'$ has measure $\frac{[k'_{v'}:k_v]}{[k':k]}$. With respect to this measure, the elements of~$\kbar^\times$ satisfy the product formula
\[
-\int_{v\in\Mm(\kbar)}\log|\kappa|_v = 0 \,,
\]
where we normalise absolute values on~$\kbar$ so that they extend one of the given absolute values on~$k$ \cite[Proposition~1.4.2]{bombieri-gubler:heights}.

Now suppose that~$A/\kbar$ is an abelian variety and~$M=(M,\tilde0)$ is a rigidified line bundle on~$A$. Any point~$x\in A(\kbar)$ lifts to a point~$\tilde x\in M^\times(\kbar)$, and we have
\[
\hat h_M(x) = \int_{v\in\Mm(\kbar)}\lambda_{M,v}(\tilde x)
\]
by \cite[Corollary~9.5.14]{bombieri-gubler:heights}. This local decomposition of the N\'eron--Tate height extends automatically to augmentations.

\begin{proposition}\label{prop:local_decomposition}
	Let~$k$ be a global field, $A$ an abelian variety over~$\kbar$ and~$M=(M,\tilde0)$ a rigidified line bundle on~$A$. Then any augmentation~$\alpha\in\Aug(A/\kbar)$ lifts to an augmentation~$\tilde\alpha\in\Aug(M^\times/\kbar)$, and we have
	\[
	\hat h_M(\alpha) = \int_{v\in\Mm(\kbar)}\lambda_{M,v}(\tilde\alpha) \,.
	\]
\end{proposition}
\begin{proof}
	Proposition~\ref{w24d} implies that~$\alpha$ lifts. It also implies that~$\tilde\alpha$ is the augmentation attached to a point~$\tilde x\in M^\times(\kbar)$, and the result follows.
\end{proof}

\subsection{Weil heights on curves}

Now we wish to pull back the theory of motivic N\'eron--Tate heights to curves. We continue to work over a global field~$k$, and introduce the notion of a \emph{locally geometric} augmentation.

\begin{definition}\label{def:locally_geometric}
	Let~$k$ be a field which is an algebraic extension of a global field, and let~$X/k$ be a smooth variety. For each place~$v$ of~$k$, there is a natural map
	\[
	\Aug(X/k) \to \Aug(X_{k_v}/k_v) \,.
	\]
	We say that an augmentation $\alpha\in\Aug(X/k)$ is \emph{locally geometric} just when its image in~$\Aug(X_{k_v}/k_v)$ is the augmentation associated to a $k_v$-rational point $x_v\in X(k_v)$ for all places~$v$ of~$k$. We write~$\Aug(X/k)^\lgeom$ for the set of locally geometric augmentations.
\end{definition}

\begin{theorem}[Rational motivic Weil height machine]\label{thm:weil_heights}
	Let~$k$ be a global field and $X/\kbar$ a smooth projective curve of genus~$\geq1$. Then there is a function
	\[
	h\colon \Pic(X) \to \frac{\{\text{functions $\Aug(X/\kbar)^\lgeom\to\bR$}\}}{\{\text{bounded functions}\}}
	\]
	satisfying the following properties:
	\begin{enumerate}
		\item for line bundles $L_1,L_2\in \Pic(X)$ we have $h_{L_1\otimes L_2} = h_{L_1}+h_{L_2}+O(1)$;
		\item for a morphism $f\colon X\to Y$ of smooth projective curves of genus~$\geq1$ and a line bundle~$L$ on~$Y$ we have $h_{f^*L}(\alpha)=h_L(f(\alpha))+O(1)$; and
		\item for a morphism $f\colon X\to A$ from~$X$ to an abelian variety and a line bundle~$M$ on~$A$ we have $h_{f^*M}(\alpha)=\hat h_M(f(\alpha))+O(1)$.
	\end{enumerate}
	Here, $O(1)$ is shorthand for a bounded function on $\Aug(X/\kbar)^\lgeom$.
\end{theorem}

Before we start, and to motivate the construction, we remark that the first and third conditions above determine the Weil height~$h_L$ uniquely. This is thanks to the following lemma.

\begin{lemma}\label{lem:presentations_of_line_bundles}
	Let~$X/\kbar$ be a smooth projective curve of genus~$\geq1$. Then for every line bundle~$L$ on~$X$ there exists an abelian variety~$A/\kbar$, a morphism $f\colon X\to A$, a line bundle~$M$ on~$A$ and a positive integer~$m$ such that $f^*M$ is isomorphic to $L^{\otimes m}$.
	\begin{proof}
		Let~$A=J$ be the Jacobian of~$A$, and~$f\colon X \to J$ the Abel--Jacobi embedding with respect to some basepoint $x_0\in X(\kbar)$. The pullback map
		\[
		f^*\colon \Pic(J) \to \Pic(X)
		\]
		is an isomorphism on~$\Pic^0$, and pulls back any ample line bundle on~$J$ to an ample line bundle on~$X$. So~$f^*$ has finite cokernel and we are done.
	\end{proof}
\end{lemma}

This implies that there is only one possibility for the Weil height associated to~$L$, namely
\[
h_L(\alpha) = \frac1m\hat h_M(f(\alpha)) + O(1)
\]
for any $(A,f,M,m)$ as in Lemma~\ref{lem:presentations_of_line_bundles}. There is something to be proved here, namely that this height function is independent of the choice of~$(A,f,M,m)$ up to bounded functions.

\begin{lemma}\label{lem:trivial_pullback_implies_bounded_height}
	Let~$f\colon X \to A$ be a morphism from a smooth projective curve~$X/\kbar$ of genus~$\geq1$ to an abelian variety~$A/\kbar$. Suppose that~$M$ is a line bundle on~$A$ such that $f^*M\simeq\cO_X$ is the trivial line bundle on~$X$. Then the function
	\[
	\alpha \mapsto \hat h_M(f(\alpha))
	\]
	is bounded on $\Aug(X/\kbar)^\lgeom$.
	\begin{proof}
		Let~$M^\times = M\smallsetminus 0$ be the complement of the zero section in the total space of~$M$ as usual, and choose a rigidification~$\tilde0$ of~$M$. The assumption that $f^*M\simeq\cO_X$ is trivial implies that $f$ lifts to a morphism $\tilde f\colon X\to M^\times$. We may then spread out~$\tilde f$: there exists a finite extension~$k'/k$ inside~$\kbar$ and a finite set~$S$ of places of~$k'$ such that the curve~$X$ is the geometric generic fiber of a smooth proper curve~$\cX$ over~$\cO_{k',S}$, the abelian variety~$A$ is the geometric generic fiber of an abelian scheme~$\cA$ over~$\cO_{k',S}$, the rigidified line bundle~$M$ is the geometric generic fiber of a rigidified line bundle~$\cM$ on~$\cA$, and the morphism $\tilde f\colon X \to M^\times$ is the geometric generic fiber of a morphism $\tilde f\colon \cX \to \cM^\times$ over~$\cO_{k',S}$.
		
		For any place~$v$ of~$\kbar$, we know that $\tilde f(X(\kbar_v))$ is a bounded subset of $M^\times(\kbar_v)$ by Remark~\ref{rmk:boundedness_rmks}. Since the local N\'eron--Tate height is the negative logarithm of a locally bounded metric, this implies that
		\[
		B_v \coloneqq \sup_{x\in X(\kbar_v)}|\lambda_{M,v}(\tilde f(x))| < \infty \,.
		\]
		Moreover, if~$v$ and~$v'$ are two places of~$\kbar$ lying over the same place~$v_0$ of~$k'$, then the completions $\kbar_v$ and $\kbar_{v'}$ are isomorphic as valued field extensions of~$k'_{v_0}$, so we have~$B_v=B_{v'}$. And if~$v_0\notin S$, then $\tilde f(X(\kbar_v)) = \tilde f(\cX(\Obar_v)) \subseteq \cM^\times(\Obar_v)$, so we have~$B_v=0$ by Lemma~\ref{lem:good_local_metrics}. Hence we have
		\[
		\int_{v\in\Mm(\kbar)}B_v < \infty \,.
		\]
		
		Now suppose that~$\alpha\in\Aug(X_\kbar/\kbar)^\lgeom$ is a locally geometric augmentation, whose image in~$\Aug(X_{\kbar_v}/\kbar_v)$ is the augmentation associated to a point~$x_v\in X(\kbar_v)$ for all places~$v\in\Mm(\kbar)$. We have $|\lambda_{M,v}(\tilde f(\alpha))| = |\lambda_{M,v}(\tilde f(x_v))| \leq B_v$ for all places~$v$, and so
		\[
		|\hat h_M(f(\alpha))| = \Bigl|\int_{v\in\Mm(\kbar)}\lambda_{M,v}(\tilde f(\alpha))\Bigr| \leq \int_{v\in\Mm(\kbar)}B_v < \infty
		\]
		by Proposition~\ref{prop:local_decomposition}. So~$\alpha\mapsto\hat h_M(f(\alpha))$ is a bounded function, as claimed.
	\end{proof}
\end{lemma}

\begin{proof}[Proof of Theorem~\ref{thm:weil_heights}]
	For a line bundle~$L$ on~$X$, we choose a quadruple $(A,f,M,m)$ as in Lemma~\ref{lem:presentations_of_line_bundles} and define the Weil height function $h_L\colon\Aug(X_\kbar/\kbar)^\lgeom\to\bR$ by
	\[
	h_L(\alpha) \colonequals \frac1m\hat h_M(f(\alpha)) \,.
	\]
	We claim that the function $h_L$ is independent of the choice of $(A,f,M,m)$ up to bounded functions. For this, suppose that $(A_1,f_1,M_1,m_1)$ and $(A_2,f_2,M_2,m_2)$ are two quadruples such that $f_i^*M_i\simeq L^{\otimes m_i}$ for~$i=1,2$. Consider the abelian variety $A_3=A_1\times A_2$ and the map $f_3=(f_1,f_2)\colon X\to A_3$. Let~$\pr_i\colon A_3\to A_i$ denote the $i$th projection for~$i=1,2$ and define a line bundle~$M_3$ on~$A$ by
	\[
	M_3 \colonequals \pr_1^*M_1^{\otimes m_2}\otimes \pr_2^*M_2^{\otimes(-m_1)} \,.
	\]
	It follows from Proposition~\ref{prop:neron-tate_heights} that we have
	\begin{align*}
		\frac1{m_1}\hat h_{M_1}(f_1(\alpha)) - \frac1{m_2}\hat h_{M_2}(f_2(\alpha)) &= \frac1{m_1}\hat h_{\pr_1^*M_1}(f_3(\alpha)) - \frac1{m_2}\hat h_{\pr_2^*M_2}(f_3(\alpha)) \\
		&= \frac1{m_1m_2}\hat h_{M_3}(f_3(\alpha))
	\end{align*}
	for all~$\alpha\in \Aug(X/\kbar)^\lgeom$. But by construction $f_3^*M_3\simeq \cO_X$ is trivial, and so by Lemma~\ref{lem:trivial_pullback_implies_bounded_height} the final term above is bounded. Hence the function~$h_L$ defined above is independent of the choice of quadruple~$(A,f,M,m)$ up to bounded functions.
	
	The claimed properties of the heights $h_L$ are then easy to verify.
\end{proof}

We note the following asymptotics for these motivic Weil heights. In the statement of this lemma and beyond, we adopt the usual convention that $O(\phi(\alpha))$ is a shorthand for some unspecified function $g\colon\Aug(X/\kbar)^\lgeom\to\bR$ such that there exists a constant~$C>0$ with $|g(\alpha)|\leq C\cdot\phi(\alpha)$ for all~$\alpha\in\Aug(X/\kbar)^\lgeom$.

\begin{lemma}\label{lem:degree_bound}
	Let~$k$ be a global field, $X$ a smooth projective curve of genus~$\geq1$ over~$\kbar$, and let~$L$ and~$L_0$ be line bundles on~$X$ with~$L$ ample. Then
	\[
	h_{L_0}(\alpha) = \frac{\deg(L_0)}{\deg(L)}h_L(\alpha)+O(1+|h_L(\alpha)|^{1/2})
	\]
	for~$\alpha\in\Aug(X/\kbar)^\lgeom$.
	\begin{proof}
		It suffices to prove the result upon replacing~$L_0$ by $L_0^{\otimes a}\otimes L^{\otimes b}$ for any integers~$a,b$ with $a\neq0$. In particular we are free to assume that~$\deg(L_0)=0$.
		
		Let~$J$ denote the Jacobian of~$X$, and let~$f\colon X\to J$ be the $\kbar$-morphism induced by a divisor of positive degree on~$X$. Replacing~$L$ and~$L_0$ by further tensor powers if necessary, we may assume that~$L_0=f^*M_0$ and~$L=f^*(M_1\otimes M_2)$, for line bundles~$M_0$, $M_1$ and~$M_2$ on~$J$ with~$M_0,M_1\in\Pic^0(J)$ and~$M_2$ symmetric ample. Thus we have $h_{L_0}(\alpha)=\hat h_{M_0}(f(\alpha))+O(1)$ and $h_L(\alpha)=\hat h_{M_1}(f(\alpha))+\hat h_{M_2}(f(\alpha))+O(1)$, so the result follows from Proposition~\ref{prop:neron-tate_heights}.
	\end{proof}
\end{lemma}

\subsection{Northcott's Theorem for motivic heights}

One key property of Weil heights in the classical theory is Northcott's Theorem, which says that a set of points on a quasi-projective variety~$X$ of bounded height and bounded degree is finite. At this stage, we do not know whether the heights we construct in Theorem~\ref{thm:weil_heights} satisfy the corresponding property, however, we can prove something which at least gets close to the expected finiteness.

We now restrict ourselves to the case that~$X$ is defined over a number field~$k$. Let us say that a subset $\Sigma\subseteq\Aug(X_{\kbar}/\kbar)^\lgeom$ is \emph{bounded in height} (with respect to an ample line bundle~$L$ on $X_{\kbar}$) just when there is some $B\in\bR$ such that $h_L(\alpha)\leq B$ for all $\alpha\in\Sigma$. Note that the property of being bounded in height is independent of the choice of height function~$h_L$ associated to~$L$, and is also independent of the choice of ample line bundle~$L$ by Lemma~\ref{lem:degree_bound}. We say that~$\Sigma$ is \emph{bounded in degree} just when there is some~$d\in\bN$ such that every~$\alpha\in\Sigma$ is the image of an element of~$\Aug(X_{k'}/k')^\lgeom$ for an extension~$k'/k$ of degree~$\leq d$.

\begin{theorem}[Motivic Northcott Property]\label{thm:northcott_abelian}
	Suppose that~$X$ is a curve of genus~$\geq2$ over a number field~$k$. Let $\Sigma\subseteq\Aug(X_{\kbar}/\kbar)^\lgeom$ be a subset which is bounded in height and degree. Then for any morphism $f\colon X_{\kbar}\to A$ from $X$ to an abelian variety~$A$ defined over~$\kbar$, the image
	\[
	f(\Sigma) \subseteq \Aug(A/\kbar) = A(\kbar)/A(\kbar)_\tors
	\]
	is finite.
\end{theorem}

The idea in the proof is the following. The image~$f(\Sigma)\subseteq\Aug(A/\kbar)$ is certainly bounded in height with respect to the N\'eron--Tate height~$\hat h_M$ for any ample line bundle~$M$ on~$A$. So, if we could show that~$f(\Sigma)$ were contained in
\[
\bigcup_{[k':k]\leq d}A(k')/A(k')_\tors \,,
\]
then the result would reduce to the usual Northcott property. It turns out that $f(\Sigma)$ is not obviously contained in this union, but we do have a containment ``up to scaling'' as follows.

\begin{lemma}\label{lem:contained_in_lattice}
	Let~$A/k$ be an abelian variety over a number field~$k$, and let~$d$ be a positive integer. Then there is a positive integer~$n = n(k,A,d)$ with the following property. For every finite extension~$k'/k$ of degree~$\leq d$ and every non-constant morphism $f\colon X\to A_{k'}$ from a smooth projective curve~$X/k'$ to~$A_{k'}$ whose image is not an elliptic curve, we have
	\[
	f(\Aug(X_{k'}/k')^\lgeom) \subseteq \frac1n(A(k')/A(k')_\tors)
	\]
	as subsets of $\Aug(A_{k'}/k')=A(k')\otimes\bQ$.
\end{lemma}

The value of~$n$ we need is given by the following auxiliary lemma.

\begin{lemma}\label{lem:kernel_of_localisation}
	For any abelian variety~$A/k$ over a number field~$k$ and every positive integer~$d$ there is a positive integer~$n = n(k,A,d)$ with the following property. For any positive integer~$m$ and finite extension~$k'/k$ of degree~$\leq d$, the kernel of the localisation map
	\begin{equation}\label{eq:restriction_map}\tag{$\ast$}
		\rH^1(G_{k'},A[m])\to\prod_{v\in\Mm(k')}\rH^1(G_v,A[m])
	\end{equation}
	is $n$-torsion.
\end{lemma}

\begin{proof}
	We first deal with the case$~d=1$. By a result of Stoll \cite[Lemma~3.1]{stoll:finite_descent}, there is a positive integer~$n = n(k,A,1)$ such that the continuous cohomology group
	\[
	\rH^1(G_{k_m|k},A[m])
	\]
	is $n$-torsion for all~$m$, where~$k_m=k(A[m])$ is the field generated by the $m$-torsion points of~$A$. We have a commuting diagram
	\small{
	\begin{center}
	\begin{tikzcd}[column sep = small]
		0 \arrow[r] & \rH^1(G_{k_m|k},A[m]) \arrow[r] & \rH^1(G_k,A[m]) \arrow[r]\arrow[d] & \rH^1(G_{k_m},A[m]) \arrow[d] \\
		& & \prod\limits_{v\in\Mm(k)}\rH^1(G_v,A[m]) \arrow[r] & \prod\limits_{w\in\Mm(k_m)}\rH^1(G_{k_m,w},A[m]) \,,
	\end{tikzcd}
	\end{center}
	}
	\noindent in which the top row is exact. The right-hand vertical arrow is injective by the Chebotarev Density Theorem, see for instance the proof of \cite[Proposition~2.2]{saidi:injectivity_localisation}. So the kernel of the localisation map is contained in $\rH^1(G_{k_m|k},A[m])$, and perforce is $n$-torsion.
	
	Before we proceed to the general case, we observe that the proof of \cite[Lemma~3.1]{stoll:finite_descent} gives an explicit value for the constant~$n=n(k,A)$ above. A result of Serre \cite[p.~60]{serre:oeuvres4} implies that there exists a positive integer~$c=c(k,A)$ such that the image of the representation $G_k\to\GL_{2g}(\hat\bZ)$ attached to~$A$ meets the scalars~$\hat\bZ^\times$ in a subgroup containing $\hat\bZ^{\times c}$. We may assume that~$c$ is even, and write~$\rad(c)$ for the product of the distinct prime divisors of~$c$. Then one can take
	\[
	n(k,A) = 4\rad(c)^2c^2\cdot\#A(k)_\tors \,.
	\]
	
	Thus for a general finite extension~$k'/k$, we know that the kernel of~\eqref{eq:restriction_map} is~$n(k',A)$-torsion for
	\[
	n(k',A) = 4\rad(c(k',A))^2c(k',A)^2\cdot\#A(k')_\tors \,,
	\]
	the constant~$c(k',A)$ being defined in the same way as~$c$ above. We want to show that~$n(k',A)$ can be bounded in terms of the degree of~$k'/k$ only. Indeed, if~$[k':k]\leq d$, then we have $c(k',A)\leq d c(k,A)$ by definition. And~$\#A(k')_\tors$ is also bounded uniformly in~$k'$ (this can be proved, for example, by bounding the order of the prime-to-$p$ part of~$A(k')_\tors$ by injecting it into the mod-$v$ points of~$A$ for a $p$-adic place~$v$ of good reduction for~$A$). So~$n(k',A)$ is bounded in terms of $[k':k]$ only, and we are done.
\end{proof}

\begin{proof}[Proof of Lemma~\ref{lem:contained_in_lattice}]
	Let~$n$ be as in Lemma~\ref{lem:kernel_of_localisation}, and let~$k'$, $X/k'$ and~$f$ be as in the statement of Lemma~\ref{lem:contained_in_lattice}.
	
	Suppose that~$\alpha\in\Aug(X_{k'}/k')^\lgeom$. Since~$f(\alpha)\in\Aug(A_{k'}/k')=A(k')\otimes\bQ$, we may pick a positive integer~$m_0$ and a $k'$-rational point~$y_0\in A(k')$ such that $m_0f(\alpha)=y_0$ in $A(k')\otimes\bQ$. Let~$(x_v)_v\in X(\bA_{k'})$ be an adelic point of~$X$ corresponding to~$\alpha$. Then the images of~$m_0f(x_v)$ and~$y_0$ in $\Aug(A_{k'_v}/k'_v)$ are equal for all places~$v\in\Mm(k')$. So we see that $y_0-m_0f(x_v)\in A(k'_v)_\tors$ for all~$v$. But by \cite[Th\'eor\`eme~1]{raynaud:torsion_on_curves}, the intersection of $A(\bar k)_\tors$ with the image of the map $y_0-m_0f\colon X\to A$ is finite. So there is a positive integer~$m_1$, independent of~$v$, such that $y_0-m_0f(x_v)\in A(k'_v)[m_1]$ for all places~$v\in\Mm(k')$.
	
	Now let~$y\colonequals m_1y_0$ and $m\colonequals m_1m_0$, so that $y=mf(x_v)$ for all finite places~$v$. This says that the Kummer class $[\xi_y]\in\rH^1(G_{k'},A[m])$ associated to~$y$ lies in the kernel of the localisation map
	\[
	\rH^1(G_{k'},A[m]) \to \prod_{v\in\Mm(k')}\rH^1(G_v,A[m]) \,.
	\]
	So by Lemma~\ref{lem:kernel_of_localisation}, we find that $n[\xi_y]$ is the trivial cohomology class in~$\rH^1(G_{k'},A[m])$, whence there exists some~$z\in A(k')$ with $ny=mz$. This implies that $f(\alpha)=\frac1my=\frac1nz\in\frac1n(A(k')/A(k')_\tors)$, as claimed.
\end{proof}

\begin{proof}[Proof of Theorem~\ref{thm:northcott_abelian}]
	Enlarging the field~$k$ if necessary, we may assume that~$X(k)$ is non-empty and write~$j\colon X\hookrightarrow J$ for the Abel--Jacobi map. We may, by assumption, choose a positive integer~$d$ such that every element of~$\Sigma$ is the image of an element of $\Aug(X_{k'}/k')^\lgeom$ for a finite extension $k'/k$ of degree~$\leq d$; let~$n=n(k,J,d)$ be as in Lemma~\ref{lem:contained_in_lattice} for this value of~$d$.
	
	Now since any morphism from~$X_\kbar$ to an abelian variety factors through the Abel--Jacobi map, it suffices to prove the first part for~$(A,f)=(J,g)$ where~$g=nj$. We may take the line bundle~$L=g^*M$ to be the pullback of an ample line bundle~$M$ on~$J$, and take the height function~$h_L$ to be
	\[
	h_L(\alpha) = \hat h_M(g(\alpha)) = \hat h_M(nj(\alpha)) \,.
	\]
	Now by Lemma~\ref{lem:contained_in_lattice} we see that~$g(\Sigma)$ is contained in
	\[
	\bigcup_{[k':k]\leq d}A(k')/A(k')_\tors \subset A(\kbar)/A(\kbar)_\tors
	\]
	and it is bounded in height with respect to~$\hat h_M$ by definition. So the classical Northcott Theorem implies that~$g(\Sigma)$ is finite, as desired.
\end{proof}

\begin{theorem}[Motivic Northcott Property, version II]\label{thm:northcott_finite_subscheme}
	Suppose that~$X$ is a curve of genus~$\geq2$ over a number field~$k$. Let $\Sigma\subseteq\Aug(X_{\kbar}/\kbar)^\lgeom$ be a subset which is bounded in height and degree. Then there exists a finite subscheme~$Z\subset X$ defined over~$k$ with the following property: for every non-archimedean place~$v$ of~$\kbar$ and every point~$x_v\in X(\kbar_v)$ whose associated augmentation in $\Aug(X_{\kbar_v}/\kbar_v)$ is the image of an element of~$\Sigma$, we have
	\[
	x_v\in Z(\kbar_v) \,.
	\]
	In particular, the set of such points~$x_v$ is finite for all non-archimedean place~$v$.
\end{theorem}
\begin{proof}
	Let~$J$ be the Jacobian of~$X$ and let~$f\colon X\to J$ be the morphism induced by a divisor of positive degree on~$X$. Note that~$f$ need not be an embedding, but its image in~$J$ is not contained in any translate of a proper abelian subvariety. By our first Northcott Theorem~\ref{thm:northcott_abelian}, the image $f(\Sigma)\subseteq\Aug(J_{\kbar}/\kbar)=J(\kbar)\otimes\bQ$ is finite; the preimage~$W$ of~$f(\Sigma)$ in~$J(\kbar)$ is then a finite disjoint union of cosets of~$J(\kbar)_\tors$. By Manin--Mumford, the inverse image~$f^{-1}(W)$ is a finite subset of $X(\kbar)$. Let~$Z\subset X$ be the reduced finite subscheme such that~$Z(\kbar)$ is the union of the Galois orbits of elements of~$f^{-1}(W)$; we will prove the theorem for this~$Z$.
	
	Fix a non-archimedean place~$v$ and consider the commuting diagram
	\begin{center}
	\begin{tikzcd}
		 & \Aug(X_{\kbar}/\kbar)^\lgeom \arrow[r,"f_*"]\arrow[d] & \Aug(J_{\kbar}/\kbar) \arrow[d] \\
		X(\kbar_v) \arrow[r] & \Aug(X_{\kbar_v}/\kbar_v) \arrow[r,"f_*"] & \Aug(J_{\kbar_v}/\kbar_v) \,.
	\end{tikzcd}
	\end{center}
	The right-hand vertical map is injective, because it is the tensoring with~$\bQ$ of the injection $J(\kbar)\hookrightarrow J(\kbar_v)$. Thus, if~$x_v\in X(\kbar_v)$ has the property that its associated augmentation is the image of an element of~$\Sigma$, then the image of~$x_v$ in~$\Aug(J_{\kbar_v}/\kbar_v)$ lies in~$f(\Sigma)$ and so~$f(x_v)\in W$. So~$x_v\in Z(\kbar_v)$ as desired.
\end{proof}

\section{Motivic Manin--Dem'janenko}
\label{sec:ManinDemjanenko}

With all of the preceding theory set up, it is now straightforward to prove our motivic Manin--Dem'janenko theorem, exactly paralleling the proof of the classical statement \cite{manin:bounded_torsion}.

\begin{theorem}\label{thm:manin-demjanenko}
	Let~$X$ be a smooth projective curve over a number field~$k$. Suppose that there exists an abelian variety~$A/k$ such that
	\[
	\rk A(X) > 2\rk A(k) \,.
	\]
	Then for any ample line bundle~$L$ on $X$, the image of~$\Aug(X/k)^\lgeom$ in~$\Aug(X_\kbar/\kbar)^\lgeom$ is bounded in height. So, by Northcott, for any morphism $f\colon X_\kbar \to B$ with~$B$ an abelian variety, the image $f(\Aug(X/k)^\lgeom)\subseteq\Aug(B/\kbar)=B(\kbar)\otimes\bQ$ is finite.
\end{theorem}

\begin{remark}
	In Theorem~\ref{thm:manin-demjanenko}, $A(X)$ denotes the set of $k$-morphisms $X\to A$, which is a finitely generated abelian group. This is slightly different from the notation in~\cite{manin:bounded_torsion}, where~$A(X)$ is used to denote the group of $k$-morphisms $X\to A$ taking a fixed $k$-rational point on~$X$ to~$0$.
\end{remark}

\begin{proof}[Proof of Theorem~\ref{thm:manin-demjanenko}]
	Fix an ample line bundle~$L$ on~$X$. The key claim is the following. Suppose that $f_1,\dots,f_r\in A(X)$ are $\bZ$-linearly independent modulo $A(k)$. Then for $\alpha\in\Aug(X_\kbar/\kbar)^\lgeom$ with $h_L(\alpha)\gg0$ the points $f_1(\alpha),\dots,f_r(\alpha)$ are $\bQ$-linearly independent in~$\Aug(A_\kbar/\kbar)=A(\kbar)\otimes\bQ$. To see this, let us fix a symmetric ample line bundle~$M$ on~$A$. The function $A(X)/A(k)\to\bQ$ given by $f\mapsto\deg(f^*M)$ is then a positive-definite quadratic form on $A(X)/A(k)$. The pairing matrix~$P$ of the associated bilinear pairing on the span of~$f_1,\dots,f_r$ has entries
	\[
	P_{ij} = \deg((f_i+f_j)^*M) - \deg(f_i^*M) - \deg(f_j^*M) \,,
	\]
	and positive-definiteness implies that $\det(P)>0$.
	
	For~$\alpha\in\Aug(X_\kbar/\kbar)^\lgeom$, the pairing matrix~$P(\alpha)$ of the bilinear height pairing associated to~$\hat h_M$ on $f_1(\alpha),\dots,f_r(\alpha)$ has entries
	\begin{align*}
		P(\alpha)_{ij} &= \hat h_M(f_i(\alpha)+f_j(\alpha)) - \hat h_M(f_i(\alpha)) - \hat h_M(f_j(\alpha)) \\ &= \frac{P_{ij}}{\deg(L)}h_L(\alpha) + O(1+|h_L(\alpha)|^{1/2})
	\end{align*}
	by Lemma~\ref{lem:degree_bound} (applied to $(f_i+f_j)^*M$, $f_i^*M$ and $f_j^*M$). In particular, we have
	\[
	\det(P(\alpha)) = \frac{\det(P)}{\deg(L)^r}\cdot h_L(\alpha)^r + O(1+|h_L(\alpha)|^{r-1/2})
	\]
	by Lemma~\ref{lem:degree_bound}. Since~$\det(P)>0$, we thus have $\det(P(\alpha))>0$ for~$h_L(\alpha)\gg0$. This implies that $f_1(\alpha),\dots,f_r(\alpha)$ are $\bQ$-linearly independent, proving the claim.
	\smallskip
	
	Now to finish the proof of Theorem~\ref{thm:manin-demjanenko}, suppose that $\rk(A(X)) > 2\rk A(k)$, and choose $f_1,\dots,f_r\in A(X)$ forming a basis of $A(X)/A(k)$. Since $r>\rk A(k)=\dim_\bQ\Aug(A/k)$, we know that $f_1(\alpha),\dots,f_r(\alpha)$ are $\bQ$-linearly dependent in $\Aug(A/k)\subset\Aug(A_\kbar/\kbar)$ for all~$\alpha\in\Aug(X/k)^\lgeom$. But the above claim showed that $f_1(\alpha),\dots,f_r(\alpha)$ are $\bQ$-linearly independent if $h_L(\alpha)\gg0$. The only possibility is that $h_L$ is bounded on the image of $\Aug(X/k)^\lgeom$ in $\Aug(X_\kbar/\kbar)^\lgeom$, as desired.
\end{proof}
\section{Realisations and groupoids}\label{s:realisations}

In these final sections, we explain a modification of the definition of locally geometric augmentations which is closer in spirit to Deligne's conceptualisation of the motivic fundamental group as a system of related realisations, as well as to Kim's Selmer varieties, whose points face ``Selmer conditions'' locally in realisations. This modified definition, which we call \emph{realisationwise locally geometric} augmentations, is weaker than the definition we gave in \ref{def:locally_geometric}, and accordingly the set of realisationwise locally geometric augmentations is \emph{a priori} larger than $\Aug(X/k)^\lgeom$. Nonetheless, as we shall explain, all of the main results of this paper are true for realisationwise locally geometric augmentations as well.

\spar{30524A}
\textbf{Algebraic realisation functors.}
Suppose that~$K$ is a field containing~$\bQ$ and that~$T$ is a $K$-linear Tannakian category. Let~$\Dd_T= \Ind \Dd^b(T)$ be the ind completion of the bounded derived $\infty$-category of~$T$. If
\[
C^*_T\colon \Sm_k^\op\to\Dd_T
\]
is a symmetric monoidal functor obeying certain formal properties which we do not list here, then~$C^*_T$ factors uniquely through a symmetric monoidal functor
\[
R_T\colon \Dd \to \Dd_T \,,
\]
which we call the \emph{realisation functor} corresponding to~$C^*_T$. By symmetric monoidality, the functor $C^*_T$ induces a functor
\[
\Sm_k^\op \to \Cc_T \coloneqq \CAlg(\Dd_T)
\]
which we also denote by~$C^*_T$. It factors through the functor
\[
R_T^a\colon \Cc \to \Cc_T
\]
induced by~$R_T$. There's a symmetric monoidal t-exact fully faithful functor 
\[
\Dd_T = \Ind \Dd^b(T) \to \Dd(\Ind T)
\]
to the unbounded derived category, which we tacitly use below to pass to the latter as needed.

We will be interested in just two different realisation functors. The first is the $l$-adic \'etale realisation functor, for~$l$ a prime number invertible in the base field~$k$. Here, we let~$K=\bQ_l$ and let~$T=T_l$ be the $\QQ_l$-Tannakian category of constructible $l$-adic \'etale sheaves on $\Spec(k)$ (equivalently, continuous finite dimensional $l$-adic $G_k$-modules). Let $\Dd_l = \Ind \Dd^b(T_l)$. By Bhatt--Scholze \cite{ScholzeProet} there's a symmetric monoidal functor 
\[
C^*_\et( \cdot , \QQ_l) \colon
\Sm_k^\op \to \Dd_l \,;
\]
the induced functor $R_l\colon \Dd \to \Dd_l$ is called the \emph{$l$-adic \'etale realisation}.

The second realisation functor we need is the $\bR$-linear Hodge realisation. Let $k = \CC$, let~$K=\bR$, and let $T_\infty$ be the $\bR$-Tannakian category of graded polarizable $\bR$-mixed Hodge structures. Let $\Dd_\infty = \Ind \Dd^b(T_\infty)$. By Drew \cite{drew2018motivic} there's a symmetric monoidal functor
\[
C^*_H( \cdot, \RR)\colon \Sm_\bC^\op \to \Dd(T_\infty) \subset \Dd_\infty \,;
\]
the induced functor $R_\infty\coloneqq \Dd \to \Dd_\infty$ is called the \emph{Hodge realisation}. More generally, if~$k$ is a field with an embedding~$\si\colon k\hookrightarrow\bC$, then we define the \emph{Hodge realisation relative to~$\si$} to be the composition
\[
R_{\infty,\si}\colon \Dd \to \Dd(\Spec(\bC)) \xrightarrow{R_\infty} \Dd_\infty \,.
\]

\spar{QF11} \textbf{Augmentations in a realisation.} Let us return to the general setting of a realisation functor $R_T\colon \Dd \to \Dd_T$ induced by a suitable functor $C^*_T\colon \Sm_k^\op \to \Dd_T$ with~$T$ a $K$-linear Tannakian category. For a smooth variety~$X$, we define the \emph{$T$-augmentation space of $X$ over $k$} by 
\[
\Aaug_T(X/k) := \Hom_{\Cc_T}(C_T^*(X), \one)
\]
and we define the \emph{$T$-augmentation set} by
\[
\Aug_T(X/k) = \pi_0 \Aaug_T(X/k) \,.
\]

The $T$-augmentation space $\cAug_T(X/k)$ is the image of~$X$ under the composition
\[
\Sm_k \xrightarrow{C_T^*} \Cc_T^\op \xrightarrow{\Hom_{\Cc_T}(?,\one)} \Hh \,,
\]
where~$\Hh$ is the \oo-category of homotopy types. Each of these functors is symmetric monoidal for the Cartesian symmetric monoidal structure, and so we find that the assignment 
\[
X \mapsto \Aaug_T(X/k)
\]
belongs to a symmetric monoidal functor
\[
\Aaug_T(?/k)\colon \Sm_k^{\prod} \to \Hh^{\prod} \,.
\]
Applying this functoriality to the one-point variety (which is the unit of the Cartesian monoidal structure on~$\Sm_k$), we have augmentation maps
\[
\ka_T\colon X(k) \to \Aaug_T(X/k)
\]
which form part of a natural transformation 
\[
?(k) \to \Aaug_T(?/k)
\quad
\mbox{of functors}
\quad
\Sm_k \to \Hh. 
\]

Because the functor $C^*_T$ factors through the realisation~$R_T$, it follows that the function
\[
\Aug(X/k) = \pi_0\Hom_{\Cc}(C^*(X),\one) \to \Aug_T(X/k) = \pi_0\Hom_{\Cc_T}(C^*_T(X),\one)
\]
induced by~$R_T^a$ fits into a commuting triangle
\begin{center}
\begin{tikzcd}
	X(k) \arrow[r,"\kappa"]\arrow[rd,"\kappa_T"'] & \Aug(X/k) \arrow[d,"R_T^a"] \\
	 & \Aug_T(X/k) \,.
\end{tikzcd}
\end{center}

When $R=R_l$ is the $l$-adic \'etale realisation, we write $\Aaug^\et_{\QQ_l}(X)$, and when $R=R_\infty$ is the real Hodge realisation, we write $\Aaug_\RR^H(X)$, and similarly for the augmentation \textit{sets}.

\spar{QF12} \textbf{From augmentations to groupoids.}
Let us now assume that our functor $C^*_T\colon \Sm_k^\op \to \Dd_T$ has the following property: for any geometrically connected $X\in\Sm_k$, the algebra $C^*_T(X)\in\Cc_T$ is \emph{cohomologically connected}, meaning that $H^i(C^*_T(X))=0$ for~$i<0$ and $H^0(C^*_T(X))=\one$. This assumption is satisfied for both $l$-adic \'etale cohomology and $\bR$-linear Hodge cohomology. We are going to show that, under this assumption, one can attach to any geometrically connected~$X$ a prounipotent groupoid in the Tannakian category~$T$ (see \cite{Deligne89} for the definition). Similar constructions may be found for instance in \cite[\S2]{dancohen2021rational}.

\spar{300609A}
	We begin by considering the case when~$T$ is the category of vector spaces over the field~$K$. Fix a field $K$ of characteristic $0$. Let $\cdga(K) = \CAlg \Cpx(K)$ denote the category of (unbounded) commutative differential graded algebras. We define a morphism $f: A\to B$ in $\cdga(K)$ to be an \emph{equivalence} if the underlying morphism of complexes is a quasi-isomorphism. We define $f$ to be a \emph{fibration} if $f^i: A^i \to B^i$ is surjective for every $i \in \ZZ$. These definitions determine a structure of a model category on $\cdga(K)$; see e.g. Proposition 2.5 of Olsson \cite{OlssonBar}.
	
	Let $\cdga^{\ge 0}(K)$ denote the category of non-negative cdga's ($A^i = 0$ for $i <0$). By Proposition 5.3 of loc. cit., the same definitions as above determine a structure of model category. We write $\opnm{ho}$ for the homotopy category of a model category and $h$ for the homotopy category of an \oo-category.

\spar{QF14}
	We denote the (unbounded) derived \oo-category of complexes of $K$-\textit{vector spaces} by $\Dd(K)$. By \cite{HinichRectification} we have an equivalence between the \oo-category associated to the category $\cdga(K)$ and $\CAlg \Dd(K)$; in particular, we have an equivalence of homotopy categories
	\[
	\opnm{ho} \cdga(K) 
	\simeq h \CAlg \Dd(K).
	\]

\begin{proposition}
	\label{300609B}
	Suppose $A \in \CAlg \Dd(K)$ is cohomologically connected. We denote the set of augmentations of $A$ by
	\[
	\Aug(A) := 
	\pi_0 \Hom_{\CAlg \Dd(K)}(A, \one).
	\]
	Suppose $A$ admits augmentations 
	$
	A \rightrightarrows \one
	$
	(in particular, $A \neq 0$) and let 
	\[
	B = \one \otimes_A \one
	\]
	be the associated relative tensor product. Then $\Aug(A) = \{\ast\}$ is a singleton. Moreover, we have
	\[
	H^i(B) =0
	\text{ for }
	i<0
	\quad \text{and} \quad
	H^0(B) \neq 0.
	\]
\end{proposition}

\begin{proof}
	By the equivalence $\opnm{ho} \cdga(K) \simeq h \CAlg \Dd(K)$, we may lift $A$ to a cohomologically connected cdga, which we denote again by $A$. Applying the proof of the existence of minimal models, e.g. Theorem 12.1 of \cite{felix2012rational}, we may assume $A$ is a cofibrant object in $\cdga^{\ge 0}(K)$ and connected.\footnote{Traditional accounts of the construction of minimal models, such as this one, restrict attention to $\cdga^{\ge 0}(K)$. However, the same construction goes through in $\cdga(K)$.} Thus, $A$ has only one augmentation and 
	\[
	\Aug(A)=
	\Hom_{\opnm{ho} \cdga} (A, \one)
	\simeq
	\Hom_{\opnm{ho} \cdga^{\ge 0}} (A, \one)
	\simeq
	\{*\}.
	\]
	This establishes the first statement. Returning to a cofibrant model of $A$ in $\cdga(K)$, we may replace the pushout $B$ by the homotopy pushout
	\[
	B = \one \coprod_{A}^{\m{ho}} \one
	\]
	in $\cdga(K)$, where both augmentations are given by the unique augmentation (up to homotopy). According to Corollary 7.6 of Olsson \cite{OlssonBar}, $B$ is cohomologically nonnegative. Consequently, $B$ has the structure of a Hopf algebra in $h\Dd^{\ge 0}(K)$
	\footnote{See for instance \cite[\S2]{dancohen2021rational} for the Hopf algebra structure.}
	and in particular, $H^0(B)$ carries a counit map
	\[
	H^0(B) \to \one
	\]
	which factors the identity map of $\one$. This shows that $H^0(B) \neq 0$, and completes the proof of the proposition.
\end{proof}

\spar{30608B}
We now return to the setting of a general neutral Tannakian category~$T$ over a field~$K$. By a \emph{connected set-based prounipotent $T$-groupoid} $\Gg$ we mean the following. 
\begin{itemize}
	\item
	A set $\Gg = \Oo b(\Gg)$ of \emph{objects} or \emph{base-points}. 
	\item
	For each $a,b \in \Gg$, an affine $T$-scheme ${_bP_a} = \Spec {_b A_a}$ (by which we mean a commutative algebra object ${_bA_a} \in \opnm{Ind} T$ regarded as an object of $(\opnm{Ind} T)^\op$).
	\item
	Unit maps $\one \to {_a P_a}$ and composition maps
	\[
	{_cP_b} \times_K {_b P_a} \to {_c P_a}.
	\]
	These are required to be unital and associative, to make each ${_aP_a}$ into a prounipotent $T$-group, and to make ${_bP_a}$ into a left torsor under $_bP_b$ and a right torsor under $_aP_a$.
\end{itemize}
See Deligne \cite{Deligne89} for detailed definitions of prounipotent groups and their torsors in a Tannakian category. 

\spar{30611A}
Let $T$ be a neutral Tannakian category over $K$. Fix an algebra $A \in \CAlg \Dd (\Ind T)$ cohomologically connected and let $\Aug(A)$ denote its set of augmentations. We construct an associated set-based prounipotent $T$-groupoid. As object-set we take the set $\Aug(A)$ of augmentations. Given augmentations $a,b: A \rightrightarrows \one$, we set
\[
{_bB_a} := \one 
\underset{b, \, A, \, a} \otimes \one.
\]
The constructions of \cite[\S2]{dancohen2021rational} endow the objects ${_b B_a} \in h \Dd (\Ind T)$ with unit, counit, composition and cocomposition maps.\footnote{By \cite{DCHorev, brantner2021pd}, each object $_bB_a$ is further equipped with a structure of comodule for $_aB_a$ and for $_bB_b$ up to coherent higher homotopy. This more refined structure is not needed here.} By Proposition \ref{300609B}, ${_b B_a}$ is cohomologically nonnegative, so we may apply the symmetric monoidal functor 
\[
H^0: \Dd^{\ge 0}_T \to \opnm{Ind} T
\]
to endow the objects $H^0({_b B_a}) \in \opnm{Ind} T$ with similar structures. In particular, $H^0({_b B_a})$ has the structure of a commutative ring object, and we define the affine $T$-scheme associated to the pair of objects $a,b$ by
\[
{_bP_a} = \Spec H^0({_b B_a}).
\]

\begin{proposition}
	\label{30611B}
	Let $T$ be a neutral Tannakian category over $K$ and let $A$ be a cohomologically connected object of $\CAlg \Dd(\opnm{Ind} T)$. The data
	\[
	\Aug(A), \quad \{{_bP_a}\}_{a,b \in \Aug(A)},
	\]
	along with the unit and composition maps indicated above, determine a set-based prounipotent $T$-group. 
\end{proposition}

\begin{proof}
	We fix a $K$-rational fiber functor $\om$. That each ${_aP_a}$ is a prounipotent $T$-group may be checked after applying $\om$, where this follows from the analogous statement concerning cdga's; see, for instance, \cite{Wojtkowiak}. For $b \neq a$, we find that ${_bP_a}$ has a left action of the prounipotent $T$-group ${_aP_a}$. By Proposition \ref{300609B}, there exists a homotopy between the two augmentations
	\[
	\om(a), \om(b): \om(A) \rightrightarrows K.
	\]
	Any homotopy $\om(a) \to \om(b)$ gives rise to an $\om({_aP_a})$-equivariant isomorphism
	\[
	\om({_aP_a}) \simeq \om({_bP_a}). 
	\]
	Thus, $\om({_bP_a})$ is a (trivial) torsor, and it follows that ${_bP_a}$ is a torsor as well.
\end{proof}

\spar{QF15}
We can apply this construction to the two realisations of interest: \'etale and Hodge. In the \'etale setting, we obtain for every prime~$l$ a connected set-based prounipotent groupoid in the category of $\bQ_l$-linear Galois representations, which we denote by ${}_?P_?^l$. If a base point~$x\in X(k)$ is fixed, then we denote by~$U^l$ the group~${}_xP_x^l$. The non-abelian Galois cohomology set~$H^1(G_k,U^l(\bQ_l))$ parametrises right $U^l$-torsors with a compatible action of~$G_k$; we write
\[
j_l\colon X(k) \to H^1(G_k,U^l(\bQ_l))
\]
for the function sending~$y\in X(k)$ to the class of~${}_yP_x^l$.

In the Hodge setting, we obtain a connected set-based prounipotent groupoid in the category of $\bR$-mixed Hodge structures, which we denote by ${}_?P_?^\infty$. If a base point~$x\in X(k)$ is fixed, then we denote by~$U^\infty$ the group~${}_xP_x^\infty$. If~$G_{\MHS,\bR}$ denotes the Tannaka group of the category of $\bR$-mixed Hodge structures, then the non-abelian Galois cohomology set~$H^1(G_{\MHS,\bR},U^\infty)$ parametrises right $U^\infty$-torsors with a compatible mixed Hodge structure; we write
\[
j_\infty\colon X(k) \to H^1(G_{\MHS,\bR},U^\infty)
\]
for the function sending~$y\in X(k)$ to the class of~${}_yP_x^\infty$.

\spar{QF16}
The groupoid ${}_?P^l_?$ we have attached to any geometrically connected $X\in\Sm_k$ turns out to be none other than the $\bQ_l$-pro-unipotent \'etale fundamental groupoid of~$X_\kbar$. Recall that this is defined using the Tannakian formalism: if~$\Loc_{\bQ_l}^\un(X_\kbar)$ denotes the Tannakian category of unipotent $\bQ_l$-local systems on~$X_\kbar$ in the \'etale topology, then any $k$-rational points~$x,y\in X(k)$ determine fibre functors
\[
\omega_x^l,\omega_y^l\colon\Loc_{\bQ_l}^\un(X_\kbar) \to \Vect_{\bQ_l} \,,
\]
and the $\bQ_l$-pro-unipotent \'etale path torsor is the affine $\bQ_l$-scheme representing the functor of tensor-natural isomorphisms $\Iso^\otimes(\omega_x^l,\omega_y^l)$. Because everything in sight is defined over~$k$, the absolute Galois group~$G_k$ acts in a natural way on~$\Iso^\otimes(\omega_x^l,\omega_y^l)$. Thus, the $\bQ_l$-pro-unipotent \'etale fundamental groupoid $\Iso^\otimes(\omega_?^l,\omega_?^l)$ is a connected set-based pro-unipotent groupoid in the category of representations\footnote{Implicit here is the assertion that the action of~$G_k$ on the affine ring of~$\Iso^{\otimes}(\omega_?^l,\omega_?^l)$ is continuous; this is because $\Iso^{\otimes}(\omega_?^l,\omega_?^l)$ is the topological $\bQ_\ell$-Mal\u cev completion of the profinite \'etale path torsor $\pi_1^\et(X_\kbar;?,?)$ (see e.g.~the discussion in~\cite[\S2.4]{BettsThesis}), and the action on the latter is continuous. Alternatively, this continuity is a consequence of Lemma~\ref{lem:realisation_vs_tannakian_etale}.} of~$G_k$.

By Theorem~6.5 of \cite{dancohen2021rational} as corrected by \cite{DCSchlankCorr}, the groupoid ${}_?P_?^l$ we have defined agrees with the Tannakian definition.

\begin{lemma}\label{lem:realisation_vs_tannakian_etale}
	Let~$k$ be a field and let~$X/k$ be a smooth geometrically connected variety. Then there is a $G_k$-equivariant isomorphism
	\[
	{}_yP_x^l \cong \Iso^\otimes(\omega_x^l,\omega_y^l)
	\]
	for all~$x,y\in X(k)$, compatible with composition on either side.
\end{lemma}

\begin{corollary}
	If a base point~$x\in X(k)$ is fixed, then the group~$U^l$ is $G_k$-equivariantly isomorphic to the $\bQ_l$-pro-unipotent \'etale fundamental group of~$(X_\kbar,x)$, and the function $j_l\colon X(k)\to H^1(G_k,U^l(\bQ_l))$ is the non-abelian Kummer map (see~\cite{kimii} for the definition).
\end{corollary}

\spar{QF17}
We now want to prove the analogous compatibility for the Hodge-theoretic realisation. The corresponding Tannakian object will be the $\bQ$-pro-unipotent Betti fundamental groupoid of a connected $X\in\Sm_\bC$, endowed with the mixed Hodge structure constructed by Hain and Zucker. Let~$\Loc_\bQ^\un(X^\an)$ denote the Tannakian category of unipotent $\bQ$-local systems on~$X^\an$, and for points $x,y\in X(\bC)$, let
\[
\omega_x^\infty,\omega_y^\infty\colon \Loc_\bQ^\un(X^\an)\to\Vect_\bQ
\]
denote the corresponding fibre functors. The $\bQ$-pro-unipotent Betti path torsor is, by definition, the affine $\bQ$-scheme representing the functor of tensor-natural isomorphisms $\Iso^\otimes(\omega_x^\infty,\omega_y^\infty)$.

The path torsor $\Iso^\otimes(\omega_x^\infty,\omega_y^\infty)$ comes with a mixed Hodge structure on its affine ring, constructed by Hain and Zucker \cite{hain-zucker:unipotent_variations}. This requires a little explanation to connect the Tannakian definition of path torsors with the more explicit language in \cite{hain-zucker:unipotent_variations}. Let~$\pi_1(X^\an;?,?)$ denote the usual fundamental groupoid of the topological space~$X^\an$. If~$x,y$ are any points, then the free vector space $\bQ[\pi_1(X^\an;x,y)]$ comes with a canonical cocommutative coalgebra structure, whose comultiplication sends any element~$\gamma\in\pi_1(X^\an;x,y)$ to~$\gamma\otimes\gamma$. It also comes with a canonical right module structure under the group ring~$\bQ[\pi_1(X^\an;x)]$. If we let~$I$ denote the augmentation ideal of~$\bQ[\pi_1(X^\an;x)]$ and let~$\Hom_\bQ(-,\bQ)$ denote the $\bQ$-linear dual, then the vector space
\[
{}_yH_x = \varinjlim_n\Hom_\bQ(\bQ[\pi_1(X^\an;x,y)]/I^{n+1},\bQ)
\]
is a commutative $\bQ$-algebra (whose multiplication is dual to the comultiplication on~$\bQ[\pi_1(X^\an;x,y)]$), and is a right comodule under ${}_xH_x$. Hain and Zucker construct an ind-mixed Hodge structure on~${}_yH_x$ for all~$x,y$ \cite[Proposition~3.21(ii)]{hain-zucker:unipotent_variations}, which we can view as an ind-mixed Hodge structure on the affine ring of~$\Iso^\otimes(\omega_x^\infty,\omega_y^\infty)$ via the following lemma.

\begin{lemma}\label{lem:H_is_affine_ring_of_tannakian_paths}
	There is a canonical isomorphism
	\[
	{}_yH_x \cong \cO(\Iso^\otimes(\omega_x^\infty,\omega_y^\infty))
	\]
	of $\bQ$-algebras for all~$x,y\in X(\bC)$, compatible with composition.
\end{lemma}
\begin{proof}
	This is surely well known; we nevertheless include a proof for the reader's convenience. Let us first handle the case that~$x=y$. In this case, the category of unipotent local systems on~$X^\an$ is, via the fibre functor at~$x$, equivalent to the category of unipotent representations of~$\pi_1(X^\an;x)$, i.e.~to the category of finite-dimensional $\bQ[\pi_1(X^\an;x)]$-modules on which some power of~$I$ acts trivially. This category is in turn equivalent to the category of finite-dimensional ${}_xH_x$-comodules, and then to the category of $\Spec({}_xH_x)$-representations. So by Tannakian duality, there is a canonical isomorphism $\Aut^\otimes(\omega_x^\infty)\cong\Spec({}_xH_x)$.
	
	This isomorphism can be described explicitly. Any element~$\gamma\in\pi_1(X^\an;x)$ determines an element of $\Spec({}_xH_x)(\bQ)$, namely the functional ${}_xH_x\to\bQ$ given by evaluation at~$\gamma$. Under the isomorphism $\Aut^\otimes(\omega_x^\infty)\cong\Spec({}_xH_x)$, this element corresponds to the natural isomorphism of~$\omega_x^\infty$ given by path-lifting along~$\gamma$. Since these elements of~$\Spec({}_xH_x)$ are Zariski-dense, this determines the isomorphism uniquely.
	
	Thus, for general~$x,y$, consider the function
	\[
	\alpha_{x,y}\colon \pi_1(X^\an;x,y) \to \Iso^\otimes(\omega_x^\infty,\omega_y^\infty)(\bQ)
	\]
	given by path-lifting. For varying~$x,y$, these functions are compatible with composition of paths/isomorphisms. Pullback along~$\alpha_{x,y}$ then defines a morphism of $\bQ$-algebras
	\[
	\alpha_{x,y}^*\colon\cO(\Iso^\otimes(\omega_x^\infty,\omega_y^\infty)) \to \Hom_\bQ\bigl(\bQ[\pi_1(X^\an;x,y)],\bQ\bigr) \,.
	\]
	We have shown that when~$y=x$, then~$\alpha_{x,x}^*$ is an isomorphism onto~${}_xH_x\subseteq\Hom_\bQ(\bQ[\pi_1(X^\an;x)])$. To show the same in general, let us choose any path~$\gamma\in\pi_1(X^\an;x,y)$. Composition with~$\alpha_{x,y}(\gamma)$ defines an isomorphism $\Aut^\otimes(\omega_x^\infty)\xrightarrow\sim\Iso^\otimes(\omega_x^\infty,\omega_y^\infty)$, while composition with~$\gamma$ defines an isomorphism $\bQ[\pi_1(X^\an;x)]\xrightarrow\sim\bQ[\pi_1(X^\an;x,y)]$. These isomorphisms fit together into a commuting square
	\begin{center}
	\begin{tikzcd}
		\cO(\Aut^\otimes(\omega_x^\infty)) \arrow[r,"\alpha_{x,x}^*"]\arrow[d,equal,"\wr"] & \Hom_\bQ(\bQ[\pi_1(X^\an;x)],\bQ) \arrow[d,equal,"\wr"] \\
		\cO(\Iso^\otimes(\omega_x^\infty,\omega_y^\infty)) \arrow[r,"\alpha_{x,y}^*"] & \Hom_\bQ(\bQ[\pi_1(X^\an;x,y)],\bQ) \,.
	\end{tikzcd}
	\end{center}
	The top arrow is injective with image~${}_xH_x$; this implies that the bottom arrow is also injective with image~${}_yH_x$, so we have constructed the desired isomorphism.
\end{proof}

The compatibility between the groupoid ${}_?P_?^\infty$ we have constructed and the Tannakian construction then reads as follows.

\begin{proposition}\label{lem:realisation_vs_tannakian_Hodge}
\label{prop:realisation_vs_tannakian_Hodge}
	Let~$X/\bC$ be a smooth connected variety. Then there is an isomorphism
	\[
	{}_yP_x^\infty \cong \Iso^\otimes(\omega_x^\infty,\omega_y^\infty)
	\]
	for all~$x,y\in X(\bC)$, compatible with ind-mixed Hodge structures on affine rings, and with the right action of~${}_xP_x^\infty\cong\Iso^\otimes(\omega_x^\infty,\omega_y^\infty)$ on either side.
\end{proposition}

The proof of Proposition \ref{lem:realisation_vs_tannakian_Hodge} is in paragraph \ref{IA33}. This will take rather more effort to prove than its \'etale counterpart, largely because our construction of~${}_yP_x^\infty$ is very abstract while Hain and Zucker's construction of the mixed Hodge structure on~${}_yH_x$ is very explicit and concrete. The key to bridging the two definitions is an abstract characterisation of the Hain--Zucker mixed Hodge structure. For this, fix a point~$x\in X(\bC)$ and let~$\bV$ denote the ind-local system on~$X^\an$ corresponding to~${}_xH_x$ with the left-regular action of~$\pi_1(X^\an;x)$. The fibre of~$\bV$ at any point~$y\in X(\bC)$ is canonically identified with~${}_yH_x$, and the Hain--Zucker mixed Hodge structure on~${}_yH_x$ arises from an admissible variation of ind-mixed Hodge structure on~$\bV$ \cite[Definition~4.21(ii)]{hain-zucker:unipotent_variations}. It turns out that this ind-mixed Hodge structure is unique.

\begin{lemma}\label{lem:unique_MHS}
	There is a unique admissible variation of ind-mixed Hodge structure on the ind-local system~$\bV$ for which evaluation at $1\in\pi_1(X^\an;x)$ is a morphism of ind-mixed Hodge structures ${}_xH_x\to\bQ(0)$.
\end{lemma}
\begin{proof}
	We saw in the proof of Lemma~\ref{lem:H_is_affine_ring_of_tannakian_paths}, that the category of unipotent local systems on~$X^\an$ is equivalent to the category of ${}_xH_x$-comodules which are finite-dimensional as $\bQ$-vector spaces. By \cite[Theorem~1.6]{hain-zucker:unipotent_variations}, the category of admissible unipotent\footnote{We follow here the convention of \cite{hain-zucker:unipotent_variations} that the adjective ``unipotent'' applied to variations of mixed Hodge structure means that the underlying local system is unipotent.} variations of $\bQ$-mixed Hodge structure is equivalent to the category of mixed Hodge structures~$V$ endowed with a ${}_xH_x$-comodule structure for which the comultiplication $V\to {}_xH_x\otimes V$ is a morphism of ind-mixed Hodge structures. (The mixed Hodge structure on~${}_xH_x$ is the Hain--Zucker one.)
	
	The ind-local system~$\bV$ corresponds to the ${}_xH_x$-comodule $V={}_xH_x$ with its left-regular coaction. We need to show that there is a unique ind-mixed Hodge structure on~$V$ for which both the comultiplication $\Delta\colon V\to {}_xH_x\otimes V$ and the counit $\varepsilon\colon V\to\bQ(0)$ are morphisms of mixed Hodge structures. On the one hand, the Hain--Zucker ind-mixed Hodge structure on~$V={}_xH_x$ has this property. On the other, given another such ind-mixed Hodge structure on~$V$, the composition
	\[
	V \xrightarrow{\Delta} {}_xH_x\otimes V \xrightarrow{1\otimes\varepsilon} {}_xH_x
	\]
	is a morphism of ind-mixed Hodge structures. Forgetting mixed Hodge structures, this composition is equal to the identity on~${}_xH_x$, and so we have shown that the identity map~$V\to{}_xH_x$ is a morphism of ind-mixed Hodge structures. It is therefore an isomorphism of ind-mixed Hodge structures \cite[Corollary~3.7]{PetersSteenbrink}, i.e.\ the ind-mixed Hodge structure on~$V$ agrees with the Hain--Zucker mixed Hodge structure on~${}_xH_x$ as desired.
\end{proof}

The ind-local system~$\bV$, moreover, comes with additional algebraic and coalgebraic structures.

\begin{lemma}\label{lem:unique_comodule_algebra}
	There is a unique commutative $\bQ$-algebra structure on~$\bV$ for which the map $\varepsilon\colon\bV_x\to\bQ$ is a homomorphism of $\bQ$-algebras. There is also a unique morphism
	\[
	\Delta\colon \bV \to \bV\otimes\bV_x
	\]
	of ind-local systems such that the fibre of~$\Delta$ at~$x$ makes~$\bV_x$ into a coalgebra with counit~$\varepsilon$. Moreover, the map~$\Delta$ is a morphism of algebras and makes~$\bV$ into a right $\bV_x$-comodule. Finally, both the algebra and comodule structures on~$\bV$ are compatible with the unique variation of mixed Hodge structure afforded by Lemma~\ref{lem:unique_MHS}.
\end{lemma}
\begin{proof}
	Although this can be phrased as a general fact about Tannakian categories, we just prove it in the case of unipotent local systems for concreteness. Throughout the proof, we will conflate local systems with representations of the fundamental group~$\pi\coloneqq\pi_1(X^\an;x)$ as usual.
	
	For the algebra structure, let~$\bV_n$ be the dual of $\bQ[\pi]/I^{n+1}$, so that~$\bV=\varinjlim_n\bV_n$. Let~$\varepsilon_n\colon\bV_n\to\bQ$ be evaluation at~$1\in\pi$. We claim that for each $i+j\leq n$, there exists a unique $\pi$-equivariant map $\mu_{i,j,n}\bV_i\otimes\bV_j\to\bV_n$ such that the square
	\begin{center}
	\begin{tikzcd}
		\bV_i\otimes\bV_j \arrow[r,"\mu_{i,j,n}"]\arrow[d,"\varepsilon_i\otimes\varepsilon_j"] & \bV_n \arrow[d,"\varepsilon_n"] \\
		\bQ\otimes\bQ \arrow[r,equal] & \bQ
	\end{tikzcd}
	\end{center}
	commutes. Dually, there exists a unique $\pi$-equivariant map
	\[
	\mu_{i,j,n}^*\colon \bQ[\pi]/I^{n+1} \to \bigl(\bQ[\pi]/I^{i+1}\bigr)\otimes\bigl(\bQ[\pi]/I^{j+1}\bigr)
	\]
	such that $\mu_{i,j,n}^*(1)=1\otimes1$. Uniqueness of such a map~$\mu_{i,j,n}^*$ is clear because~$1$ generates $\bQ[\pi]/I^{n+1}$ as a $\pi$-representation. For existence, the comultiplication on the Hopf algebra $\bQ[\pi]$ carries the ideal~$I^{n+1}$ into $(I^{i+1}\otimes\bQ[\pi])\oplus(\bQ[\pi]\otimes I^{j+1})$, and so factors through a $\pi$-equivariant map~$\mu_{i,j,n}^*$ as above. This proves the claim.
	
	Taking a colimit first in~$n$, and then in~$i,j$, we find that there is a unique $\pi$-equivariant map $\mu\colon\bV\otimes\bV\to\bV$ such that the square
	\begin{equation}\label{eq:counit_and_multiplication}\tag{$\ast$}
	\begin{tikzcd}
		\bV\otimes\bV \arrow[r,"\mu"]\arrow[d,"\varepsilon\otimes\varepsilon"] & \bV_n \arrow[d,"\varepsilon"] \\
		\bQ\otimes\bQ \arrow[r,equal] & \bQ
	\end{tikzcd}
	\end{equation}
	commutes. This map~$\mu$ is the multiplication for a commutative algebra structure on~$\bV$, because it is induced by the comultiplication on the cocommutative Hopf algebra~$\bQ[\pi]$. Moreover, $\varepsilon$ is a homomorphism of commutative algebras because it is dual to the coalgebra homomorphism $\bQ\to\bQ[\pi]$ sending~$1\in\bQ$ to~$1\in\pi$. This implies that~$\mu$ is the unique $\pi$-equivariant algebra structure on~$\bV$ for which~$\varepsilon$ is a homomorphism, because it is even the unique $\pi$-equivariant map making~\eqref{eq:counit_and_multiplication} commute.
	
	For the comodule structure, we claim that for each $m\leq n$, there exists a unique $\pi$-equivariant map $\Delta_{m,n}\colon \bV_m\to\bV_n\otimes\bV_{n,x}$ such that
	\begin{center}
	\begin{tikzcd}
		\bV_{m,x} \arrow[r,"\Delta_{m,n,x}"]\arrow[rd,hook] & \bV_{n,x}\otimes\bV_{n,x} \arrow[d,"\varepsilon_n\otimes1_{\bV_{n,x}}"] \\
		 & \bV_{n,x}
	\end{tikzcd}
	\end{center}
	commutes. (Note that $\bV_{m,x}$ is the underlying vector space of~$\bV_m$, thought of as a trivial $\pi$-representation.) Dually, there exists a unique $\pi$-equivariant map
	\[
	\Delta_{m,n}^*\colon \bigl(\bQ[\pi]/I^{n+1}\bigr)\otimes\bigl(\bQ[\pi]/I^{n+1}\bigr)^{\sim}\to\bQ[\pi]/I^{m+1}
	\]
	such that $\Delta_{m,n}^*(1\otimes x)=x$ mod~$I^{m+1}$ for all~$x\in\bQ[\pi]/I^{n+1}$, where the superscript~$\sim$ denotes trivial $\pi$-action. The uniqueness of~$\Delta_{m,n}^*$ is clear, since the tensors of the form $1\otimes x$ generate $\bigl(\bQ[\pi]/I^{n+1}\bigr)\otimes\bigl(\bQ[\pi]/I^{n+1}\bigr)^{\sim}$. For existence, multiplication in the Hopf algebra~$\bQ[\pi]$ induces a $\pi$-equivariant map with the desired property.
	
	Taking a colimit first in~$n$, and then in~$m$, we find that there is a unique $\pi$-equivariant map $\Delta\colon\bV\to\bV\otimes\bV_x$ such that
	\begin{equation}\label{eq:counit_and_comultiplication}\tag{$\ast\ast$}
	\begin{tikzcd}
		\bV_x \arrow[r,"\Delta_x"]\arrow[rd,equal] & \bV_x\otimes\bV_x \arrow[d,"\varepsilon\otimes1_{\bV_x}"] \\
		& \bV_x
	\end{tikzcd}
	\end{equation}
	commutes. The fibre~$\Delta_x$ induces a coalgebra structure on~$\bV_x$ with counit~$\varepsilon$ because it is induced from the multiplication on the Hopf algebra $\bQ[\pi]$. This implies that~$\Delta$ is the unique $\pi$-equivariant map inducing such a coalgebra structure on~$\bV_x$, because it is even the unique $\pi$-equivariant map making~\eqref{eq:counit_and_comultiplication} commute.
	
	Finally, the fact that~$\Delta$ is a morphism of algebras and that it makes~$\bV$ into a right $\bV_x$-comodule follow from the dual properties of the Hopf algebra~$\bQ[\pi]$. The fact that both~$\mu$ and~$\Delta$ are morphisms of admissible variations of ind-mixed Hodge structures is a consequence of the fact that the Hain--Zucker ind-mixed Hodge structure on~$\bV_x$ is compatible with the Hopf algebra operations and \cite[Theorem~1.6]{hain-zucker:unipotent_variations}.
\end{proof}

Our strategy for proving Lemma~\ref{lem:realisation_vs_tannakian_Hodge}, then, is to show that the affine rings of the torsors ${}_yP_x^\infty$ we have constructed vary in an admissible ind-variation of mixed Hodge structure on~$X$, and that the underlying ind-local system on~$X^\an$ is the universal one~$\bV$ above. This we will do in the coming section.

\section{Variation of base point in the Hodge realisation}\label{s:VMHS}

\spar{HB21}
	We continue working over $\CC$. For $g\colon Y \to \CC$ a finite type $\CC$-scheme, we let $\MHM(Y)$ denote the category of mixed Hodge modules on $Y_\m{red}$, for which our main references are Saito \cite{saito1990mixed} and Peters--Steenbrink \cite{PetersSteenbrink}. We abbreviate $\MHM(\Spec(\CC))$ to $\MHM$ for short; this is the category of ``absolute mixed Hodge modules'', better known as graded polarizable mixed Hodge structures. For our present purposes we have no need of any schemes that are not smooth over $\CC$.
	
\label{HB22}
	Tubach \cite{tubach2023nori} constructs an \oo-categorically enhanced six functor formalism 
	\[
	Y \mapsto \Dd^b(\MHM(Y))
	\]
	on the category of finite type $\CC$-schemes. In particular, $\Dd^b(\MHM(Y))$ is a symmetric monoidal stable \oo-category whose underlying stable \oo-category is the bounded derived \oo-category of $\MHM(Y)$. Pullback functors are monoidal and have right adjoint lax-monoidal pushforward functors.

\spar{HB23}
	Suppose $g\colon Y\to \Spec(\CC)$ is smooth of dimension $d$. Then for $L \in \MHM$, $g^*L$ belongs to the shifted heart $\MHM(Y)[-d]$. We define an object $M \in \MHM(Y)[-d]$ to be \emph{unipotent} if it admits a filtration by subobjects in the abelian category $\MHM(Y)[-d]$ such that $\gr M \simeq g^*L$ for some $L \in \MHM$.\footnote{Such objects are more often called \textit{relatively unipotent}; we choose to drop the adjective ``relatively'', which would otherwise repeat itself incessantly below.} 
	
	If $M \in \MHM(Y)[-d]$ is unipotent, then the underlying perverse analytic sheaf $\rat(M[d])$ of $M[d]$ admits a filtration by perverse subsheaves such that $\gr \rat M[d] \simeq \QQ_X^{\oplus n}[d]$ is a constant sheaf shifted by $d$. It follows that $\rat M[-d]$ is a locally constant sheaf shifted by $-d$. Hence there's a polarizable variation of Hodge structures $V=(\VV, \Ff^\bullet)$ such that 
	\[
	M[-d] = V^\m{Hdg}  
	\]
	is the associated Hodge module as in Definition 14.53 of Peters--Steenbrink \cite{PetersSteenbrink}. Hence, for every $\CC$-point $y\colon \Spec(\CC) \to X$, $y^*M$ is concentrated in degree $0$.
	
	We denote the full subcategory of $\MHM(Y)[-d]$ spanned by unipotent objects by $\MHM_\un(Y)$.

\spar{HB25}
	We define
	\[
	\Dd^{b, \ct, \le 0} \MHM(Y)
	\subset
	\Dd^b \MHM(Y)
	\]
	to be the full \oo-subcategory spanned by those objects $M$ such that for every $\CC$-point $x\colon \Spec(\CC) \to Y$, $x^*M \in \Dd^{b,\le 0} \MHM$, and we define $\Dd^{b, \ct, \ge 0} \MHM(Y)$ similarly. According to Section 4.6, Remark 2, pp.~328-329 of Saito \cite{saito1990mixed}, this defines a second t-structure, the \emph{constructible t-structure}. We denote its heart by
	\[
	\MHM_\ctr(Y).
	\]
	We denote the associated truncation functors by $\tau^{\le 0}_\ctr$, $\tau^{\ge 0}_\ctr$, and $\Hh^0_\ctr = \tau^{\le 0}_\ctr \circ \tau^{\ge 0}_\ctr \simeq \tau^{\ge 0}_\ctr \circ \tau^{\le 0}_\ctr$. Corollary 2.19 of Tubach \cite{tubach2023nori} furnishes an equivalence
	\[
	\Dd^b(\MHM_\ctr(Y)) \simeq \Dd^b(\MHM(Y)).
	\]
\label{zzz}
	Referring back to paragraph \ref{HB23}, we find that for $Y$ smooth,
	\[
	\MHM_\un(Y) \subset \MHM_\ctr(Y) \,.
	\]

\spar{HB27}
	For $Y$ a finite type $\CC$-scheme, the ind-completion 
	\[
	\Ind \Dd^b \MHM(Y)
	\]
	is (in a natural way) symmetric monoidal \cite[Corollary 4.8.1.14]{LurieAlg} and stable \cite[Proposition 1.1.3.6]{LurieAlg}. The t-structures (perverse and constructible) extend to $\Ind \Dd^b \MHM(Y)$ by defining $\Ind \Dd^b \MHM(Y)^{\ge 0}$ to be the essential image of the fully faithful functor
	\[
	\Ind \Dd^{b, \ge 0} \MHM(Y) 
	\to 
	\Ind \Dd^b \MHM(Y),
	\]
	and similarly for $\Ind \Dd^b \MHM(Y)^{\le 0}$ \cite[Lemma C.2.4.3]{lurie2018spectral}.
	Given $g: Y \to Y'$, there's an induced lax-monoidal pushforward functor
	\[
	Rg_*\colon \Ind \Dd^b\MHM(Y) \to \Ind \Dd^b \MHM(Y').
	\]
	In particular, if $\one_Y \in \MHM(Y) \subset \Ind \Dd^b \MHM(Y)$ denotes the unit object endowed with its essentially canonical algebra structure, then $Rg_* \one_Y$ has an induced structure of commutative algebra in
	\[
	\Ind \Dd^b \MHM(Y').
	\]
	The commutative algebra $Rg_*\one_Y$ is covariantly functorial in $Y \in \Sm_\CC$. There's also a monoidal pullback functor
	\[
	\Ind \Dd^b \MHM(Y) \xfrom{g^*} \Ind \Dd^b \MHM(Y')
	\]
	which is right adjoint to $Rg_*$.

\label{HB28}
When $Y' = \Spec(\CC)$ and $g\colon Y \to \Spec(\CC)$ is the structure morphism of $Y$, there's an equivalence
	\[
	Rg_* \one_Y \simeq C^*_H(Y, \QQ)
	\] 
	in $\CAlg \Ind \Dd^b (\MHM)$, a $\QQ$-form of the object considered in \S\ref{30524A}.

\spar{HB29}
	Let $f\colon X \to \Spec(\CC)$ be a smooth complex algebraic variety and $x \in X(\CC)$ a point. Let
	\[
	\tilde f\colon \tilde X \coloneqq X\times_\CC X \to X
	\]
	denote the second projection. The point $x$ gives rise to a section $\tilde x $ of $\tilde f$. The new structure morphism $\tilde f$ also admits a canonical section $\Delta\colon X \to X \times_\CC X$. Applying the construction of paragraph \ref{HB27} with $g = \tilde f$, we obtain a commutative algebra 
	\[
	R\tilde f_* \one_{\tilde X} \in 
	\CAlg \Ind \Dd^b \MHM(X).
	\]
	We further apply the covariant functoriality of $Rg_*\one_Y$ in $Y$ to obtain augmentations
	\[
	\tilde x^*, \Delta^*\colon R\tilde f_* \one_{\tilde X}
	\rightrightarrows
	\one_X.
	\]
Using these augmentations, we define
	\[
	_\Delta {\tilde B}_\Delta^H = 
	\one_X 
	\underset{\Delta^*, R\tilde f_* \one_{\tilde X}, \Delta^*}
	\otimes
	\one_X, 
	\quad
	_\Delta {\tilde B}_{\tilde x}^H = 
	\one_X 
	\underset{\Delta^*, R\tilde f_* \one_{\tilde X}, {\tilde x}*}
	\otimes
	\one_X.
	\]
	\footnote{The Koszul duality for modules over associative algebras \cite{DCHorev, brantner2021pd} endows this pair of commutative algebras with the additional structure of a left comodule up to coherent higher homotopy. This extra structure, however, will not intervene in the sequel.}

\spar{HB30}
	If $y\colon\Spec(\CC) \to X$ is a point, then pullback along $y$ is monoidal and preserves colimits. We also have equivalences of commutative algebras $y^*\one_X \simeq \one_{\Spec(\CC)}$ and $y^* R\tilde f_* \one_{\tilde X} \simeq Rf_*\one_X$. Together, these give rise to an equivalence
	\[
	y^*\left(
	{_\Delta {\tilde B}_{\tilde x}^H}
	\right)
	\simeq 
	{_y B^H_x}
	= \one_{\Spec(\CC)} \otimes_{y^*,Rf_*\one_X,x^*} \one_{\Spec(\CC)}
	\]
	in $\CAlg \Ind \MHM$. The latter is known to be concentrated in non-negative cohomological degrees. Indeed, this may be verified after passing to de Rham realisation, where it follows, for instance, from Olsson \cite[Corollary 7.6]{OlssonBar}.
	
	The pullback along $y$ is also exact for the constructible t-structure, so
	\[
	y^* \Hh^i_\ctr
	\left(
	{_\Delta {\tilde B}_{\tilde x}^H}
	\right)
	\simeq
	\Hh^i
	\left( {_y B^H_x} \right)
	\]
	for every $y \in X$ and the latter vanishes for $i <0$. It follows that 
	\[
	\Hh^i_\ctr
	\left(
	{_\Delta {\tilde B}_{\tilde x}^H}
	\right)
	= 0
	\]
	for $i <0$. Thus, ${_\Delta {\tilde B}_{\tilde x}^H}$ is concentrated in non-negative cohomological degrees for the constructible t-structure.

\label{HB31}
	Applying truncation into the constructible heart, we obtain an object
	\[
	\Hh_\ctr^0(_\Delta {\tilde B}_{\tilde x}^H)
	=
	\tau_\ctr^{\le 0}(_\Delta {\tilde B}_{\tilde x}^H)
	\in
	\CAlg \Ind \MHM_\ctr(X)
	\]
	such that for $y \in X$, 
	\[
	y^*\Hh_\ctr^0(_\Delta {\tilde B}_{\tilde x}^H)
	\simeq
	\Hh^0({_y B_x^H}).
	\]

\begin{proposition}
	\label{HB32}
	The Ind-constructible mixed Hodge module $\Hh_\ctr^0(_\Delta {\tilde B}_{\tilde x}^H)$ is Ind-unipotent. 
\end{proposition}

	The proof of Proposition \ref{HB32} will require some preparation. For the remainder of this section we use the abbreviations
	\[
	\Dd \coloneqq \Ind \Dd^b \MHM(X)
	\qandq
	\tilde B := {_\Delta {\tilde B}_{\tilde x}^H}.
	\]

\spar{HB34}
	Let $f^*(-) = f^* \Dd^b \MHM \subset \Dd$ denote the full subcategory spanned by objects of the form $f^*L$ with $L \in \Dd^b \MHM$. We define the category of \emph{cellular mixed Hodge modules}
	\[
	\Dd_\m{cell} \subset \Dd
	\]
	as the closure of $f^*(-)$ under colimits. In view of the equivalence $R\tilde f_* \one_{\tilde X} = f^* Rf_* \one_X$, the colimit 
	\[
	{\tilde B}= 
	\one_X 
	\underset{\Delta^*, R\tilde f_* \one_{\tilde X}, {\tilde x}*}
	\otimes
	\one_X
	= 
	\colim_{[n] \in \Delta^\m{op}}
	(R\tilde f_* \one_{\tilde X})^{\otimes n}
	\]
	in $\Dd$ is cellular.

\spar{HB35}
	Let 
	$
	\Dd^\m{c}_\m{cell} \subset \Dd_\m{cell}
	$
	denote the full subcategory spanned by compact objects and let $\Dd^\m{f}_\m{cell}$ denote the closure of $f^*(-)$ in $\Ind \Dd^b \MHM(X)$ under finite colimits.

\begin{lemma}
	\label{HB36}
	We have
	$
	\Dd^\m{c}_\m{cell}
	=
	\Dd^\m{f}_\m{cell}.
	$
\end{lemma}

\begin{proof}
	According to Proposition 5.3.5.5 of \cite{LurieTopos}, every object in any \oo-category is compact when regarded as an object of the associated Ind-category. For any finite type $\CC$-scheme $Y$, $\Dd^b \MHM(Y)$ is closed under finite direct sums and cofibers, hence under all finite colimits. A finite colimit of compact objects is compact. Together, these facts imply (in more than one way) that $\Dd^\m{f}_\m{cell} \subset \Dd^\m{c}_\m{cell}$. 
	
	We turn to the reverse inclusion. The closure of $\Dd^\mf_\cell$ in $\Dd$ under filtered colimits is equal to $\Dd_\cell$, so an arbitrary $M \in \Dd_\cell$ may be written as a filtered colimit 
	\[
	M = \colim_i M_i
	\]
	of objects $M_i \in \Dd^\mf_\cell$. If, moreover, $M$ is compact, then we have an equivalence of homotopy types
	\[
	\End_\Dd(M) 
	=
	\colim_i \Hom(M, M_i).
	\]
	As a general fact about colimits in the \oo-category of homotopy types, it follows that every $f \in \pi_0\End_\Dd(M)$ is in the image of $\pi_0\Hom(M, M_i)$ for some $i$. In particular, the identity of $M$ factors through some $M_i$ up to homotopy. Since $\Dd^\mf_\cell$ is closed under finite colimits, and hence is idempotent complete, it follows that $M$ is equivalent to an object of $\Dd^\mf_\cell$. 
\end{proof}

\spar{HB37}
	Let $\Dd_\un \subset \Dd$ be the full subcategory spanned by objects $M$ such that $\Hh^i_\ctr(M)$ is unipotent for all $i$ and zero for all but finitely many $i$.

\begin{lemma}
	\label{HB38}
	We have 
	$
	\Dd_\un = \Dd^\mf_\cell.
	$
\end{lemma}

\begin{proof}
	Since $f^*$ is t-exact, $\Dd_\un \supset f^*(-)$. Since the former is closed under finite direct sums and cofibers, hence under all finite colimits, we have $\Dd_\un \supset \Dd^\mf_\cell$. 
	
	For the converse, recall that the Betti realisation of a unipotent mixed Hodge module is a locally constant sheaf, which thus endows unipotent mixed Hodge modules with a notion of \emph{rank}. We define the rank of an object of $\Dd_\un$ to be the sum of the ranks of its constructible cohomology sheaves. Now consider $M \in \Dd_\un$. If $t$ is the greatest integer such that $\Hh^t_\ctr(M) \neq 0$, then there exists a morphism
	\[
	\phi:f^*N \to M
	\]
	for some $N \in \MHM[-t]$ such that the cofiber $L$ of $\phi$ has strictly lower rank than that of $M$. By an induction on the rank with evident base case, we may assume $L$ belongs to $\Dd^\mf_\cell$ and it follows that $M$ belongs to $\Dd^\mf_\cell$. 
\end{proof}

\begin{corollary}
	\label{HB39}
	For any full subcategory $\Cc \subset \Dd$, we let $\Cc^{\ctr, \ge 0}$ denote the full subcategory $\Cc^{\ctr, \ge 0} = \Cc \cap \Dd^{\ctr, \ge 0}$.
	With this notation, we have an equivalence 
	\[
	\Ind \Dd_\un^{\ctr, \ge 0} 
	\simeq
	\Dd_\cell^{\ctr, \ge 0}.
	\]
\end{corollary}

\begin{proof}[Proof of Proposition \ref{HB32}]
	By Corollary \ref{HB39}, $\tilde B = \colim \tilde B_i$ is a filtered colimit with $\tilde B_i \in \Dd_\un^{\ctr, \ge 0} $. Since $\tau_\ctr^{\le 0}$ is a left adjoint, hence preserves colimits, we have
	\[
	\Hh^0_\ctr(\tilde B)
	\simeq
	\tau_\ctr^{\le 0} (\tilde B)
	\simeq
	\colim_i \tau_\ctr^{\le 0} (\tilde B_i)
	\simeq
	\colim_i \Hh^0_\ctr(\tilde B_i).
	\qedhere
	\]
\end{proof}

\spar{HC22}
For any finite type $\CC$-scheme $Y$, let $Y^\an$ denote the associated complex analytic space and let $\Vect^\m{c} \QQ_{Y^\an}$ denote the symmetric monoidal abelian category of sheaves of finite dimensional $\QQ$-\textit{vector spaces}. By its construction, the \oo-categorically enhanced mixed Hodge six functor formalism of Tubach (\ref{HB22}) comes together with functors
\[
\rat_Y: \Ind \Dd^b\big(\MHM(Y) \big)
\to
\Ind \Dd^b \big(\Vect^\m{c} \QQ_{Y^\an} \big)
\]
compatible with the six operations. Moreover, $\rat_Y$ is $t$-exact for the constructible t-structure on $\Ind \Dd^b\big(\MHM(Y) \big)$  and admits a right adjoint. The induced symmetric monoidal functor 
\[
\Ind \MHM_\ctr(Y) 
\to
\Ind \Dd^b \big(\Vect^\m{c} \QQ_{Y^\an} \big)
\]
maps any Ind-\textit{varation of mixed Hodge structure} to its underlying Ind-\textit{local system}. 

\spar{HC23}
Returning to the situation and the notation of paragraph \ref{HB29}, we obtain a commutative algebra
\[
R\tilde f_* \QQ_{\tilde X^\an}
\in
\CAlg \Ind \Dd^b \big( \Vect^\m{c} \QQ_{X^\an} \big)
\]
and augmentations
\[
\tilde x^*, \Delta^*\colon R\tilde f_* \QQ_{\tilde X^\an}
\rightrightarrows
\one_X.
\]
We define
\[
_\Delta {\tilde B}_{\tilde x} = 
{_\Delta} {\tilde B}_{\tilde x}^{\an} = 
\QQ_{X^\an} 
\underset{\Delta^*, R\tilde f_* \QQ_{\tilde X^\an}, {\tilde x}^*}
\otimes
\QQ_{X^\an} \,.
\]
As in paragraph \ref{HB30},  $_\Delta {\tilde B}_{\tilde x}^{\an}$ is concentrated in non-negative cohomological degrees. Applying truncation into the heart, we obtain an object 
\[
\Hh^0(_\Delta {\tilde B}_{\tilde x}^\an)
=
\tau^{\le 0}(_\Delta {\tilde B}_{\tilde x}^\an)
\in
\CAlg \Ind \Vect^c \QQ_{X^\an}.
\]

\begin{proposition}
	\label{HC24}
	In the situation and the notation of paragraph \ref{HC23}, there's an isomorphism of commutative algebras in Ind-\textit{local systems}
	\[
	\rat_X \big(
	\Hh_\ctr^0(_\Delta {\tilde B}_{\tilde x}^H)
	\big)
	\simeq
	\Hh^0(_\Delta {\tilde B}_{\tilde x}^\an).
	\]
	Hence, $\Hh^0(_\Delta {\tilde B}_{\tilde x}^\an)$ is the underlying Ind-\emph{local system} of the Ind-variation $\Hh_\ctr^0(_\Delta {\tilde B}_{\tilde x}^H)$.
\end{proposition}

\begin{proof}
	Direct consequence of the compatibilities noted in paragraph \ref{HC22}.
\end{proof}

\begin{proposition}
	\label{HC25}
	Under Tannaka duality, the Ind-\textit{local system}  $\Hh^0(_\Delta {\tilde B}_{\tilde x}^\an)$ corresponds to the left regular representation of the Tannakian fundamental group $\pi_1^\un(X, x)_B$ of the Tannakian category of unipotent local systems. 
\end{proposition}

We will prove Proposition~\ref{HC25} by considering mapping class group actions. Fix a point~$y\in X^\an\smallsetminus\{x\}$, and let $\Homeo^+(X^\an;x,y)$ denote the group of orientation-preserving self-homeomorphisms of~$X^\an$ which fix both~$x$ and~$y$. Every element $\phi\in\Homeo^+(X^\an;x,y)$ induces an automorphism of~$Rf_*\bQ_{X^\an}$ preserving the two augmentations~$x^*$ and~$y^*$ up to homotopy, and thus induces an automorphism of the algebra $H^0({}_yB_x)$, where ${}_yB_x$ is the pushout of~$Rf_*\bQ_{X^\an}$ along~$x^*$ and~$y^*$ in $\CAlg(\Dd(\bQ))$. This defines a right action of~$\Homeo^+(X^\an;x,y)$ on~$H^0({}_yB_x)$, which we turn into a left action in the usual way.

\begin{lemma}\label{lem:MCG_equivariant_iso}
	There is a $\Homeo^+(X^\an;x,y)$-equivariant isomorphism
	\[
	\tag{*}
	H^0({}_yB_x)\cong \cO(\pi_1^\un(X^\an;x,y))
	\]
	of~$\bQ$-algebras.
\end{lemma}
\begin{proof}
	The isomorphism (*) is a special case of Theorem 6.4.5 of \cite{dancohen2021rational}. The equivariance is a special case of the fact that this isomorphism is natural in $(X,x,y)$. Obtaining this naturality is a matter of adding various compatibilities to the constructions used in loc. cit. Despite the \oo-categorical setting, these compatibilities require only tried and true techniques. We omit the details.
\end{proof}

The \textit{mapping class group} $\MCG(X^\an;x,y)$ is the quotient of the group $\Homeo^+(X^\an;x,y)$ by the subgroup of self-homeomorphisms which are isotopic to the identity via an isotopy fixing both~$x$ and~$y$. This group sits in the Birman exact sequence (see \cite[\S4.2.1]{farb-margalit:primer})
\[
1 \to \pi_1(X^\an\smallsetminus\{x\};y) \to \MCG(X^\an;x,y) \to \MCG(X^\an;x) \to 1 \,.
\]

\begin{lemma}\label{lem:MCG_action_is_monodromy_action}
	The action of $\Homeo^+(X^\an;x,y)$ on~$H^0({}_yB_x)$ factors through an action of the mapping class group $\MCG(X^\an;x,y)$. Moreover, the restriction of this action to~$\pi_1(X^\an\smallsetminus\{x\};y)$ agrees with the monodromy action coming from the local system $\Hh^0({}_\Delta\tilde B_{\tilde x})$.
\end{lemma}
\begin{proof}
	The homomorphism $\pi_1(X^\an\smallsetminus\{x\};y)\to\MCG(X^\an;x,y)$ in the Birman exact sequence is the point-pushing map, defined as follows. If~$\gamma\colon[0,1]\to X^\an\smallsetminus\{x\}$ is a continuous loop based at~$y$, then we can extend~$\gamma$ to an isotopy
	\[
	\Phi\colon X^\an\times[0,1] \to X^\an
	\]
	fixing~$x\in X^\an$ and restricting to the identity on~$X^\an\times\{0\}$. (Explicitly: $\Phi(y,t)=\gamma(t)$ for all~$t$, $\Phi(x,t)=x$ for all~$t$, and~$\Phi(z,0)=z$ for all~$z\in X^\an$.) This implies that~$\phi=\Phi|_{X^\an\times\{1\}}$ is an orientation-preserving self-homeomorphism of~$X^\an$ fixing both~$x$ and~$y$. The point-pushing map in the Birman exact sequence takes the homotopy class of~$\gamma$ to the isotopy class of~$\phi$\footnote{Because our convention is that the group law on $\pi_1(X^\an\smallsetminus\{x\};y)$ is the functorial one, opposite to the usual convention in topology, one would expect us to need to take the inverse of the point-pushing map described in \cite{farb-margalit:primer}. However, it seems that the point-pushing map as described in \cite{farb-margalit:primer} is actually an anti-homomorphism rather than a homomorphism, so this is not an issue.}.
	
	Let~$\gamma$, $\Phi$ and~$\phi$ be as above. We are going to prove that the monodromy action of~$\gamma$ agrees with the action of~$\phi$. For this, let $g\colon X^\an\times[0,1]\to[0,1]$ be the projection on the second factor, and let~$\hat x$, $\hat y$ and $\hat\gamma$ be the sections of~$g$ corresponding to the points~$x,y\in X^\an$ and the map~$\gamma$. Let
	\[
	Rg_*\bQ_{X^\an\times[0,1]}\in\CAlg\Ind\Vect\bQ_{[0,1]}
	\]
	denote the higher direct image of the constant local system on~$[0,1]$, and let~$\hat x^*$, $\hat y^*$ and $\hat\gamma^*$ denote the augmentations coming from~$\hat x$, $\hat y$ and~$\hat\gamma$. We form the pushouts
	\begin{align*}
		{}_{\hat y}\hat B_{\hat x} &= \QQ_{[0,1]}	\underset{\hat y^*, Rg_* \QQ_{X^\an\times[0,1]}, {\hat x}*} \otimes \QQ_{[0,1]} \,, \\
		{}_{\hat \gamma}\hat B_{\hat x} &= \QQ_{[0,1]}	\underset{\hat \gamma^*, Rg_* \QQ_{X^\an\times[0,1]}, {\hat x}*} \otimes \QQ_{[0,1]} \,.
	\end{align*}
	
	The map $\hat\Phi\colon X^\an\times[0,1]\to X^\an\times[0,1]$ given by~$\hat\Phi(z,t)=(\Phi(z,t),t)$ is an automorphism of topological spaces over~$[0,1]$ fixing the section~$\hat x$ and taking the section~$\hat y$ to~$\hat\gamma$. This automorphism induces an automorphism $\hat\Phi^*$ of~$Rg_*\bQ_{X^\an\times[0,1]}$ preserving the augmentation~$\hat x^*$ and taking~$\hat\gamma^*$ to~$\hat y^*$. In particular, it induces an isomorphism
	\[
	\hat\Phi^*\colon \Hh^0({}_{\hat \gamma}\hat B_{\hat x}) \xrightarrow\sim \Hh^0({}_{\hat y}\hat B_{\hat x})
	\]
	of commutative algebras in ind-local systems on~$[0,1]$.
	
	Now the second projection $g\colon X^\an\times[0,1]\to[0,1]$ is the pullback of~$\tilde f\colon\tilde X^\an\to X^\an$ along the map $\gamma\colon[0,1]\to X^\an$, with the sections $\tilde x$ and~$\Delta$ of~$\tilde f$ pulling back to the sections~$\hat x$ and~$\hat\gamma$, respectively. Because the base change map $\gamma^*R\tilde f_*\bQ_{\tilde X^\an}\to Rg_*\bQ_{X^\an\times[0,1]}$ is an isomorphism, it follows that $\Hh^0({}_{\hat\gamma}\hat B_{\hat x})$ is isomorphic to~$\gamma^*\Hh^0({}_\Delta\tilde B_{\tilde x})$. Applying a similar argument to the constant map $p\colon[0,1]\to\{*\}$, we find that~$\Hh^0({}_{\hat y}\hat B_{\hat x})$ is isomorphic to~$p^*H^0({}_yB_x)$. So we can view~$\hat\Phi^*$ as an isomorphism
	\begin{equation}\label{eq:point-pushing}
		\hat\Phi^*\colon \gamma^*\Hh^0({}_\Delta\tilde B_{\tilde x}) \xrightarrow\sim p^*H^0({}_yB_x) \,.
	\end{equation}
	
	On the one hand, the fibre of~$\hat\Phi$ at~$1\in[0,1]$ is equal to the self-homeomorphism~$\phi$, and so the fibre of~\eqref{eq:point-pushing} at~$1$ is equal to the automorphism $\phi^*\colon H^0({}_yB_x)\xrightarrow\sim H^0({}_yB_x)$. On the other hand, \eqref{eq:point-pushing} is a trivialisation of the pulled-back local system~$\gamma^*\Hh^0({}_\Delta\tilde B_{\tilde x})$ on~$[0,1]$, and so the monodromy action of~$\gamma$ is equal to the composition
	\[
	\Hh^0({}_\Delta\tilde B_{\tilde x})_y = \bigl(\gamma^*\Hh^0({}_\Delta\tilde B_{\tilde x})\bigr)_0 \xrightarrow[\sim]{\hat\Phi^*_0} H^0({}_yB_x) \xrightarrow[\sim]{(\hat\Phi^*_1)^{-1}} \bigl(\gamma^*\Hh^0({}_\Delta\tilde B_{\tilde x})\bigr)_1 = \Hh^0({}_\Delta\tilde B_{\tilde x})_y \,.
	\]
	We've seen that the fibre~$\hat\Phi^*_1$ at~$1$ is equal to~$\phi^*$, while the fibre~$\hat\Phi^*_0$ at~$0$ is equal to the identity (because the fibre of~$\hat\Phi$ at~$0$ is the identity). Putting this all together, we find that the monodromy action of~$\gamma$ on~$H^0({}_yB_x)$ is equal to $(\phi^*)^{-1}$, i.e.~that it agrees with the left action of~$\phi$.
	
	To conclude the proof, specialising to the case when~$\gamma$ is the constant path with value~$y$ shows that the action of any~$\phi\in\Homeo^+(X^\an;x,y)$ isotopic to the identity is trivial, and hence the action factors through the mapping class group. The remaining claim in the lemma follows by considering general~$\gamma$.
\end{proof}

\begin{lemma}\label{lem:MCG_action_is_left_regular}
	The action of $\Homeo^+(X^\an;x,y)$ on~$\cO(\pi_1^\un(X^\an;x,y))$ factors through an action of the mapping class group~$\MCG(X^\an;x,y)$. Moreover, the restriction of this action to $\pi_1(X^\an\smallsetminus\{x\};y)$ agrees with the left regular action.
\end{lemma}
\begin{proof}
	It suffices to prove the same for the action of $\Homeo^+(X^\an;x,y)$ on the set $\pi_1(X^\an;x,y)$. Let~$\gamma\colon[0,1]\to X^\an\smallsetminus\{x\}$ be a loop based at~$y$, extend~$\gamma$ to an isotopy $\Phi\colon X^\an\times[0,1]\to X^\an$ fixing~$x$ such that~$\Phi|_{X^\an\times\{0\}}$ is the identity, and let~$\phi=\Phi|_{X^\an\times\{1\}}$. As in the previous lemma, it suffices to prove that the action of~$\phi\in\Homeo^+(X^\an;x,y)$ on~$\pi_1(X^\an;x,y)$ agrees with left-multiplication by~$\gamma$.
	
	For this, let~$\gamma'\colon[0,1]\to X^\an$ be a path from~$x$ to~$y$. We will write down an explicit homotopy from~$\phi(\gamma')$ to the composition $\gamma\cdot\gamma'$. Let $H\colon[0,1]^2\to X^\an$ be the continuous function given by
	\[
	H(t,s) = 
	\begin{cases}
		\Phi(\gamma'(2t),1-s) & \text{if~$t\leq\frac12$,} \\
		\gamma(1+2s(t-1)) & \text{if~$t\geq\frac12$.}
	\end{cases}
	\]
	This function is well-defined because~$\Phi(\gamma'(1),1-s)=\Phi(y,1-s)=\gamma(1-s)$. By construction, the fibre of~$H$ at~$s=0$ is the composition of~$\phi(\gamma')$ with the constant path at~$y$, while the fibre of~$H$ at~$s=1$ is the composition of~$\gamma'$ with~$\gamma$. Since~$H$ is a homotopy relative to the endpoints~$t=0,1$, it follows that $\phi(\gamma')=\gamma\cdot\gamma'$ as elements of~$\pi_1(X^\an;x,y)$, which is what we wanted to prove.
\end{proof}

\spar{IA33}
We're ready to prove Proposition \ref{lem:realisation_vs_tannakian_Hodge}.
Combining Lemmas~\ref{lem:MCG_equivariant_iso}, \ref{lem:MCG_action_is_monodromy_action} and~\ref{lem:MCG_action_is_left_regular}, we find that there exists an isomorphism
\begin{equation}\label{eq:iso_of_local_systems_1}
\Hh^0({}_\Delta\tilde B_{\tilde x})|_{X^\an\smallsetminus\{x\}} \cong \bV|_{X^\an\smallsetminus\{x\}}
\end{equation}
of ind-local systems on~$X^\an\smallsetminus\{x\}$, where~$\bV$ is the local system on~$X^\an$ whose fibre at~$x$ is isomorphic to~$\cO(\pi_1^\un(X^\an;x,y))$ with its left-regular action. (The fibre of~$\bV$ at~$y\neq x$ is isomorphic to $\cO(\pi_1^\un(X^\an;x))$ with its left-regular action of~$\pi_1(X^\an;y)$.) Because both sides of~\eqref{eq:iso_of_local_systems_1} are restrictions of local systems on~$X^\an$, it follows that~\eqref{eq:iso_of_local_systems_1} extends uniquely to an isomorphism of ind-local systems
\begin{equation}\label{eq:iso_of_local_systems_2}
	\Hh^0({}_\Delta\tilde B_{\tilde x}) \cong \bV
\end{equation}
on~$X^\an$.

The isomorphism~\eqref{eq:iso_of_local_systems_2} is an isomorphism of algebras by Lemma~\ref{lem:MCG_equivariant_iso}, and hence so too is its fibre
\begin{equation}\label{eq:iso_of_local_systems_3}
	H^0({}_x\tilde B_x) \cong {}_xH_x = \cO(\pi_1^\un(X^\an;x))
\end{equation}
at~$x$. In particular, the counit $\varepsilon\colon H^0({}_xB_x)\to\bQ$ corresponds to a $\bQ$-algebra homomorphism ${}_xH_x\to\bQ$, i.e.~to a group element
\[
\gamma\in\pi_1^\un(X^\an;x)(\bQ) \,.
\]
Now the right-regular action of~$\pi_1^\un(X^\an;x)$ on~$\cO(\pi_1^\un(X^\an;x))$ commutes with the left-regular action, and so comes from a right action on the local system~$\bV$. Composing the isomorphism~\eqref{eq:iso_of_local_systems_2} with the right action of~$\gamma^{-1}$, we may assume without loss of generality that~$\gamma=1$, i.e.~that~\eqref{eq:iso_of_local_systems_3} is an isomorphism of augmented algebras. Because the counit $H^0({}_xB_x)\to\bQ(0)$ is a morphism of mixed Hodge structures, it follows from Lemma~\ref{lem:unique_MHS} that the isomorphism~\eqref{eq:iso_of_local_systems_2} is in fact an isomorphism of admissible variations of ind-mixed Hodge structures, where the variation of ind-mixed Hodge structure on~$\bV$ is the one constructed by Hain and Zucker. It also follows from Lemma~\ref{lem:unique_comodule_algebra} that the isomorphism~\eqref{eq:iso_of_local_systems_2} is an isomorphism of comodules, and its fibre at~$x$ is an isomorphism of Hopf algebras. Taking the fibre of~\eqref{eq:iso_of_local_systems_2} at a point~$y$ and taking $\Spec$ of the result then provides an isomorphism
\[
{}_yP_x^\infty \cong \pi_1^\un(X^\an;x,y)
\]
of~$\bQ$-schemes, compatible with both the ind-mixed Hodge structures on affine rings and the right action of ${}_xP_x^\infty\cong\pi_1^\un(X^\an;x)$ on either side. So we have proven Lemma~\ref{lem:realisation_vs_tannakian_Hodge}. \qed

\begin{remark}
	In the above argument, it is presumably the case that~$\gamma=1$ straight away, without having to modify the isomorphism~\eqref{eq:iso_of_local_systems_2}. However, it seemed easier to write the argument in a way which does not rely on this fact.
\end{remark}

\section{Realisationwise locally geometric augmentations}\label{s:realisationwise_augs}

With the preceding material set up, we are now in a position to state our slightly more general version of the Weil height machine. This is based on a weaker notion of local geometricity. Let~$k$ be a number field\footnote{Our results here should also be true for global function fields. One result we will cite is only stated for number fields, so we restrict to number fields for safety.}, $X\in\Sm_k$ a smooth geometrically connected variety, and~$x\in X(k)$ a $k$-rational base point. For every place~$v$ of~$k$ (finite or infinite), we define a \emph{localisation map} on the motivic augmentation set~$\Aug(X/k)$ as follows.

If~$v$ is a finite place, then we make an auxiliary choice of prime number~$l$, which may or may not be the residue characteristic of~$v$. Let~$\loc_v^l$ denote the composition
\[
\Aug(X/k) \to \Aug(X_{k_v}/k_v) \to \Aug_l(X_{k_v}/k_v) \to H^1(G_v,U^l(\bQ_l)) \,,
\]
where the first map is base-change, the second is $l$-adic \'etale realisation, and the third map is the construction of a $U^l$-torsor associated to any augmentation of~$C^*(X_{k_v},\bQ_l)$ in the previous section.

If instead~$v$ is an infinite place, then we let~$\loc_v^\infty$ denote the composition
\[
\Aug(X/k) \to \Aug(X_{k_v}/k_v) \to \Aug_\infty(X_{k_v}/k_v) \to H^1(G_{\MHS,\bR},U^\infty) \,,
\]
where the first map is base-change, the second is $\bR$-linear Hodge realisation, and the third map is the construction of a $U^\infty$-torsor associated to any augmentation of~$C^*(X_{k_v},\bR)$.

\begin{definition}
	Let~$X$ be a smooth variety over a number field~$k$, and let~$b\in X(k)$ be a $k$-rational base point. We say that an augmentation~$\alpha\in\Aug(X/k)$ is \emph{realisationwise locally geometric} if:
	\begin{itemize}
		\item for all finite places~$v$ of~$k$ and all primes~$l$, $\loc_v^l(\alpha)$ lies in the image of~$j_l\colon X(k_v)\to H^1(G_v,U^l(\bQ_l))$; and
		\item for all infinite places~$v$ of~$k$, $\loc_v^\infty(\alpha)$ lies in the image of~$j_\infty\colon X(k_v)\to H^1(G_{\MHS,\bR},U^\infty)$.
	\end{itemize}
\end{definition}

\begin{remark}
	It is clear from the definition that every locally geometric augmentation is realisationwise locally geometric.
\end{remark}

The aim of this section is to show that realisationwise locally geometric augmentations for curves also admit a well-behaved Weil height theory. For the remainder of this section, we deviate from our earlier notation, and let~$\Aug(X/k)^\lgeom$ denote the set of realisationwise locally geometric augmentations, for~$X$ a smooth variety over a number field~$k$. We also set
\begin{align*}
	\Aug(X_\kbar/\kbar) &\coloneqq \varinjlim_{k'/k}\Aug(X_{k'}/k') \,, \\
	\Aug(X_\kbar/\kbar)^\lgeom &\coloneqq \varinjlim_{k'/k}\Aug(X_{k'}/k')^\lgeom \,,
\end{align*}
the colimits being taken over finite extensions~$k'/k$ inside~$\kbar$. The main result we are going to prove is:

\begin{theorem}[Rational motivic Weil height machine II]\label{thm:weil_heights_ii}
	Let~$k$ be a number field and $X/\kbar$ a smooth projective curve of genus~$\geq1$. Then there is a function
	\[
	h\colon \Pic(X) \to \frac{\{\text{functions $\Aug(X/\kbar)^\lgeom\to\bR$}\}}{\{\text{bounded functions}\}}
	\]
	satisfying the following properties:
	\begin{enumerate}
		\item for line bundles $L_1,L_2\in \Pic(X)$ we have $h_{L_1\otimes L_2} = h_{L_1}+h_{L_2}+O(1)$;
		\item for a morphism $f\colon X\to Y$ of smooth projective curves of genus~$\geq1$ and a line bundle~$L$ on~$Y$ we have $h_{f^*L}(\alpha)=h_L(f(\alpha))+O(1)$; and
		\item for a morphism $f\colon X\to A$ from~$X$ to an abelian variety and a line bundle~$M$ on~$A$ we have $h_{f^*M}(\alpha)=\hat h_M(f(\alpha))+O(1)$.
	\end{enumerate}
	Here, $O(1)$ is shorthand for a bounded function on $\Aug(X/\kbar)^\lgeom$.
\end{theorem}

The construction is the same as before. For any line bundle~$L\in\Pic(X)$, there is a morphism $f\colon X\to A$ to an abelian variety and a line bundle~$M$ on~$A$ such that~$f^*M\simeq L^{\otimes m}$ for some~$m\in\bN$. We then define the height associated to~$L$ by
\[
h_L(\alpha) \coloneqq \frac1m\hat h_M(f(\alpha)) + O(1) \,.
\]
Again, checking that this function is independent of the tuple $(A,f,M,m)$ up to $O(1)$ amounts to checking the following analogue of Lemma~\ref{lem:trivial_pullback_implies_bounded_height}.

\begin{lemma}\label{lem:trivial_pullback_implies_bounded_height_ii}
	Let~$f\colon X \to A$ be a morphism from a smooth projective curve~$X/\kbar$ of genus~$\geq1$ to an abelian variety~$A/\kbar$. Suppose that~$M$ is a line bundle on~$A$ such that $f^*M\simeq\cO_X$ is the trivial line bundle on~$X$. Then the function
	\[
	\alpha \mapsto \hat h_M(f(\alpha))
	\]
	is bounded on $\Aug(X/\kbar)^\lgeom$.
\end{lemma}

The key ingredient in the proof is a description of local N\'eron--Tate heights as functions on the local cohomology sets $H^1(G_v,U^l(\bQ_l))$ as well as the classifying sets for mixed Hodge structures $H^1(G_{\MHS,\bR},U^\infty)$. These are from the first author's thesis.

\begin{lemma}[{\cite[Theorem~1.6]{alex:motivic_anabelian}}]\label{lem:local_height_non-archimedean}
	Let~$A/k$ be an abelian variety, and $M=(M,\tilde0)$ a rigidified line bundle on~$A$, with associated $\bG_m$-torsor~$M^\times$. Let~$v$ be a finite place of~$k$, let~$l$ be a prime number not divisible by~$v$, and let~$U^l$ be as in the previous section. Then~$U^l$ is a central extension of~$V_lA$ by $\bQ_l(1)$, the inclusion $\bQ_l(1)\hookrightarrow U^l$ induces a bijection
	\[
	H^1(G_v,\bQ_l(1)) \xrightarrow\sim H^1(G_v,U^l(\bQ_l)) \,,
	\]
	and the composition
	\[
	M^\times(k_v) \xrightarrow{j_l} H^1(G_v,U^l(\bQ_l)) \cong H^1(G_v,\bQ_l(1)) = \bQ_l
	\]
	takes values in~$\bQ\subseteq\bQ_l$, and is equal to
	\[
	-\frac{\log||\cdot||_{L,v}}{\log l} \,.
	\]
\end{lemma}

\begin{lemma}[{\cite[Theorem~1.8]{alex:motivic_anabelian}}]\label{lem:local_height_archimedean}
	Let~$A/\bC$ be an abelian variety, and $M=(M,\tilde0)$ a rigidified line bundle on~$A$, with associated $\bG_m$-torsor~$M^\times$. Let~$U^\infty$ be as in the previous section. Then~$U^\infty$ is a central extension of~$H_1(A,\bR)$ by $\bR(1)$, the inclusion $\bR(1)\hookrightarrow U^\infty$ induces a bijection
	\[
	H^1(G_{\MHS,\bR},\bR(1)) \xrightarrow\sim H^1(G_{\MHS,\bR},U^\infty) \,,
	\]
	and the composition
	\[
	M^\times(\bC) \xrightarrow{j_\infty} H^1(G_{\MHS,\bR},U^\infty) \cong H^1(G_{\MHS,\bR},\bR(1)) = \bR
	\]
	is equal to $-\log||\cdot||_{L,\infty}$.
\end{lemma}

\begin{proof}[Proof of Lemma~\ref{lem:trivial_pullback_implies_bounded_height_ii}]
	We are free to assume that~$X$, $x$, $A$, $M$ and~$f$ are defined over a number field~$k$; we switch to using the same letters for their $k$-forms.
	
	Let~$\tilde f\colon X\to M^\times$ be a lift of~$f$ (which exists by assumption that~$f^*M$ is trivial), and let~$\tilde 0\coloneqq \tilde f(x)$ be the image of the base point~$x\in X(k)$. Because the zero element of~$A$ plays no role in the problem, we are free to assume that~$\tilde0$ lies in the fibre over~$0\in A(k)$. Let~$U^l_M$ and~$U^\infty_M$ denote the pro-unipotent groups attached to~$M^\times$.
	
	For each finite place~$v$ of~$\kbar$, let~$l$ be a prime number not divisible by~$v$, and let~$\lambda_{M,v,l}$ denote the composition
	\[
	\Aug(M^\times_\kbar/\kbar) \to \varinjlim_{k'/k}H^1(G_{v'},U^l_M(\bQ_l)) \cong \bQ_l \,,
	\]
	where the colimit runs over all finite extensions~$k'/k$ inside~$\kbar$ and~$v'$ denotes the restriction of~$v$ to~$k'$. The bijection on the right is afforded by Lemma~\ref{lem:local_height_non-archimedean}. The function~$\lambda_{M,v,l}$ thus has the property that its composition with the map $\kappa\colon M^\times(\kbar)\to\Aug(M^\times_\kbar/\kbar)$ is equal to the local N\'eron--Tate height~$\lambda_{M,v}$. For each infinite place~$v$, let $\lambda_{M,v,\infty}$ denote the similar composition
	\[
	\Aug(M^\times_\kbar/\kbar) \to H^1(G_{\MHS,\bR},U^\infty_M) \cong \bR \,.
	\]
	
	Now let~$\alpha\in\Aug(X_\kbar/\kbar)^\lgeom$ be a realisationwise locally geometric augmentation, and consider its image $\tilde f(\alpha)\in\Aug(M^\times_\kbar/\kbar)^\lgeom$. By realisationwise local geometricity, for every non-archimedean place~$v$, $\lambda_{M,v,l}(\tilde f(\alpha))$ lies in the image of the composition
	\[
	X(\kbar_v) \xrightarrow{\tilde f} M^\times(\kbar_v) \xrightarrow{j_v} \varinjlim_{k'/k}H^1(G_{v'},U^l_M(\bQ_l)) \cong \bQ_l \,,
	\]
	i.e.~in the image of the composition
	\[
	X(\kbar_v) \xrightarrow{\tilde f} M^\times(\kbar_v) \xrightarrow{\lambda_{M,v}} \bQ_l \,.
	\]
	In particular, $\lambda_{M,v,l}(\tilde f(\alpha))\in\bQ$, and it lies inside the bounded set $\lambda_{M,v}(\tilde f(X(\kbar_v)))\subseteq\bQ$. So there is a bound~$B_v$, independent of~$\alpha$, such that
	\[
	|\lambda_{M,v,l}(\tilde f(\alpha))|\leq B_v
	\]
	for every~$\alpha\in\Aug(X_\kbar/\kbar)^\lgeom$ (where~$|\cdot|$ is the archimedean norm on~$\bQ$). As before, the bounds~$B_v$ depend only on the restriction of~$v$ to~$k$.
	
	At archimedean places~$v$, something similar happens. The element $\lambda_{M,v,\infty}(\tilde f(\alpha))$ lies in the bounded subset $\lambda_{M,v}(\tilde f(X(\kbar_v)))\subseteq\bR$, so there is again some~$B_v$, depending only on the place of~$k$ below~$v$, such that $|\lambda_{M,v,\infty}(\tilde f(\alpha))|\leq B_v$ for all~$\alpha\in\Aug(X_\kbar/\kbar)^\lgeom$. From here on, the proof proceeds exactly as in the proof of Lemma~\ref{lem:trivial_pullback_implies_bounded_height}.
\end{proof}

The other properties of the motivic Weil height hold verbatim for realisationwise locally geometric augmentations. The analogue of Lemma~\ref{lem:degree_bound} is proved by exactly the same argument. The only modification required to the proof of Theorem~\ref{thm:northcott_abelian} (Northcott property) comes in the proof of Lemma~\ref{lem:contained_in_lattice}, where one needs to replace the local augmentation sets $\Aug(A_{k'_v}/k'_v)$ for~$v$ finite by the Galois cohomology sets $H^1(G_{v'},U^p(\bQ_p))$, where~$p$ is the rational prime below~$v$, and use the fact that the kernel of the Kummer map $A(k'_v) \to H^1(G_{v'},U^p(\bQ_p))$ is $A(k'_v)_\tors$. The proof of the Manin--Dem'janenko Theorem~\ref{thm:manin-demjanenko} requires no modifications.

\bibliography{AH_refs,other_refs}
\bibliographystyle{plain}

\vfill

\Small\textsc{
AB: Department of Mathematics, Cornell University.
}

\texttt{alex.betts@cornell.edu}

\medskip

\Small\textsc{
ID: Department of Mathematics, Ben-Gurion University of the Negev.
}

\texttt{ishaida@bgu.ac.il}

\end{document}